\documentclass[12pt]{article}
\usepackage[utf8]{inputenc}
\usepackage{amssymb,graphicx,mathtools,amsfonts}
\usepackage{amsmath}
\usepackage{amsthm}
\usepackage{hyperref}
\usepackage[dvipsnames]{xcolor}
\usepackage{dsfont}
\usepackage{comment}
\usepackage{xcolor}
\usepackage{tikz-cd}

\usepackage{enumitem}

\usepackage{a4wide}
\usepackage{esint}
\catcode`\@=11
\def\theequation{\@arabic{\c@section}.\@arabic{\c@equation}}
\catcode`\@=12
\usepackage{enumitem}
\graphicspath{{images/}}
\newtheorem{theorem}{Theorem}[section] 
\newtheorem{lemma}{Lemma}[section] 
\newtheorem{definition}{Definition}[section] 
\newtheorem{proposition}{Proposition}[section]

\newtheorem{remark}{Remark}[section]

\allowdisplaybreaks

\newcommand{\e}{\varepsilon}

\newcommand{\om} {\Omega}

\newcommand{\real}{\mathbb{R}}

\newcommand{\rnn}{\mathbb{R}^{N}}
\newcommand{\rtwon}{\mathbb{R}^{2N}}
\newcommand{\lv}{\lVert}
\newcommand{\rv}{\rVert}

\newcommand{\grad}{\nabla}
\newcommand{\ntrl}{\mathbb{N}}

\newcommand{\X}{\mathcal{X}_0} 
\newcommand{\J}{\mathcal{J}} 

\newcommand{\Lo}{\mathcal{L}_{p,q}}
\usepackage{tikz}
\usetikzlibrary{positioning}

\DeclareMathOperator*{\essinf}{ess\,inf}

\providecommand{\subjclass}[2][2020]{%
  \par\medskip\noindent\textbf{#1 MSC.}\ #2\par}
\providecommand{\keywords}[1]{%
  \par\smallskip\noindent\textbf{Keywords.}\ #1\par}

\title{
The isocritical regime for mixed local-nonlocal $(p,q)$
Laplacian: existence of ground state, and decay estimates}

\author{Diksha Gupta \footnote{Department of Mathematical Sciences, Korea Advanced Institute of Science and Technology, Daejeon 34141, Republic of Korea, dikshagupta1232@gmail.com}, Shammi Malhotra\footnote{Department of Mathematics, Indian Institute of Technology Delhi, Hauz Khas New Delhi 110016,  India, shammi22malhotra@gmail.com}, and K. Sreenadh\footnote{Department of Mathematics, Indian Institute of Technology Delhi, Hauz Khas New Delhi 110016,  India, sreenadh@maths.iitd.ac.in}}

\date{}

\begin{document}
\maketitle

\begin{abstract}
We study the mixed local-nonlocal operator $\mathcal{L}_{p,q} := -\Delta_p + (-\Delta)_q^s$
in the \emph{isocritical regime} $p^* = q_s^*$, i.e. $1 - N/p = s - N/q$,
under which both operators become critical for the same nonlinearity.
We consider
\[
  -\Delta_p u + (-\Delta)_q^s u = |u|^{p^*-2}u \qquad \text{in } \mathbb{R}^N,
\]
with $N \geq 2$, $1 < p < N$, $0 < s < 1$, $1 < sq < N$.
In this regime the energy space reduces to $\mathcal{D}_0^{1,p}(\mathbb{R}^N)$, and both best Sobolev constants enter the variational structure simultaneously.
We prove: \textup{(i)} existence of a nonneg\-ative radial ground state via Nehari manifold methods and a double-threshold concentration-compactness analysis; \textup{(ii)} a logarithmic energy estimate, weak comparison principle, and strong maximum principle for all admissible exponents; \textup{(iii)} a weak Harnack inequality; and \textup{(iv)} sharp two-sided decay $U(x) \asymp |x|^{-(N-p)/(p-1)}$ for positive radial solutions, matching the fundamental solution of the $p$-Laplacian.
\end{abstract}

\section{Introduction}
A great deal of the analysis of nonlinear elliptic equations is, at bottom, a
study of \emph{diffusion}. The $p$-Laplacian
\[
-\Delta_p u := -div \big(|\nabla u|^{p-2}\nabla u\big),\qquad 1<p<\infty,
\]
models a flux that responds nonlinearly to the local gradient of a quantity,
and arises in non-Newtonian fluid mechanics, nonlinear elasticity and torsion
problems, glaciology, and image processing \cite{cuccu_plap_nonlinear_elasticity,glowinski_plap_glaciology, kuijper_plap_image_processing}. Its
defining feature is \emph{locality}: the value of $-\Delta_p u$ at a point is
determined by the behaviour of $u$ in an arbitrarily small neighbourhood of that
point. Many phenomena, however, are governed by interactions that are
emphatically \emph{not} local. Particles that move by L\'evy-type jumps,
populations that disperse over long ranges, and phase transitions with nonlocal
interaction energies, see \cite{valdinoci_logistic_equation, pellacci_best_dispersal_strategies, pellacci_logistic_equation}, are all naturally described by the fractional
$q$-Laplacian
\[
(-\Delta)^s_q u(x)
:= \mathrm{P.V.}\int_{\mathbb{R}^N}
\frac{|u(x)-u(y)|^{q-2}\big(u(x)-u(y)\big)}{|x-y|^{N+qs}}\,dy,
\qquad s\in(0,1), \; q \in (1,\infty).
\] 
Thus, the two operators sit at opposite ends of a spectrum. Therefore, a fundamental question, both physically and mathematically, is what happens when both types of diffusion are present simultaneously, i.e., when a system experiences both short-range gradient-driven flux and long-range nonlocal interactions. This is precisely the setting of the operator
\begin{equation*}
  \mathcal{L}_{p,q} := -\Delta_p + (-\Delta)_q^s,
\end{equation*}
which combines the two mechanisms additively. Such \emph{mixed local-nonlocal operators} have been the subject of intense recent investigation and can be seen as a model of optimal animal foraging strategies, dispersal strategies, and anisotropic heat transfer \cite{valdinoci_ecological_niche, blazevski_anisotropic_heat_transport}.

\medskip
The systematic study of $\mathcal{L}_{p,q}$ has, in recent years, grown into one of the most active directions in nonlinear analysis. The cleanest case, $p=q=2$, gives the linear operator $-\Delta+(-\Delta)^s$, for which the systematic work of Biagi, Dipierro, Valdinoci and Vecchi \cite{biagi_dipierro_valdinoci_vecchi_2021,biagi_dipierro_valdinoci_vecchi_2022} established well-posedness, maximum principles, interior and boundary regularity, and a variational spectral theory for the associated Dirichlet problem. Subsequently, some of the results were generalised to the general case of $p = q$ in \cite{silva2024mixed}. 

\medskip 
The spectral theory of mixed local-nonlocal operators has been studied in several recent works. Del~Pezzo, Ferreira, and Rossi \cite{pezzo_ferreira_rossi_eigenvalues_mixed} established that the first Dirichlet eigenvalue of $-\Delta_p + (-\Delta)^s_p$ is isolated and simple, and analyzed its asymptotic behaviour as $p \to \infty$. Isoperimetric inequalities of Faber--Krahn and Hong--Krahn--Szeg\H{o} type for the Dirichlet spectrum of $-\Delta_p + (-\Delta)^s_p$ were subsequently established in \cite{divya_eigenvalue_mixed, biagi_faber_krahn_ineq_mixed,biagi_Hong_krahn_ineq_mixed}.

\medskip 
Moving to multiplicity of solutions to mixed local-nonlocal problems with concave-convex nonlinearities, it have been investigated in \cite{biagi_vecchi_abc_mixed, bhakta_mixed_multiplicity, dhanya_mixed_multiplicity_weights}. Biagi and Vecchi \cite{biagi_vecchi_abc_mixed} obtained a second positive solution to a critical problem driven by $-\Delta + (-\Delta)^s$, which was extended to the quasilinear operator $-\Delta_p + \varepsilon(-\Delta_p)^s$ in \cite{bhakta_mixed_multiplicity} via Ambrosetti--Brezis--Cerami type arguments. For the general operator $\mathcal{L}_{p,q}$ with sign-changing weights, multiplicity results in both subcritical and critical regimes were established in \cite{dhanya_mixed_multiplicity_weights} using Nehari manifold methods.

\medskip 
The regularity theory for mixed local-nonlocal operators has been developed extensively in recent years. In the probabilistic setting, Chen, Kim, Song, and Vondra\v{c}ek \cite{chen_kim_song_uniform_harnack_mixed} established a uniform boundary Harnack principle for the operator $\Delta + \Delta^{\alpha/2}$, connecting it to the superposition of Brownian and stable L\'evy processes. The foundational analytic regularity theory for $-\Delta_p + (-\Delta)^s_p$ was developed by Garain and Kinnunen \cite{garain_juha_transactions_regularity_mixed}, who established local boundedness, H\"older continuity, and Harnack inequalities via De~Giorgi--Nash--Moser theory. This was sharpened by Garain and Lindgren \cite{garain_lindgren_holder_pmixed}, who obtained almost Lipschitz regularity in the homogeneous case and H\"older continuity of the gradient for certain parameter ranges. In the semilinear setting, Su, Valdinoci, Wei, and Zhang \cite{xifeng_valdinoci_mixed_regularity} established $L^\infty$-boundedness and $C^{1,\alpha}$-regularity, subsequently sharpened to $C^{2,\alpha}$-regularity with optimal exponents for subcritical nonlinearities in \cite{xifeng_valdinoci_mixed_c2alpha_regularity}. A strong maximum principle for $-\Delta_p u + (-\Delta)^s_p u = c(x)|u|^{p-2}u$ was proved by Shang and Zhang \cite{shang_zhang_smp_mixed_ev}, while Antonini and Cozzi \cite{antonini_cozzi_gradient_regularity_mixed} established global gradient regularity and a Hopf boundary point lemma for general quasilinear mixed operators. At the variational level, De~Filippis and Mingione \cite{CristianaGiuseppe} proved local $C^{1,\alpha}$-regularity of minimizers of mixed $(p,q)$-functionals with $p \neq q$, and Byun, Lee, and Song \cite{byun_regularity_mixed} extended this program to mixed local-nonlocal double-phase functionals.

\medskip
A second strand of the literature, which we draw on heavily, concerns the function spaces that make these problems well posed. The homogeneous Sobolev spaces $\mathcal{D}^{1,p}_0(\mathbb{R}^N)$ and their fractional
analogues $\mathcal{D}^{s,q}_0(\mathbb{R}^N)$, defined in Subsection \ref{subsec:Function_spaces}, are the correct energy spaces for
problems set on the whole of $\mathbb{R}^N$, where no boundary condition is
available and the only constraint comes from finiteness of energy and decay at
infinity. Their fine structure -- density of smooth functions, the precise
relationship between the ``completion'' and ``zero-trace'' definitions, and the
strict inclusions among them -- were taken from \cite{Brasco_DecayEstimate,brasco_salort_spaces,BrascoCharacterization}.

\subsection*{Motivation and Isocritical Condition}

One of the fundamental questions that arises when working with a differential 
operator concerns the embedding of the underlying Sobolev space into an 
appropriate Lebesgue space. For the $p$-Laplacian, this is the celebrated 
Sobolev critical phenomenon: the embedding 
$\mathcal{D}^{1,p}_0(\mathbb{R}^N) \hookrightarrow L^{p^*}(\mathbb{R}^N)$ 
is continuous but \emph{not} compact, where $p^* = \tfrac{pN}{N-p}$. This 
leads to the study of the best Sobolev constant
\begin{equation}\label{eq:best_constant_local}
    \mathcal{S} := \inf_{u \in \mathcal{D}_0^{1,p}(\mathbb{R}^N) \setminus \{0\}} 
\frac{\displaystyle\int_{\mathbb{R}^N} |\nabla u|^p \, dx}
{\left(\displaystyle\int_{\mathbb{R}^N} |u|^{p^*}\,dx \right)^{\frac{p}{p^*}}},
\end{equation}
whose extremals were identified explicitly by Aubin~\cite{aubin} and 
Talenti~\cite{talenti}. The minimization problem~\eqref{eq:best_constant_local} 
is naturally associated with the critical exponent equation 
$-\Delta_p u = |u|^{p^*-2}u$ in $\mathbb{R}^N$. The systematic study of such 
critical problems in the presence of lower-order perturbations was initiated 
by the seminal work of Brezis and Nirenberg~\cite{brezis_nirenberg}. At the 
critical exponent, standard variational methods break down because 
Palais--Smale sequences may concentrate or escape to infinity without 
converging; compactness must then be recovered through the 
concentration-compactness analysis developed by Lions~\cite{lions_strauss_lemma}.

For the purely nonlocal problem $(-\Delta)^s_q u = |u|^{q_s^*-2}u$, an 
analogous role is played by the fractional best Sobolev constant
\begin{equation}\label{eq:best_constant_nonlocal}
    \mathcal{S}_f := \inf_{u \in \mathcal{D}_0^{s,q}(\mathbb{R}^N) \setminus \{0\}} 
\frac{\displaystyle\iint_{\mathbb{R}^{2N}} 
\dfrac{|u(x) - u(y)|^{q}}{|x-y|^{N+qs}}\, dx\, dy}
{\left(\displaystyle\int_{\mathbb{R}^N} |u|^{q_s^*}\,dx \right)^{\frac{q}{q_s^*}}}.
\end{equation}
The existence of extremals and their sharp decay estimates were established 
in~\cite{cotsiolis_tavoularis, Brasco_DecayEstimate} and the references therein.
\begin{remark}
It is easy to check from Theorem \ref{ineq:gagliardo_nirenberg} that $\mathcal{S}_f \leq C(N,s,p,q) \;\mathcal{S}^{\frac{qs}{p}}$.
\end{remark}
The considerations above naturally motivate the study of the analogous problem 
for the mixed operator $\mathcal{L}_{p,q}$. Accordingly, in this paper we seek 
a function $u \in \mathcal{D}^{1,p}_0(\rnn) \cap \mathcal{D}^{s,q}_0(\rnn)$ satisfying
\begin{equation}\tag{$\mathcal{P}$}\label{eq:main_problem}
  -\Delta_p u + (-\Delta)_q^s u = |u|^{p^*-2}u \quad \text{in } \mathbb{R}^N,
\end{equation}
under the standing assumptions
\[
  N \geq 2, \qquad 1 < p < N, \qquad 0 < s < 1, \qquad 1 < sq < N,
\]
where $p^* = \tfrac{pN}{N-p}$ is the critical Sobolev exponent for the 
$p$-Laplacian.
For the mixed operator the picture is richer, because the two operators carry \emph{different} critical exponents.
The natural critical exponent for $-\Delta_p$ is $p^*=\tfrac{pN}{N-p}$, while that for $(-\Delta)^s_q$ is the fractional Sobolev critical exponent $q^*_s=\tfrac{qN}{N-qs}$.
In general these two thresholds are distinct, and their relative position
dictates the entire structure of the problem:
\begin{itemize}
  \item if $p^*>q^*_s$, the nonlocal term is subcritical relative to the
        nonlinearity and the problem is governed, morally, by the local
        operator alone;
  \item if $p^*<q^*_s$, the roles are exactly reversed.
\end{itemize}
In both of these \emph{non-isocritical} regimes one operator dominates at the
critical scale, the other behaves as a lower-order perturbation, and the analysis
is expected to reduce to a single-operator critical problem.
In fact, we expect to establish the non-existence of positive solutions in these regimes in
our forthcoming work, by extending the Pohozaev identity developed
in~\cite{anthal2025pohozaev} to the present setting.

This is not merely a heuristic expectation: it is precisely what we observed, in
the Hilbertian model case $p=q=2$, in our previous work~\cite{chakraborty2025global}.
There we studied the Brezis--Nirenberg-type problem
\[
  -\Delta u + (-\Delta)^s u - \lambda u = |u|^{2^*-2}u
  \quad\text{in }\Omega\subset\mathbb{R}^N,
  \qquad u=0 \text{ in }\mathbb{R}^N\setminus\Omega,
\]
for the mixed operator $-\Delta+(-\Delta)^s$ on a bounded domain, and analysed
the loss of compactness through a global compactness (profile decomposition)
theorem for the associated Palais--Smale sequences.
The decisive structural fact was a \emph{scaling collapse} of the nonlocal term.
Under the critical rescaling $u_k(x)=k^{(N-2)/2}u(kx)$ that generates the
concentrating bubbles, the Dirichlet energy is scale-invariant, whereas the
Gagliardo seminorm of the nonlocal term satisfies
\[
  [u_k]_s^2 = k^{2s-2}\,[u]_s^2 \longrightarrow 0
  \qquad\text{as } k\to\infty,
\]
since $2s-2<0$.
The nonlocal contribution is thus \emph{strictly subcritical} relative to the
local one.
This was complemented by the compact embedding of the underlying fractional
Sobolev space into the local one (see~\cite[Lemma~2.2]{chakraborty2025global}).
As a consequence (building on Biagi, Dipierro, Valdinoci and
Vecchi~\cite{biagi_dipierro_valdinoci_vecchi_2021,biagi_dipierro_valdinoci_vecchi_2022}),
the best Sobolev constant for the mixed operator coincides with the purely local
one and is \emph{never attained}, the bubbles in the profile decomposition are
exactly the Aubin--Talenti bubbles of the local critical problem
$-\Delta U=|U|^{2^*-2}U$ in $\mathbb{R}^N$, and compactness can be restored
only by a perturbation (the term $\lambda u$) or by domain topology (a
Coron-type construction).
In short, despite its presence in the equation, the fractional term is invisible
to the critical geometry: the problem is governed by the Laplacian alone.

That experience is precisely what motivates the present paper.
The collapse $[u_k]_s^2=k^{2s-2}[u]_s^2\to 0$ hinges on the mismatch between
the scaling exponents of the two operators~-- or equivalently, on the strict
inequality between their critical exponents.
The natural question is therefore: \emph{what happens when the critical exponents
of both operators coincide? Does this remove the defect in the scaling? Does it
change the role of the fractional Laplacian from a subcritical perturbation to
that of a genuine competing operator?}

This delicate case, which is the subject of the present paper, is the
\emph{isocritical} regime:
\begin{equation}\label{eq:isocritical-intro}
  p^*=q^*_s,
  \qquad\text{equivalently}\qquad
  1-\frac{N}{p}=s-\frac{N}{q}.
\end{equation}
This automatically gives us $q = \tfrac{pN}{N-p+ps}$, and $p<q$. The answer to each of the above questions is affirmative.
Here both operators become critical for the \emph{same} nonlinearity at the
\emph{same} time, and both are invariant under the \emph{same} critical
rescaling; the nonlocal operator is thus a genuine co-protagonist rather than a
perturbation.
This shared scaling is at once an obstruction and an opportunity.
A classical embedding theorem (see~\cite[Theorem~7.28]{giovanni_fractional_book})
shows that under the isocritical condition,
$\mathcal{D}^{1,p}_0(\mathbb{R}^N)\hookrightarrow\mathcal{D}_0^{s,q}(\mathbb{R}^N)$,
so the natural energy space for~\eqref{eq:main_problem} collapses to
$\mathcal{D}^{1,p}_0(\mathbb{R}^N)$ alone.
However, this embedding is not compact: if $u_n(x)=n^{(p-N)/p}u(x/n)$ for some
fixed $0<u\in C_c^\infty(\mathbb{R}^N)$, then $(u_n)$ is bounded in
$\mathcal{D}^{1,p}_0(\mathbb{R}^N)$ but admits no convergent subsequence in
$\mathcal{D}_0^{s,q}(\mathbb{R}^N)$.
On the other hand, this very scale-invariance opens the possibility of finding
ground states via concentration-compactness, and this is precisely what we carry
out in Section~\ref{sec:existence} to prove the existence of a ground state
of~\eqref{eq:main_problem}.

The embedding theorem yields two pieces of information.
First, the non-compactness implies that both best Sobolev constants~-- $\mathcal{S}$
for $-\Delta_p$ and $\mathcal{S}_f$ for $(-\Delta)^s_q$~-- enter the energy level,
and the concentration-compactness alternative must account for both thresholds
simultaneously.
To the best of our knowledge, \eqref{eq:main_problem} is the first variational
problem for the mixed $(p,q)$-Laplacian at \emph{double critical} growth, in
which the local and nonlocal best constants jointly determine the variational
level.
This is conceptually distinct from both the purely local critical problem
(see~\cite{brezis_nirenberg}) and the purely fractional one
(see~\cite{mosconi_BN_pfrac}), and it necessitates a genuinely new compactness
analysis that handles both Sobolev thresholds at once
(see Section~\ref{sec:existence}).
Second, the embedding reveals that, even though both operators compete, the
$p$-Laplacian has an advantage over the fractional $q$-Laplacian.
This is consistent with our decay estimates, where solutions
of~\eqref{eq:main_problem} exhibit the pointwise decay of the fundamental
solution of the $p$-Laplacian.

\begin{remark}
  Although the $L^p$ scale does not distinguish between the $p$-Laplacian and the
  fractional $q$-Laplacian, the finer Lorentz space scale does.
  Recall that the Lorentz space $L^{\alpha,\beta}(\mathbb{R}^N)$ is defined in
  Subsection~\ref{subsec:lorentz_space}, and that
  (see~\cite[Exercise~15.19]{giovanni_book_sobolev} and~\cite{peetre_lorentz_spaces})
  \begin{equation}\label{eq:lorentz_embeddings}
    L^{\alpha,\beta_1}(\mathbb{R}^N)\subsetneq L^{\alpha,\beta_2}(\mathbb{R}^N)
    \quad\text{for } 0<\beta_1<\beta_2<\infty,
  \end{equation}
  with strict inclusion witnessed by the function $|x|^{-N/\alpha}\left(\log \left( \frac{2}{ |x|^N}\right)\right)^{-1/\beta_1} \chi_{\{ |x| \leq 1 \}}$. Moreover, by~\cite[Theorem~15.33]{giovanni_book_sobolev},
  \cite{peetre_lorentz_spaces}, and~\cite[Exercise~12.8]{giovanni_fractional_book},
  \begin{equation*}
    \mathcal{D}^{1,p}_0(\mathbb{R}^N)\hookrightarrow L^{p^*,p}(\mathbb{R}^N),
    \qquad
    \mathcal{D}_0^{s,q}(\mathbb{R}^N)\hookrightarrow L^{q_s^*,q}(\mathbb{R}^N).
  \end{equation*}
  Since $p^*=q_s^*$ in the isocritical regime, combining the above with
  \eqref{eq:lorentz_embeddings} (and using $p<q$ in the relevant parameter range)
  makes the dominance of the $p$-Laplacian over the fractional $q$-Laplacian
  visible at the Lorentz scale.
  This, in turn, motivates the study of the minimization problem
  \begin{equation*}
    \mathcal{S}_{N,p^*,p}
    = \inf_{\substack{u\in\mathcal{D}^{1,p}_0(\mathbb{R}^N) \\ u\neq 0}}
      \frac{\|\nabla u\|_{L^p}}{\|u\|_{L^{p^*,p}}},
  \end{equation*}
  whose optimal constant was first determined by Alvino~\cite{alvino_best_Lorentz_Sobolev}:
  \begin{equation*}
    \mathcal{S}_{N,p^*,p}
    = \frac{N-p}{p}\cdot
      \frac{\sqrt{\pi}}{\bigl[\Gamma\!\bigl(1+\tfrac{N}{2}\bigr)\bigr]^{1/N}}.
  \end{equation*}
  Notably, this best constant is \emph{not} attained in
  $\mathcal{D}^{1,p}_0(\mathbb{R}^N)$; the extremals lie in a strictly larger
  space (see~\cite[Theorem~4]{cassani_Sobolevineq_lorentz}
  and~\cite{cassani_Hardy_lorentz}).
\end{remark}

\subsection*{Main results and methodology} 
Some of our results require
$p^*=q^*_s$, while others hold for general exponents. To be precise, the
existence of a ground state and the sharp two-sided decay estimate are proved in
the isocritical regime $p^*=q^*_s$, under which the energy space reduces to
$\mathcal{D}^{1,p}_0(\mathbb{R}^N)$ with norm $[\,\cdot\,]_{1,p}$. By contrast,
the comparison principle, the strong maximum principle, the logarithmic
estimate, and the weak Harnack inequality are established for the operator
$\mathcal{L}_{p,q}$ with \emph{general} exponents $p,q>1$, under the standing
assumptions $1<p<N$, $0<s<1$, $1<sq<N$, with no relation imposed
between $p^*$ and $q^*_s$.

We obtain four main results which, taken together, give a fairly complete picture of
the problem. Concerning the isocritical regime, we prove ground state exists in the radial class, and it is strictly positive. The decay behaviour is settled
for the natural class of positive, radially symmetric, radially decreasing
solutions of \eqref{eq:main_problem}. The strict positivity itself follows from a strong maximum
principle for $\mathcal{L}_{p,q}$, valid for general exponents $1<p\leq q< \infty$.

\bigskip
\noindent\textbf{Existence of a radial ground state (Section \ref{sec:existence}).}\\
The strategy for proving the existence result hinges on a careful treatment of the compactness defects mentioned earlier. The crucial structural step is a \emph{localisation of the loss of compactness} as depicted in  Lemma~\ref{lem:annulus_zero}. For a weakly null $(PS)_{\tilde c}$ sequence $u_n\rightharpoonup0$,
we prove that all three energies - the
$p$-Dirichlet energy $[u_n]_{1,p}^p$, the Gagliardo energy
$[u_n]_{s,q}^q$, and the critical mass $\int|u_n|^{p^*}$ -- tend to zero on
every annulus $B_{c',d'}$. 
The mechanism is the compactness of the embedding
$\mathcal{X}^r_0(\mathbb{R}^N)\hookrightarrow L^r(B_{R_1}\setminus B_{R_2})$ for every $1\le r<\infty$. 
This localisation is precisely what makes the concentration quantities mentioned in \eqref{eq:defining kappas}
independent of the radius $\delta>0$. 

The concentration-compactness Lemma \ref{lem:k3_estimate} proves that the critical mass $\lim\int_{B_\delta}|u_n|^{p^*}$
either vanishes or is bounded below by
$\max\{\mathcal{S}^{p^*/(p^*-p)},\mathcal{S}_f^{q^*_s/(q^*_s-q)}\}$, a
single threshold that simultaneously encodes both Sobolev constants.

A rescaling step (refer Lemma \ref{lem:rescaling})
$\tilde u_n(x)=r_n^{(N-p)/p}u_n(r_nx)$ then pins the $L^{p^*}$-mass at a definite
scale 
$\int_{B_1}|\tilde u_n|^{p^*}=\xi \;\;\forall n$,
ruling out both vanishing and escape, and a Brezis--Lieb decomposition
\cite{brezis_lieb} extracts a nontrivial limit at the ground-state level. The scaling
is legitimate as the exponent identity $-N+qs+q(N-p)/p=0$ holds precisely when
$p^*=q^*_s$, so the rescaling leaves both local and non local energies unchanged. The equality $m = \tilde{c}$ of the Nehari
minimum and mountain pass level completes following the existence argument.

\begin{theorem}\label{thm:existence_equality}
Assume $p^*=q^*_s$, $1<p<N$, $0<s<1$, $1<sq<N$. Then the problem
\eqref{eq:main_problem} admits a non-negative ground state $u\in\mathcal{X}^r_0(\mathbb{R}^N)$,
that is, a nontrivial weak solution of least energy
$\mathcal{J}(u)=\inf_{\mathcal{N}}\mathcal{J}$ on the Nehari manifold
$\mathcal{N}$. 
\end{theorem}

\medskip
\noindent\textbf{Comparison and strong maximum principle (cf. Section~\ref{sec: comparison}).}\\
The technical core of this section is a logarithmic energy estimate
(Lemma~\ref{lem:logarithmic_lemma}) for the mixed operator.
Alongside it, we develop a weak comparison
principle, built on a Simon-type monotonicity inequality \cite{simon_inequalities}
together with the nonlocal monotonicity of
\cite{lindgren_lindqvist_fractional_eigenvalues}. These results are established
for general exponents $p,q>1$ under the standing assumptions, and are also of
independent interest (see also \cite{antonini_cozzi_gradient_regularity_mixed}).

Formulating these results on exterior domains like $B_R^c$, demands a careful choice of functional setting, and this is the role of
the auxiliary space $\mathcal{D}^{s,p,q}(\Omega)$. For an open set
$\Omega\subset\mathbb{R}^N$ we set
\begin{equation*}
\begin{aligned}
\mathcal{D}^{s,p,q}(\Omega):=\Bigl\{&u\in L^{q-1}_{\mathrm{loc}}(\mathbb{R}^N)\cap\mathcal{D}^{1,p}(\Omega)\ :\ \\
&\exists\,E\supset\Omega,\ E^c\text{ compact},\ 
\mathrm{dist}(E^c,\Omega)>0,\ [u]_{s,q,E}<\infty\Bigr\}.
\end{aligned}
\end{equation*}
The difficulty it resolves is that the nonlocal operator $(-\Delta)^s_q$ is
``too'' global, meaning that $u$ must be controlled far away even when only its
behaviour near $\Omega$ is of interest. The space $\mathcal{D}^{s,p,q}(\Omega)$
asks instead for exactly what is needed and no more: the local
$\mathcal{D}^{1,p}$ energy on $\Omega$, mere local $L^{q-1}$ integrability
of the tail, and finite Gagliardo
energy only on a slightly enlarged set $E$ 
that stays at positive distance controlling the near interaction without constraining the
far tail.
Demanding 
$[u]_{s,q}<\infty$ would be too strong because the barrier functions mignt not always have globally finite nonlocal energy. With this setting in place, $\mathcal{L}_{p,q}u$ is a well-defined
continuous functional on the test space
$\mathcal{D}^{1,p}_0(\Omega)\cap\mathcal{D}^{s,q}_0(\Omega)$ (Lemma~\ref{dualityLemma}),
and the comparison principle (Lemma \ref{thm:strong_comparison_mixed}) can be applied to the slowly-decaying
barriers of the decay analysis.

\medskip
\noindent\textbf{Weak Harnack inequality (Section \ref{sec: Harnack}).}\\
The real weight of the proof of the Harnack inequality is carried by the \emph{expansion
of positivity} Lemma \ref{lem:expansion of positivity}, i.e., if $u\ge k$ on a fixed proportion of a ball $B_r$, in the sense that
$|B_r(x_0)\cap\{u\ge k\}|\ge\tau|B_r|$ for some $\tau\in(0,1]$, then
$
\operatorname*{ess\,inf}_{B_{4r}(x_0)} u\ \ge\ \delta k,
$
with $\delta$ depending only on $\tau$, $\|u\|_{L^\infty_{\mathrm{loc}}(\mathbb{R}^{N})}$ and the given data. The proof of this step
is where the mixed structure is felt most acutely. It combines the logarithmic
estimate of Section \ref{sec: comparison}, used to bound the measure 
$\{u\le 2\delta k\}$, with a De~Giorgi-type iteration over a sequence of shrinking
balls and rising levels. In contrast to results of
Section~\ref{sec: comparison}, the Harnack analysis is carried out in the regime $p^*=q^*_s$ and uses the ordering $p<q$. The De~Giorgi iteration balances a local term of
order $r^{-p}$ against a nonlocal term of order $r^{-sq}$, and the isocritical
identity $q^*_s/q=p^*/q$ is what lets a single Sobolev gain exponent $\kappa = q_s^*/q= p^*/q$
control both (estimate \ref{eq:final_recursive_estimate_s2}). Iterating the expansion of positivity over
a geometric sequence of thresholds upgrades it to the weak Harnack inequality
stated below.

\begin{proposition}\label{Weak Harnack inequality}
Let
$
u \in \X(\mathbb{R}^{N})
    \cap L^{\infty}(\mathbb{R}^{N})$
be a non-negative weak supersolution of $\mathcal{L}_{p,q}$.
Then there exist constants $\theta\in(0,1)$ and $C\ge1$,
depending only on $N,p,q,s$ and
$\|u\|_{L^\infty_{\mathrm{loc}}(\mathbb{R}^{N})}$, such that
\[
\left(
\fint_{B_R} u^\theta\,dx
\right)^{\frac1\theta}
\le
C\,\inf_{B_R}u.
\]
\end{proposition}

The role of the inequality for us is qualitative rather than quantitative. Applied to a positive solution, it gives the strict interior
positivity $\inf_{B_1}U\ge c^*>0$ that the decay analysis takes as its starting
point, and this positive floor is the input to the auxiliary problem of
Section~\ref{sec:decayEstimates}.

\medskip
\noindent\textbf{Sharp two-sided decay (refer Section \ref{sec:decayEstimates}).}\\
We obtain the following result concerning the decay of positive, radially symmetric, decreasing solutions $U$ of \eqref{eq:main_problem}.
\begin{theorem}[Decay of positive radial solutions]\label{prop:decay-main}
Assume the isocritical condition \eqref{eq:isocritical-intro}, i.e. $p^\ast = q^\ast_s$, together
with the standing assumptions $1<p<N$, $0<s<1$, $1<q<N/s$, so that
$q = \tfrac{pN}{N-p+ps} > 1$. Set
\[
\alpha := \frac{N-p}{p-1}.
\] Let $U\in \X(\mathbb{R}^N)$ be a positive,
radially symmetric, and decreasing weak solution of \eqref{eq:main_problem}. Then:

\begin{enumerate}
\item[\textup{(i)}] \textup{(Upper bound.)} For all $|x|\ge 1$,
\[
U(x) \;\le\; \left( \omega_N^{-1/p^\ast}\, \mathcal{S}^{-1/p}\, \|U\|_{L^{p^\ast-1}(\mathbb{R}^N)}^{(p^\ast-1)/p} \right)^{\!p/(p-1)} |x|^{-\alpha}.
\]

\item[\textup{(ii)}] \textup{(Lower bound.)} If in addition
\[
q \neq \frac{p(N-1)}{N-p},
\]
then there exists a constant $c_1 = c_1\big(N,p,q,s, \inf_{x\in B_1} U(x\big)\big) > 0$ such that
\[
U(x) \;\ge\; c_1\, |x|^{-\alpha} \qquad \text{for all } |x|\ge 3.
\]
\end{enumerate}
In particular, $U$ decays exactly at the rate $|x|^{-\alpha}=|x|^{-(N-p)/(p-1)}$ of the
fundamental solution of the $p$-Laplacian (see \cite[p.~37]{lindqvist_notes}): there exist $0<c_1\le c_2$ and $R\ge 3$ with
$$
c_1\,|x|^{-\alpha} \;\le\; U(x) \;\le\; c_2\,|x|^{-\alpha}, \qquad |x|\ge R.
$$
\end{theorem}
\begin{remark}
The upper bound in Theorem~\ref{prop:decay-main} uses the radial Lemma
\ref{Radial Lemma for Lorentz spaces}, and thus the radial symmetry of the
solution, whereas the lower bound holds for any nonnegative solution.
\end{remark}
\begin{remark}
It is natural to ask whether every nonnegative ground state of \eqref{eq:main_problem} is radial, which would allow one to drop the symmetry hypothesis altogether. We believe this is true for $p\ge2$ using the techniques developed in \cite{chen_li_symmetry_pfrac} under appropriate regularity assumptions; for $1<p<2$, however, the question appears to remain open.
\end{remark}
This is the qualitative heart of the paper. Although both operators are critical
and neither can be dropped from the equation, it is the local operator that dictates the
large-scale behaviour. The upper bound is reached through a global
$L^\infty$ bound (Proposition~\ref{prop:global_boundedness_mixed}), a borderline
weak-Lorentz estimate placing $U$ in the Lorentz space
$L^{\eta_0,\infty}(\mathbb{R}^N)$ with the critical exponent
$\eta_0=\tfrac{(p-1)N}{N-p}$ (refer Proposition \ref{Borderline Lorentz estimate}), and a radial Lorentz lemma that converts this
membership into the pointwise rate. 
 
\smallskip
\noindent\emph{A remark on the scope of the Lorentz estimate.} The borderline
estimate is proved for
arbitrary exponents $p,q>1$ under the standing assumptions. Its proof uses
only the local Sobolev inequality together with the sign of the nonlocal term, which is
favourable regardless of the relation between $p$ and $q$. Consequently, the
membership $U\in L^{\eta_0,\infty}(\mathbb{R}^N)$, and hence the upper bound
$U(x)\le c_2|x|^{-\alpha}$, hold in generality.
\emph{What is special to the isocritical regime is not the validity of the
estimate but its optimality.} The exponent $\eta_0=\tfrac{(p-1)N}{N-p}$ is the
borderline integrability of the $p$-Laplacian's fundamental solution, and the
rate $|x|^{-\alpha}$ it encodes is attained, i.e.\ the upper bound is matched by
the lower bound and is therefore sharp, precisely when $p^*=q^*_s$, and
$|x|^{-\alpha}$ is the true asymptotic profile. Away from the isocritical
balance the same Lorentz membership still holds, but it is no longer expected to
capture the sharp rate, and $|x|^{-\alpha}$ degrades from an exact rate to a mere
upper envelope. The upper estimate, in other words, is universal; its sharpness is
isocritical.

 \medskip
The lower bound is the delicate half, and its proof splits into two
regimes according to the sign of $\alpha(q-1)-N$, that is, according to whether
$q\gtrless\tfrac{p(N-1)}{N-p}$. The two regimes call for genuinely different
techniques, because the sign of $\alpha(q-1)-N$ decides whether the singular
profile $\Gamma(x)=|x|^{-\alpha}$ can be turned into a usable subsolution at all.

 \medskip
In the regime $q>\tfrac{p(N-1)}{N-p}$ the construction is
direct. One regularises (Lemma \ref{lem:antisymmetry}) $\Gamma$ into a $C^2$ radial function $\Upsilon$,
equal to $|x|^{-\alpha}$ for $|x|\ge1$. The local part of $\mathcal{L}_{p,q}\Upsilon$
vanishes, while the
nonlocal part is strictly negative, $(-\Delta)^s_q\Upsilon(x)\le-c_*|x|^{-N-sq}$ for $|x|\ge k$. Hence $\Upsilon$ is a
subsolution of $\mathcal{L}_{p,q}$ outside a large ball, and a single application
of the comparison principle yields $U(x)\ge c_1|x|^{-\alpha}$.

\medskip
The case $q<\tfrac{p(N-1)}{N-p}$ is the tricky one, since the profile
$\Gamma(x)=|x|^{-\alpha}$ is not directly comparable to our solution. To circumvent this, we
introduce the constrained minimisation problem
\[
\mathcal{A}:=\inf\Bigl\{\,[u]_{1,p}^p+[u]_{s,q}^q\ :\ u\ge\tfrac{c^*}{2}\ \text{on }B_1\,\Bigr\},
\]
whose minimiser $U_0$ replaces $\Gamma$. Being a free minimiser outside $B_1$,
$U_0$ solves $\mathcal{L}_{p,q}U_0=0$ in $B_1^c$, and
satisfies $U_0\equiv\tfrac{c^*}{2}$ on $B_1$, so a comparison
reduces the lower bound for $U$ to one for $U_0$. Since the bare power $\Gamma$ is
not comparable on the boundary, we cap it as $\widetilde\Gamma=\min\{1,|x|^{-\alpha}\}$, for which
$\mathcal{L}_{p,q}\widetilde\Gamma\le C|x|^{-N-sq}$ (Lemma~\ref{lem:nonlocal_upper});
the same truncation applied to $U_0$ gives a lower bound with constant proportional
to $(U_0(1)-U_0(2))^{q-1}$ (Lemma~\ref{lem:nonlocal_barU}), strictly positive by
Lemma~\ref{lem: strict dercresing U_0}, and chaining the comparisons yields
Theorem~\ref{prop:decay-main}.Our sub-solution constructions and the accompanying nonlocal computations on
exterior domains draw on the techniques of \cite{PezzoQuaas,HolderRegularityFractional,Brasco_DecayEstimate}.
\noindent
\begin{remark}[The borderline case $q=\tfrac{p(N-1)}{N-p}$.] The separating
value $\alpha(q-1)=N$ is a genuine borderline, and the reason is a clash of
scales in the nonlocal term (refer \cite[Lemma 7.1]{PezzoQuaas}).
The upper bound is unaffected, as it never uses
the sign of $\alpha(q-1)-N$. Moreover, within the isocritical regime this
borderline together with
\eqref{eq:isocritical-intro} forces $sq=1$. The decay analysis at this
borderline is more delicate and will be carried out in the upcoming work.
\end{remark}

\subsection*{Plan of the paper}

Section~\ref{sec:prelim} fixes notation and assembles the functional framework:
the homogeneous spaces $\mathcal{D}^{1,p}_0(\mathbb{R}^N)$ and
$\mathcal{D}^{s,q}_0(\mathbb{R}^N)$, the isocritical reduction of the energy
space to $\mathcal{D}^{1,p}_0(\mathbb{R}^N)$, the notion of weak (super/sub)solution, the global
$L^\infty$ bound, and the Lorentz-space preliminaries used in the decay
argument. Section~\ref{sec:existence} proves the existence of a radial ground
state through the Nehari manifold and the double-critical
concentration-compactness alternative. Section~\ref{sec: comparison} establishes
the weak comparison principle, the logarithmic estimate, and the strong maximum
principle. Section~\ref{sec: Harnack} develops the weak Harnack inequality.
Section~\ref{sec:decayEstimates} combines these tools to prove the sharp two-sided
pointwise decay.

\section{Notations and Preliminaries}\label{sec:prelim}
\subsection{Notations}

Throughout the paper, we use the following notation.

\begin{itemize}
\item For $R>0$ and $x_0 \in \mathbb{R}^N$,
$
B_R(x_0) := \{ x \in \mathbb{R}^N : |x - x_0| < R \}, 
\quad B_R := B_R(0).
$
\item For $0 < c < d$, $
B_{c,d} := \{ x \in \mathbb{R}^N : c < |x| < d \}.
$

\item For a set $A \subset \mathbb{R}^N$, $
A^c := \mathbb{R}^N \setminus A.
$
\item For any measurable set $E \subset \mathbb{R}^N$, $|E|$ denotes its Lebesgue measure.

\item For sets $A,\Omega\subset\mathbb{R}^N$ we write $A\Subset\Omega$
      (``$A$ compactly contained in $\Omega$'') if $\overline{A}$ is compact
      and $\overline{A}\subset\Omega$.
\item $\omega_N:=|B_1|$ denotes the Lebesgue measure of the unit ball in
      $\mathbb{R}^N$, and $\mathcal{H}^{k}$ the $k$-dimensional Hausdorff measure.
     
\item $
C_{0,r}^\infty(\mathbb{R}^N)
:=
\{ u \in C_c^\infty(\mathbb{R}^N) : u(x) = u(|x|) \ \text{for all } x \in \mathbb{R}^N \}.
$

\item For $1 \le q \le \infty$, $L^q(\mathbb{R}^N)$ denotes the usual Lebesgue space with norm $\|\cdot\|_{L^q}$; for an open set $\Omega \subset \mathbb{R}^N$, $L^q(\Omega)$ is defined analogously with norm $\|\cdot\|_{L^q(\Omega)}$.

\item The letter $C$ denotes a positive constant which may change from line to line and may depend on parameters such as $N, p, q, s$, and $\eta$

\item For $1 < p  < \infty$, the $p$-duality pairing is defined as 
\begin{equation*}
    \langle - \Delta_p u, v \rangle_{p} := \int_{\rnn} |\grad u|^{p-2} \grad u \cdot \grad v \, dx,
\end{equation*}
whenever it is well-defined.
\item For fractional spaces, the bilinear form
\[
\langle u, v \rangle_{s,q} 
:= \iint_{\mathbb{R}^N \times \mathbb{R}^N} 
\frac{|u(x) - u(y)|^{q-2} (u(x)-u(y))(v(x)-v(y))}{|x-y|^{N+qs}}\,dx\,dy,
\]
whenever it is well-defined.

 \item We write the full operator pairing
\[
\hspace{-5cm}\langle\!\langle \mathcal{L}_{p,q}(u),\varphi\rangle\!\rangle
:= \langle -\Delta_p u,\varphi\rangle_p + \langle (-\Delta)^s_q u,\varphi\rangle_{s,q}.\]
whenever the right-hand side is well defined (cf. Lemma~\ref{dualityLemma}).
\item For $q>1$, we define $J_q:\mathbb{R}\to\mathbb{R}$ by $J_q(t):=|t|^{q-2}t$.
\item For measurable sets $X, Y \subset \mathbb{R}^N$, we write
\[
\iint_{X \times Y} f(x,y)\,dx\,dy := \int_{Y} \left( \int_{X} f(x,y)\,dx \right) dy,
\]
whenever the integrals are well-defined.

 \item For a measurable set $E \subset \mathbb{R}^N$ with $0<|E|<\infty$, $(f)_{E}=\fint_E f\,dx := \frac{1}{|E|}\int_E f\,dx$.

 \item $a \gg b$ implies that there exists a constant $C > 0$ such that $a > Cb$ for $a,b \in \real$.
 \item For $t\in\mathbb{R}$, we write $t^{+}:=\max\{t,0\}$ and
      $t^{-}:=\max\{-t,0\}$.
      \item We write $o_n(1)$ for any quantity (possibly depending on other fixed
      parameters) that tends to $0$ as $n\to\infty$.

\end{itemize}

\subsection{Function spaces: a brief survey}\label{subsec:Function_spaces}

We collect here the definitions and basic properties of all function spaces 
used in this paper. Throughout, $N \geq 2$, $1 < p < N$, $0 < s < 1$, and 
$1 < q < \infty$ with $sq < N$.

\subsubsection*{Classical Sobolev spaces on $\Omega$ and $\mathbb{R}^N$}

Let $\Omega \subseteq \mathbb{R}^N$ be an open set. 
For $1 \leq p < \infty$, the \emph{Sobolev space} $W^{1,p}(\Omega)$ is defined by
\[
W^{1,p}(\Omega) 
:= \left\{ u \in L^p(\Omega) : \nabla u \in L^p(\Omega;\mathbb{R}^N)\right\},
\]
endowed with the norm
\[
\|u\|_{W^{1,p}(\Omega)} 
:= \left( \int_\Omega |u|^p\,dx + \int_\Omega |\nabla u|^p\,dx \right)^{1/p}.
\]

The space $W^{1,p}_0(\Omega)$ is the closure of $C_c^\infty(\Omega)$ in 
$W^{1,p}(\Omega)$:
\[
W^{1,p}_0(\Omega) := \overline{C_c^\infty(\Omega)}^{\,\|\cdot\|_{W^{1,p}(\Omega)}}.
\]
When $\Omega$ is bounded with Lipschitz boundary, the Poincar\'e inequality holds: 
there exists $C = C(N,p,\Omega) > 0$ such that
\[
\|u\|_{L^p(\Omega)} \leq C \|\nabla u\|_{L^p(\Omega)} 
\qquad \forall\, u \in W^{1,p}_0(\Omega),
\]
so that $\|\nabla u\|_{L^p(\Omega)}$ is an equivalent norm on 
$W^{1,p}_0(\Omega)$ (see \cite[Corollary $9.19$]{brezis_book}).\\

When $\Omega = \mathbb{R}^N$, it is well known that $W^{1,p}_0(\mathbb{R}^N) = W^{1,p}(\mathbb{R}^N)$. 

\subsubsection*{Homogeneous Sobolev spaces $\mathcal{D}^{1,p}$, 
$\mathcal{D}^{1,p}_0, \widetilde{\mathcal{D}}_0^{1,p}$}

For the critical embedding to take the cleanest form, and because the 
energy functional $\mathcal{J}$ involves only the gradient norm (not the 
full $W^{1,p}$ norm), it is natural to work with homogeneous Sobolev spaces.


For any open set $\Omega \subseteq \mathbb{R}^N$, define the 
\emph{homogeneous Sobolev space}
\[
\mathcal{D}^{1,p}(\Omega) 
:= \left\{ u \in L^{p^*}(\Omega) 
: \nabla u \in L^p(\Omega;\mathbb{R}^N) \right\},
\]
endowed with the seminorm
\[
[u]_{1,p,\Omega} 
:= \left(\int_\Omega |\nabla u|^p\,dx\right)^{1/p}, 
\qquad [u]_{1,p} := [u]_{1,p,\mathbb{R}^N}.
\]

In particular $[\cdot]_{1,p}$ is a genuine 
norm on $\mathcal{D}^{1,p}(\mathbb{R}^N)$.

The space $\mathcal{D}^{1,p}(\mathbb{R}^N)\subsetneq 
W^{1,p}(\mathbb{R}^N)$ when $p < N$. To see this, choose  define
\begin{equation} \label{strictInclusion}
    u(x) := \frac{1}{(1+|x|)^\alpha}
\end{equation}
for any 
$\alpha \in \left(\frac{N-p}{p},\, \frac{N}{p}\right]$. 
\medskip

Further, define
\[
\widetilde{\mathcal{D}}^{1,p}_0(\Omega) 
:= \overline{C_c^\infty(\Omega)}^{\,[\,\cdot\,]_{1,p}}, \]
endowed with the norm $[u]_{1,p} = \|\nabla u\|_{L^p(\mathbb{R}^N)}$.\[
\mathcal{D}_0^{1,p}(\Omega) 
:= \left\{ u \in L^{p^*}(\mathbb{R}^N) 
: \|\nabla u\|_{L^p(\mathbb{R}^N)} < \infty,\ 
u \equiv 0 \text{ a.e. on } \Omega^c \right\},
\]

\paragraph{Relationship between $\mathcal{D}^{1,p}_0(\Omega)$ and 
$\widetilde{\mathcal{D}}_0^{1,p}(\Omega)$.}

We now show precisely when these two spaces coincide. The relationship 
depends on whether $\Omega$ all of $\mathbb{R}^N$ or not.

Now for $\Omega$ bounded and $\partial \Omega$ continious using Remark 2.1 of \cite{brasco_salort_spaces}, we have \[\widetilde{\mathcal{D}}^{1,p}_0(\Omega)= \widetilde{W}_0^{1, p}(\Omega):=\left\{u \in W^{1, p}\left(\mathbb{R}^N\right): u=0 \text { a.e. in } \mathbb{R}^N \backslash \Omega\right\} .\]
Further using Poincar\'e and H\"{o}lder, we have $ \widetilde{W}_0^{1, p}(\Omega) = \mathcal{D}_0^{1,p}(\Omega)$. Moreover, we have the following general
density result.

\begin{theorem}
Let $\Omega\subset\mathbb{R}^N$ be an open set
with continuous boundary.
Then $C^\infty_c(\Omega)$ is dense in
$\mathcal{D}^{1,p}_0(\Omega)$ with respect
to $[\cdot]_{1,p}$.
\end{theorem}

\begin{proof}
Let $u\in\mathcal{D}^{1,p}_0(\Omega)$.
Replacing $u$ by $u^\pm$ if necessary,
assume $u\ge 0$.
For $\varepsilon>0$, set
$u_\varepsilon:=(u-\varepsilon)_+$.
Since $t\mapsto(t-\varepsilon)_+$ is
$1$-Lipschitz: $[u_\varepsilon]_{1,p}
\le[u]_{1,p}<\infty$ and
$u_\varepsilon\equiv 0$ on $\Omega^c$,
so $u_\varepsilon\in\mathcal{D}^{1,p}_0(\Omega)$.
By dominated convergence,
$[u_\varepsilon-u]_{1,p}\to 0$.
Since $u\in L^{p^*}(\mathbb{R}^N)$,
Chebyshev gives $|\{u>\varepsilon\}|<\infty$,
and H\"{o}lder then gives
$u_\varepsilon\in L^p(\mathbb{R}^N)$.
Hence $u_\varepsilon\in \widetilde W^{1,p}_0(\Omega)$,
and since $\partial\Omega$ is continuous,
$C^\infty_c(\Omega)$ is dense in
$\widetilde{W}^{1,p}_0(\Omega)$ with respect to
$\|\cdot\|_{W^{1,p}(\Omega)}$
(see \cite[Theorem~1.4.2.2]{grisvard_book}).
\end{proof}

For $\Omega = \mathbb{R}^N$, the condition $u \equiv 0$ on $\Omega^c$ is 
vacuous, and both definitions reduce to the same space:
\[
\widetilde{\mathcal{D}}_0^{1,p}(\mathbb{R}^N) 
= \mathcal{D}^{1,p}(\mathbb{R}^N)=\mathcal{D}^{1,p}_0(\mathbb{R}^N).
\]
For $p<N$, the equality
$\mathcal{D}^{1,p}_0(\mathbb{R}^N)=\mathcal{D}^{1,p}(\mathbb{R}^N)$
follows from \cite[Proposition 7.2.2]{WillemFunctionalAnalysis}; 
see also \cite[p.~5]{BrascoCharacterization} for a more detailed discussion of this characterization in the case $p \geq N$.

\medskip
\noindent\textbf{Radial subspace.} We write 
\[
\mathcal{D}_r^{1,p}(\mathbb{R}^N) 
:= \left\{ u \in \mathcal{D}^{1,p}(\mathbb{R}^N) : u(x) = u(|x|) 
\text{ a.e.} \right\}.
\]
For the Strauss compactness lemma for the above radial space, the reader may look into \cite[Corollary II.3]{lions_strauss_lemma}.
\subsubsection*{Fractional Sobolev spaces $W^{s,q}$, $\mathcal{D}^{s,q}_0$, 
and $\hat{\mathcal{D}}^{s,q}_0$}
For $0 < s < 1$ and $1 < q < \infty$, 
the \emph{Gagliardo seminorm} of a measurable function 
$u : \Omega \to \mathbb{R}$ is defined by
\[
[u]_{s,q,\Omega} 
:= \left( \iint_{\Omega \times \Omega} 
\frac{|u(x)-u(y)|^q}{|x-y|^{N+qs}}\,dy\,dx \right)^{1/q},
\qquad [u]_{s,q} := [u]_{s,q,\mathbb{R}^N}.
\]
For an open set $\Omega \subseteq \mathbb{R}^N$, the 
\emph{fractional Sobolev space} is
\[
W^{s,q}(\Omega) 
:= \left\{ u \in L^q(\Omega) : [u]_{s,q,\Omega} < \infty \right\},
\]
endowed with the norm
\[
\|u\|_{W^{s,q}(\Omega)} 
:= \left( \|u\|_{L^q(\Omega)}^q + [u]_{s,q,\Omega}^q \right)^{1/q}.
\]
\medskip
Further, define
\[
W^{s,q}_0(\Omega)
:= \overline{C_c^\infty(\Omega)}^{\,\|\cdot\|_{W^{s,q}(\Omega)}},
\]
For any open $\Omega \subseteq \mathbb{R}^N$, we consider the spaces
\begin{align*}
\mathcal{D}^{s,q}(\Omega) 
&:= \left\{ u \in L^{q^*_s}(\Omega) : [u]_{s,q} < \infty \right\},\\[6pt]
\mathcal{D}^{s,q}_0(\Omega) 
&:= \left\{ u \in L^{q^*_s}(\mathbb{R}^N) : [u]_{s,q} < \infty, 
\ u \equiv 0 \text{ on } \Omega^c \right\}.
\end{align*}

endowed with the norm $[u]_{s,q}$. The fractional Sobolev inequality (see \eqref{eq:best_constant_nonlocal}) implies $[\cdot]_{s,q}$ is indeed a norm (and not merely a seminorm) on 
$\mathcal{D}^{s,q}_0(\mathbb{R}^N)$.
For $\Omega = \mathbb{R}^N$ the condition $u \equiv 0$ on $\Omega^c$ 
is vacuous, and we obtain
\[
\mathcal{D}^{s,q}_0(\mathbb{R}^N) 
= \left\{ u \in L^{q^*_s}(\mathbb{R}^N) : [u]_{s,q} < \infty \right\}.
\]
Also,
\[
\hat{\mathcal{D}}^{s,q}_0(\Omega) 
:= \overline{C_c^\infty(\Omega)}^{\,[\,\cdot\,]_{s,q}}. 
\]

\medskip
\noindent\textbf{Relationship between the fractional spaces.}\\
Let $\Omega=\mathbb{R}^N$ and recall our standing assumption $sq<N$.
Then $C_c^\infty(\mathbb{R}^N)$ is dense in $\mathcal{D}^{s,q}_0(\mathbb{R}^N)$
with respect to the seminorm $[\cdot]_{s,q}$(see
\cite[Theorem~3.1]{BrascoCharacterization}).
The proof follows along the same lines as the classical density result
for $W^{s,q}(\mathbb{R}^N)$ with respect to the full norm
$\|\cdot\|_{W^{s,q}}$ (see \cite[Theorem~6.66]{giovanni_fractional_book}).
Consequently,
\[
\hat{\mathcal{D}}^{s,q}_0(\mathbb{R}^N)=\mathcal{D}^{s,q}(\mathbb{R}^N),
\qquad
W^{s,q}_0(\mathbb{R}^N)=W^{s,q}(\mathbb{R}^N).
\]
For an open set $\Omega \subset \mathbb{R}^N$, the situation
is more delicate. The precise relationship between
$W^{s,q}_0(\Omega)$ and $W^{s,q}(\Omega)$ depends on the values of
$s$ and $q$ as well as the regularity of $\partial\Omega$; we refer
to \cite[Theorem 6.105]{giovanni_fractional_book} for a complete
treatment. Regarding $\mathcal{D}^{s,q}_0(\Omega)$,  we recall the following
result from \cite{Brasco_DecayEstimate}.

\begin{theorem}[{\cite[Theorem~2.1]{Brasco_DecayEstimate}}]
Let $\Omega \subset \mathbb{R}^N$ be an open set such that $\partial \Omega$ is compact and locally the graph of a continuous function. Then $\mathcal{D}^{s,q}_0(\om)$ is the completion of $C_c^{\infty}(\Omega)$ with respect to the norm $[\cdot]_{s, q}$ for $sq<N$, i.e., $\hat{\mathcal{D}}^{s,q}_0(\Omega) = \mathcal{D}^{s,q}_0(\Omega)$.
\end{theorem}
Moreover, as observed in \cite[Remark~3.2]{BrascoCharacterization},
an example constructed in a way analogous to the local case
\eqref{strictInclusion} shows that
\[
W^{s,q}_0(\mathbb{R}^N)=W^{s,q}(\mathbb{R}^N)\subsetneq \mathcal{D}^{s,q}(\mathbb{R}^N)=\hat{\mathcal{D}}^{s,q}_0(\mathbb{R}^N).
\]

\begin{remark}
The cases corresponding to $sq\ge N$ are also treated in
\cite{BrascoCharacterization}. For further discussion on the
relationship between the spaces $\mathcal{D}^{1,r}(\Omega)$ and
$\mathcal{D}^{s,r}(\Omega)$ for $r>1$, see
\cite{JeanBrezisMironescu,Ponce}.
\end{remark}

\subsubsection*{The energy space $\mathcal{X}_0$ and the isocritical 
reduction}

The natural energy space for problem \eqref{eq:main_problem} is
$
\mathcal{X}_0 (\rnn)$
where for any open set $\Omega \subset \rnn$
\[\X(\Omega):= \mathcal{D}^{1,p}_0(\om) \cap \mathcal{D}^{s,q}_0(\om)\]
endowed with the norm
\[
\|u\|_{\mathcal{X}_0} := [u]_{1,p} + [u]_{s,q}.
\]
In general, $\mathcal{X}_0(\rnn)$ is a strict subspace of both 
$\mathcal{D}^{1,p}_0(\mathbb{R}^N)$ and 
$\mathcal{D}^{s,q}_0(\mathbb{R}^N)$, and the two norms 
$[\cdot]_{1,p}$ and $[\cdot]_{s,q}$ are independent. 
However, under the isocritical condition \eqref{eq:isocritical-intro} the following embedding holds. 
\begin{theorem}\cite[Theorem 7.28]{giovanni_fractional_book}
Let $1 \leq p_1$, $1 < p_2 < \infty$, and $0<s <1$ be such that 
$$ 1 - \frac{N}{p_2} = s - \frac{N}{p_1}.$$
Then for every $u \in \mathcal{D}_0^{1,p_2}(\rnn)$, there exists $C= C(p_1,p_2,N,s)>0$ such that
\begin{equation}\label{ineq:main_embedding}
  [u]_{s,p_1} \leq C \lv \grad u \rv_{L^{p_2}(\rnn)}.
\end{equation}
Furthermore, $\mathcal{D}_0^{1,p_2}(\rnn) \hookrightarrow \mathcal{D}_0^{s,p_1}(\rnn)$.
\end{theorem}
In fact, a more general version of this exists and is given by 
\begin{theorem}[Gagliardo-Nirenberg]\label{ineq:gagliardo_nirenberg}
Let $0 < s < 1, 1 < p < q < \infty$ with $p^* = q_s^*$. Then the following holds
\begin{equation*}
[u]_{s,q} \leq C [u]_{1,p}^s\lv u \rv_{L^{p^*}(\rnn)}^{1-s}
\end{equation*}
for all $u \in \mathcal{D}^{1,p}(\rnn)$.
\end{theorem}
\begin{proof}
Follows by taking $\theta = 1-s, \, p = q, \, p_1 = p^*, \, p_2 = p,\, s_1 = 0$ and $s_2 = 1$ in Theorem $7.41$ of \cite{giovanni_fractional_book}.
\end{proof}

Therefore, under the assumption \eqref{eq:isocritical-intro}, we have
\[
\mathcal{X}_0(\mathbb{R}^N)
=
\mathcal{D}^{1,p}_0(\mathbb{R}^N),
\qquad
\mathcal{X}(\Omega)
=
\mathcal{D}^{1,p}_0(\Omega),
\]
with norms equivalent to $[u]_{1,p}$. The functional associated with our problem \eqref{eq:main_problem} is denoted by $\mathcal{J} : \X(\rnn)\to \real $ and defined as  
\begin{equation}\label{eq:energy_functional}
    \mathcal{J}(u) := \frac{1}{p} \int_{\rnn} |\grad u|^p dx + \frac{1}{q} [u]_{s,q}^q - \frac{1}{p^*} \int_{\rnn} |u|^{p^*} dx.
\end{equation}

  For $u\in \X(\mathbb{R}^N)$, the Fr\'echet derivative of $J$ acts as
      \[
        \langle J'(u),v\rangle
        =\int_{\mathbb{R}^N}|\nabla u|^{p-2}\nabla u\cdot\nabla v\,dx
        +\langle u,v\rangle_{s,q}
        -\int_{\mathbb{R}^N}|u|^{p^*-2}u\,v\,dx ,
        \quad v\in \X(\mathbb{R}^N).
      \]
Since we are working with radial functions, the following radial lemma will be helpful
\begin{lemma}[Corollary II.1 of \cite{lions_strauss_lemma}]\label{lem:radial_lemma}
For $N \geq 2$, $1 \leq p < \infty$, $u \in \mathcal{D}_{r}^{1,p}(\rnn)$, there exists a continuous function $\tilde u$ with $u = \tilde u$ a.e. such that the following inequality holds
\begin{equation*}
    |\tilde u(x) | \leq C (N,p) \lv u\rv_{1,p} |x|^{(p-N)/p} \qquad \text{for every } x\in \rnn.
\end{equation*}
\end{lemma}

\subsection{Notion of solution}\label{sec: notion of solution}

Let $\Omega \subseteq \mathbb{R}^N$ be open. For $u,\varphi \in \X(\Omega)$ set
\[
\langle\!\langle \mathcal{L}_{p,q}(u),\varphi\rangle\!\rangle := \int_{\Omega} |\nabla u|^{p-2}\nabla u \cdot \nabla\varphi \,dx
+ \iint_{\mathbb{R}^{2N}} \frac{J_q\big(u(x)-u(y)\big)\big(\varphi(x)-\varphi(y)\big)}{|x-y|^{N+qs}}\,dx\,dy .
\]
Under the isocritical condition \eqref{eq:isocritical-intro}, $\X(\Omega) = \mathcal{D}^{1,p}_0(\Omega)$,
and $\langle\!\langle \mathcal{L}_{p,q}(u),\varphi\rangle\!\rangle$ is finite for all $u,\varphi \in \X(\Omega)$ by Hölder's
inequality.
\begin{definition}[Weak solution]
A function $u \in \X(\mathbb{R}^N)$ is a \emph{weak solution} of \eqref{eq:main_problem} if
\[
\langle\!\langle \mathcal{L}_{p,q}(u),\varphi\rangle\!\rangle = \int_{\mathbb{R}^N} |u|^{p^*-2}u\,\varphi \,dx
\qquad \text{for all } \varphi \in \X(\mathbb{R}^N).
\]
\end{definition}

\begin{definition}[Weak sub- and supersolution]\label{def:sub-super}
Let $\Omega \subseteq \mathbb{R}^N$ be open and $u \in \X(\Omega)$.
We say $u$ is a \emph{weak subsolution} of $\mathcal{L}_{p,q}:=-\Delta_p + (-\Delta)_q^s$ in $\Omega$, written
$\mathcal{L}_{p,q}u \le 0$ weakly, if
\[
\langle\!\langle \mathcal{L}_{p,q}(u),\varphi\rangle\!\rangle \le 0
\qquad \text{for all } \varphi \in \X(\Omega),\ \varphi \ge 0 \text{ a.e. in } \mathbb{R}^N,
\]
and a \emph{weak supersolution} of $\mathcal{L}_{p,q}$ in $\Omega$, written $\mathcal{L}_{p,q}u \ge 0$
weakly, if the reverse inequality holds for all such $\varphi$. A function is a
\emph{weak solution} of $\mathcal{L}_{p,q}u = 0$ in $\Omega$ iff it is simultaneously a weak
sub- and supersolution.
\end{definition}

\begin{remark}
When a right-hand side $f$ is present, the inequalities are interpreted against $f$:
$u$ is a weak supersolution of $\mathcal{L}_{p,q}u = f$ if
$\langle\!\langle \mathcal{L}_{p,q}(u),\varphi\rangle\!\rangle \ge \int_{\mathbb{R}^N} f\,\varphi\,dx$
for all nonnegative $\varphi \in \X(\Omega)$ (subsolution: reverse inequality). In
particular, $u$ is a weak supersolution of $\mathcal{L}_{p,q}u = |u|^{p^*-2}u$ if
$\langle\!\langle \mathcal{L}_{p,q}(u),\varphi\rangle\!\rangle \ge \int_{\mathbb{R}^N} |u|^{p^*-2}u\,\varphi\,dx$ for all such $\varphi$.
\end{remark}

\begin{remark}
For the comparison principles in Section~\ref{sec: comparison}, we use a larger class $\mathcal{D}^{s,p,q}(\Omega)$
and test functions $\varphi \in \mathcal{D}^{1,p}_0(\Omega)\cap \mathcal{D}^{s,q}_0(\Omega)$ (see Lemma \eqref{dualityLemma}).
Accordingly, the notions of weak sub-/supersolution of
Definition~\ref{def:sub-super} extend verbatim to $u\in\mathcal{D}^{s,p,q}(\Omega)$
by testing against nonnegative $\varphi \in \mathcal{D}^{1,p}_0(\Omega)\cap \mathcal{D}^{s,q}_0(\Omega)$.
\end{remark}

\subsection{Uniform estimate}
We begin with a global boundedness result for general $1<p<N,q>1, sq<N$, which will be used repeatedly throughout the subsequent analysis. To this end, we first recall an iteration lemma (see \cite[Lemma~7.1]{giusti_book}), which will be instrumental in the argument.
\begin{lemma}\label{lem:iteration_decay}
Let $\{a_j\}_{j=0}^{\infty}$ be a sequence of positive real numbers satisfying 
\begin{enumerate}
    \item[$(i)$] $a_0 \le c_0^{-\frac{1}{\alpha}}\, b^{-\frac{1}{\alpha^2}},$
    \item[$(ii)$] for each $j = 0,1,2,\dots$, we have 
\begin{equation*}
a_{j+1} \le c_0\, b^j\, a_j^{1+\alpha}, \qquad j = 0,1,2,\dots,
\end{equation*}
where $c_0 > 0$, $b > 1$, and $\alpha > 0$ are fixed constants.
\end{enumerate}
Then the sequence $\{a_j\}$ converges to zero.
\end{lemma}
\begin{proposition}\label{prop:global_boundedness_mixed}
Let $u\in \mathcal{D}^{1,p}(\mathbb R^N)\cap \mathcal{D}^{s,q}(\mathbb R^N)$ be a nonnegative weak solution of \eqref{eq:main_problem}. Then
$
u\in L^\infty(\mathbb R^N).
$
\end{proposition}

\begin{proof}

\textbf{Step 1: }This step establishes that $u^{p^*-1}\in L^{\alpha}(\mathbb R^N)$ for $\alpha>\frac{N}{p}$.  Fix $L>0$ and $\beta>1$, and define
\[
u_L:=\min\{u,L\}, \qquad h_{\beta,L}(t):= t\,\min\{t,L\}^{\beta-1}, \qquad t\ge 0.
\]

\medskip
\noindent
Since $h_{\beta,L}$ is Lipschitz continuous, the standard stability of Sobolev and fractional Sobolev spaces under Lipschitz compositions yields
$
h_{\beta,L}(u)\in \mathcal{D}^{1,p}(\mathbb R^N)\cap \mathcal{D}^{s,q}(\mathbb R^N).
$ Therefore, $h_{\beta,L}(u)$ is an admissible test function in the weak formulation, and hence
\begin{align}
&\int_{\mathbb R^N} |\nabla u|^{p-2}\nabla u\cdot \nabla(h_{\beta,L}(u))\,dx + \iint_{\mathbb R^{2N}}
\frac{J_q(u(x)-u(y))\bigl(h_{\beta,L}(u(x))-h_{\beta,L}(u(y))\bigr)}{|x-y|^{N+qs}}\,dx\,dy
\notag\\
& \hspace{6cm}= \int_{\mathbb R^N} u^{p^*-1}h_{\beta,L}(u)\,dx
= \int_{\mathbb R^N} u^{p^*}u_L^{\beta-1}\,dx.
\label{eq:test-h}
\end{align}

Now define
\[
H_{\beta,L}(t):=\int_0^t \bigl(h'_{\beta,L}(\tau)\bigr)^{1/p}\,d\tau.
\]

By the chain rule and since $h_{\beta,L}$ is increasing on $[0,\infty)$,  we have
\[
|\nabla H_{\beta,L}(u)|^p
=
h'_{\beta,L}(u)\,|\nabla u|^p
=
|\nabla u|^{p-2}\nabla u\cdot \nabla(h_{\beta,L}(u)).
\]

Therefore the first term in \eqref{eq:test-h} becomes
$
\int_{\mathbb R^N} |\nabla H_{\beta,L}(u)|^p\,dx.
$ For the nonlocal term, since $h_{\beta,L}$ is increasing on $[0,\infty)$, for every $a,b\ge0$ we have
$
J_q(a-b)(h_{\beta,L}(a)-h_{\beta,L}(b))\ge0,
$
see also Lemma A.2 in \cite{brasco_second_eigenvalue}. Consequently, using the Sobolev inequality \eqref{eq:best_constant_local}, we obtain from \eqref{eq:test-h}
\[
\mathcal{S}\left(\int_{\mathbb R^N} |H_{\beta,L}(u)|^{p^*}\,dx\right)^{p/p^*}
\le \int_{\mathbb R^N} u^{p^*}u_L^{\beta-1}\,dx.
\]
A direct computation yields the following estimate for the function $H_{\beta,L}$
\[
H_{\beta,L}(t)
\ge
\frac{p}{p+\beta-1}\,
t\,\min\{t,L\}^{\frac{\beta-1}{p}} .
\]

Substituting this in the Sobolev estimate, we obtain

\[
\mathcal{S}\left(\frac{p}{p+\beta-1}\right)^p
\left(\int_{\mathbb R^N} u^{p^*}u_L^{(\beta-1)\frac{p^*}{p}}\,dx\right)^{p/p^*}
\le \int_{\mathbb R^N} u^{p^*}u_L^{\beta-1}\,dx.
\]
We now estimate the right-hand side.
Fix $M_0>0$, to be chosen later, we get
\begin{align*}
\int_{\mathbb R^N} u^{p^*}u_L^{\beta-1}\,dx
&=
\int_{\{u<M_0\}} u^{p^*}u_L^{\beta-1}\,dx
+
\int_{\{u\ge M_0\}} u^{p^*}u_L^{\beta-1}\,dx\\
& \leq M_0^{\beta-1}\int_{\mathbb R^N} u^{p^*}\,dx+ \left(\int_{\{u\ge M_0\}} u^{p^*}\,dx\right)^{\frac{p^*-p}{p^*}}
\left(\int_{\mathbb R^N}(u^p u_L^{\beta-1})^{\frac{p^*}{p}}\,dx\right)^{\frac{p}{p^*}}.
\end{align*}

Combining the previous estimates, we arrive at
\begin{align}
\mathcal{S}\left(\frac{p}{p+\beta-1}\right)^p &
\left(\int_{\mathbb R^N} u^{p^*}u_L^{(\beta-1)\frac{p^*}{p}}\,dx\right)^{p/p^*} \notag \\ 
\le M_0^{\beta-1}&\int_{\mathbb R^N} u^{p^*}\,dx +
\left(\int_{\{u\ge M_0\}} u^{p^*}\,dx\right)^{\frac{p^*-p}{p^*}}
\left(\int_{\mathbb R^N}(u^p u_L^{\beta-1})^{\frac{p^*}{p}}\,dx\right)^{\frac{p}{p^*}}. \label{estimate1}
\end{align}
Now choose $\beta>1$ so that
\[
p^*+(\beta-1)\frac{p^*}{p}=\alpha(p^*-1), \text{i.e.,\;}\beta
=p\,\alpha\,\frac{p^*-1}{p^*}-(p-1),
\]
for some $\alpha>\frac{N}{p}$.
Notice that this choice is possible. Therefore we may choose $M_0=M_0(\beta,u)>0$ so large that
\[
\left(\int_{\{u\ge M_0\}}u^{p^*}\,dx\right)^{\frac{p^*-p}{p^*}}
\le
\frac{\mathcal{S}}{2}
\left(\frac{p}{p+\beta-1}\right)^p.
\]

With this choice, estimate \eqref{estimate1} yields
\[
\left(\frac{p}{p+\beta-1}\right)^p
\left(
\int_{\mathbb R^N} u^{p^*}\,u_L^{(\beta-1)\frac{p^*}{p}}\,dx
\right)^{\frac p{p^*}}
\le
\frac{2}{\mathcal{S}}\,M_0^{\beta-1}\int_{\mathbb R^N}u^{p^*}\,dx.
\]
We now let $L\to\infty$. Since $u_L=\min\{u,L\}\uparrow u$ pointwise yields
$
u^{p^*-1}\in L^\alpha(\mathbb R^N)\text{ for } \alpha>\frac Np.$

\medskip
\textbf{Step 2: }In this step, we complete the proof by arguing in the spirit of
\cite[Proposition~2.1]{CristianaGiuseppe}. Note that, since $\alpha>\frac{N}{p}$, we have
$
p^*>\, p\,\alpha^{'}, 
$
 where $\alpha^{'}=\frac \alpha{\alpha-1}$. Choose $\ell>0$, to be fixed later. For $j\in\mathbb N\cup \{0\}$, define
\[
\theta_j:=2\ell\Bigl(1-2^{-j-1}\Bigr),
\qquad
w_j:=(u-\theta_j)_+ .
\]

Then
\[
\ell\le \theta_j\le 2\ell,
\qquad
\theta_j\uparrow 2\ell,
\]

and $w_j\in \mathcal{D}^{1,p}(\mathbb R^N)\cap \mathcal{D}^{s,q}(\mathbb R^N)$ for every $j$.
Testing the weak formulation with $w_{j+1}$, we get
\begin{align}
0
=\int_{\mathbb R^N} |\nabla u|^{p-2}\nabla u\cdot \nabla w_{j+1}\,dx
&+\iint_{\mathbb R^{2N}}
\frac{J_q(u(x)-u(y))(w_{j+1}(x)-w_{j+1}(y))}{|x-y|^{N+qs}}\,dx\,dy \notag\\
&\qquad \qquad 
-\int_{\mathbb R^N} u^{p^*-1}\,w_{j+1}\,dx:=J_1+J_2-J_3.\label{testfunction}
\end{align}
Using the Sobolev inequality \eqref{eq:best_constant_local}, we obtain
$$J_1 \geq \mathcal{S}\,\|w_{j+1}\|_{L^{p^*}(\mathbb R^N)}^p.$$
Before estimating $J_3$, we observe that since $u\in \mathcal{D}^{1,p}(\mathbb R^N)$, it follows that $u\in L^{p^*}(\mathbb R^N)$. Therefore, by Chebyshev's inequality, for every $k>0$,
\[
|\{u>k\}|
\le
\frac{1}{k^{p^*}}
\int_{\mathbb R^N} u^{p^*}\,dx.
\]
In particular,
\[
|\{w_{j+1}>0\}|=|\{u>\theta_{j+1}\}|<\infty.
\]
Now, because $\alpha^{'}<p^*$, we apply Hölder's inequality to obtain
\begin{align*}
J_3:=\int_{\mathbb R^N} u^{p^*-1}\,w_{j+1}\,dx
&\le \|u^{p^*-1}\|_{L^\alpha(\mathbb R^N)}
\|w_{j+1}\|_{L^{p^*}(\mathbb R^N)}
\,|\{w_{j+1}>0\}|^{\frac1{\alpha^{'}}-\frac1{p^*}}.
\end{align*}
We next consider the nonlocal term $J_2$. The function $t\mapsto (t-\theta_{j+1})_+$ is increasing, hence
\[
(u(x)-u(y))(w_{j+1}(x)-w_{j+1}(y))\ge0 .
\]
Consequently,
$
J_2\ge0 $. Putting together the above estimates for $J_1$, $J_2$, and $J_3$ in \eqref{testfunction}, we obtain
\begin{equation}
\mathcal{S}\;\|w_{j+1}\|_{L^{p^*}(\mathbb{R}^N)}^{p}
\le
\,\|u^{p^*-1}\|_{{L^\alpha}(\mathbb R^N)}\,
\|w_{j+1}\|_{L^{p^*}(\mathbb R^N)}
|\{w_{j+1}>0\}|^{\frac1{\alpha^{'}}-\frac1{p^*}} . \label{eq:iterA}
\end{equation}
Next we estimate the measure of the superlevel set. On $\{w_{j+1}>0\}$ we have
\[
w_j=(u-\theta_j)_+\ge \theta_{j+1}-\theta_j= \ell\,2^{-j-1}.
\]

Therefore
\[
|\{w_{j+1}>0\}|
\le
\left(\frac{2^{j+1}}{\ell}\right)^{p^*}
\int_{\mathbb R^N} w_j^{p^*}\,dx.
\]
Writing $
\eta_j:=\ell^{-p}\|w_j\|_{L^{p^*}(\mathbb R^N)}^p.
$ Then, using \eqref{eq:iterA} together with the monotonicity $w_{j+1}\le w_j$, we obtain
\[
\eta_{j+1}
\le 
C(p,N)\|u^{p^*-1}\|_{{L^\alpha}(\mathbb R^N)}\,\left(2^{j}\right)^{\frac{p^{*}-\alpha^{'}}{\alpha^{'}}} \ell^{1-p}\,\eta_j^{\frac{p^*}{\alpha^{'}p}},
\]

Set
\[
c_0
=
C\|u^{p^*-1}\|_{L^\alpha(\mathbb R^N)}\,\ell^{1-p},
\qquad
b
=
2^{\frac{p^{*}-\alpha^{'}}{\alpha^{'}}},
\qquad
\alpha_{0}
=
\frac{p^*}{p \alpha^{'}}-1 .
\]

Next observe that
\[
\eta_0
=
\ell^{-p}\|(u-\ell)_+\|_{L^{p^*}(\mathbb R^N)}^p
\le
\ell^{-p}\|u\|_{L^{p^*}(\mathbb R^N)}^p .
\]

Because $p+\frac{p-1}{\alpha}>0$, we may enlarge the lower bound on $\ell(\|u\|_{L^{p^*}(\mathbb R^N)},\|u^{p^*-1}\|_{{L^\alpha}(\mathbb R^N)},N,p)$ so that
\[ \eta_0\le \ell^{-p}\|u\|_{L^{p^*}(\mathbb R^N)}^p\leq {\left(C\|u^{p^*-1}\|_{{L^\alpha}(\mathbb R^N)}\ell^{1-p}\right)}^{-1/\alpha_{0}}{\left(2^{\frac{p^{*}-\alpha^{'}}{\alpha^{'}}}\right)}^{-1/\alpha_{0}^2}.
\]

Therefore the assumptions of Lemma \ref{lem:iteration_decay} are satisfied, and we conclude that
$
\eta_j \to 0 .
$  
This implies
\[
0=\lim_{j\to\infty}\eta_j
=
\ell^{-p}
\left(
\int_{\mathbb R^N} (u-2\ell)_+^{p^*}\,dx
\right)^{p/p^*}.
\]
Hence, $
u\le 2\ell \;\text{ a.e. in }\mathbb R^N.
$ Since $u\ge0$, this proves $u\in L^\infty(\mathbb R^N)$.
\end{proof}

\subsection{Lorentz Spaces}\label{subsec:lorentz_space}
Let $u:\mathbb{R}^N\to\mathbb{R}$ be a measurable function. Its distribution function is defined by
\[
\mu_u(t)
:=
\bigl|\{x\in\mathbb{R}^N:\ |u(x)|>t\}\bigr|,
\qquad t>0.
\]

For $0<\beta<\infty$ and $0<\theta<\infty$, the Lorentz space
$L^{\beta,\theta}(\mathbb{R}^N)$ is defined by
\[
L^{\beta, \theta}\left(\mathbb{R}^N\right):=\left\{u:\int_0^\infty
t^{\theta-1}
\, \mu_u(t)^{\frac{\theta}{\beta}}
\, dt
<\infty\right\}.
\]

When $\theta=\infty$, the weak Lorentz space $L^{\beta,\infty}(\mathbb{R}^N)$ is given by
\[
L^{\beta,\infty}(\mathbb{R}^N)
:=
\left\{
u:\sup_{t>0}
t\, 
\mu_u(t)^{1/\beta}
<\infty
\right\}.
\]

\begin{lemma}[Radial Lemma for Lorentz spaces]\label{Radial Lemma for Lorentz spaces}
 Let $0<\theta \leq \infty$ and $0<\beta<\infty$. Let $u \in L^{\beta, \theta}\left(\mathbb{R}^N\right)$ be a non-negative and radially symmetric decreasing function. Then
 \begin{equation*}
\begin{aligned}
& 0 \leq u(x) \leq\left(\theta \omega_N^{-\frac{\theta}{\beta}} \int_0^{\infty} t^{\theta-1} \mu_u(t)^{\frac{\theta}{\beta}} d t\right)^{\frac{1}{\theta}}|x|^{-\frac{N}{\beta}}, \quad \text { if } \theta<\infty \\
& 0 \leq u(x) \leq\left(\omega_N^{-\frac{1}{\beta}} \sup _{t>0} t \mu_u(t)^{\frac{1}{\beta}}\right)|x|^{-\frac{N}{\beta}}, \quad \text { if } \theta=\infty .
\end{aligned}
\end{equation*}
\end{lemma}

For the proof of the above Lemma, refer Lemma 2.9 of \cite{Brasco_DecayEstimate}.


\section{Existence in Radial case}\label{sec:existence}
This section is devoted to proving the existence of a nonnegative radial ground state solution to \eqref{eq:main_problem} under the assumption \eqref{eq:isocritical-intro}. For this purpose, we work in the radial subspace
\[
\X^r(\mathbb{R}^N) 
:= \left\{ u \in \X(\mathbb{R}^N) : u(x) = u(|x|) 
\text{ a.e.} \right\}.
\]
The aim of this section is to prove Theorem \ref{thm:existence_equality}.
The proof is based on a variational approach in the radial framework. We begin by collecting some standard preliminary results concerning the geometry of the associated energy functional (see, for instance, \cite{willem}).
\begin{lemma}\label{MtPassGeometry}
    The functional $\mathcal{J}$ satisfies the mountain pass geometry. 
\end{lemma}
To prove the existence of a solution of \eqref{eq:main_problem} we crucially use the Nehari manifold associated with the functional $\mathcal{J}$ given by \eqref{eq:energy_functional}
\begin{equation*}
    \mathcal{N} = \{ u \in \X^r (\rnn) : \langle \mathcal{J}^{\prime} (u), u \rangle  = 0 \}.
\end{equation*}
Moreover, we take 
\begin{equation*}
    m = \inf_{ u \in \mathcal{N}} \mathcal{J}(u).
\end{equation*}

\begin{lemma}
For every nonzero $u \in \X^r(\mathbb{R}^N)$, there exists a unique scalar $t_u>0$ such that $t_u u$ lies on the Nehari manifold $\mathcal N$. Moreover, $t_u$ corresponds to the unique maximizer of the map $t \mapsto \mathcal J(tu)$ for $t \ge 0$. The assignment $u \mapsto t_u$ is continuous, and the projection $u \mapsto t_u u$ establishes a homeomorphism between the unit sphere of $\X^r(\mathbb{R}^N)$ and the Nehari manifold $\mathcal N$.
\end{lemma}

By Lemma~\eqref{MtPassGeometry} and Theorem 1.15 in \cite{willem}, we obtain the existence of a Palais--Smale sequence at the mountain pass level.

\begin{lemma}\label{ExistenceOfPSsequence}
There exists a sequence $\{u_n\} \subset \X^r(\mathbb{R}^N)$ such that
\begin{equation}
\mathcal{J}(u_n) \to \tilde{c} > 0
\quad \text{and} \quad
\mathcal{J}'(u_n) \to 0 \quad \text{as } n \to \infty, \label{PScondition}
\end{equation}
where the mountain pass level $\tilde{c}$ is defined by
\[
\tilde{c}
:=
\inf_{\gamma \in \Gamma_{mp}}
\ \sup_{t \in [0,1]} \mathcal{J}(\gamma(t)),
\]
and
\[
\Gamma_{mp}
:=
\left\{
\gamma \in C([0,1], \X^r(\mathbb{R}^N)) :
\gamma(0)=0,\ \mathcal{J}(\gamma(1))<0
\right\}.
\]
\end{lemma}

\smallskip
Any sequence $\{u_n\}$ satisfying \eqref{PScondition} will be referred to as a Palais--Smale sequence at level $\tilde{c}$, abbreviated as a $(PS)_{\tilde{c}}$ sequence.

\smallskip
\begin{lemma}\label{lem:m_equals_mtpasslevel}
$
m = \tilde{c}.
$
\end{lemma}

Throughout the subsequent analysis, we pass to subsequences when necessary, while continuing to denote them by the original sequence.
\begin{lemma}\label{lem:annulus_zero}
Let $\{u_n\} \subset \X^r(\mathbb{R}^N)$ be a 
Palais-Smale sequence for $\mathcal{J}$ at level $\tilde{c}$, 
as produced in Lemma~\ref{ExistenceOfPSsequence}, satisfying 
$u_n \rightharpoonup 0$ in $\X^r(\mathbb{R}^N)$. Then, for every annulus $B_{c^{\prime},d^{\prime}} := \{x \in \mathbb{R}^N : c^{\prime} < |x| < d^{\prime}\}$, the following limits hold
\begin{align*}
    \int_{B_{c^{\prime},d^{\prime}}} |\nabla u_n|^p \,  dx \to 0,  \quad \iint_{B_{c^{\prime},d^{\prime}} \times \mathbb{R}^N} \frac{|u_n(x) - u_n(y)|^q}{|x-y|^{N+qs}} \, dy\, dx \to 0, \quad 
    \int_{B_{c^{\prime},d^{\prime}}} |u_n|^{p^*} \, dx \to 0.
\end{align*}
\end{lemma}
\begin{proof}
Let $\eta \in C^\infty_{0,r}(\mathbb{R}^N)$ satisfy $0 \leq \eta \leq 1$, where $0<c<c ^{\prime}<d^{\prime}<d<\infty$,
$\eta \equiv 1$ on $B_{c^\prime,d^\prime}$, and 
$\operatorname{supp}(\eta) \subset B_{c,d}$.
By the compact embedding result (see \cite[Lemma~6]{willem_radial_cpt_embeddingp>2}), we have
\begin{equation}
\X^r(\mathbb{R}^N) \hookrightarrow L^q(B_{R_1} \setminus B_{R_2})
\quad \text{compactly}, \label{AnnulusCompactEmbedding}
\end{equation}
for any $R_1 > R_2 > 0$ and $1 \leq q < \infty$.  Consequently,
\begin{equation}\label{eq:NonLinearityAnnulusConvergence}
\int_{B_{c,d}} |u_n|^{p^*} \, dx \longrightarrow 0
\quad \text{as } n \to \infty.
\end{equation}

To prove the lemma, we consider the function $\eta u_n \in \X^r(\mathbb{R}^N)$ as a test function in \eqref{PScondition} and estimate the quantity $\langle \mathcal{J}'(u_n), \eta u_n \rangle$. To this end, we analyze the resulting terms.\\
By Hölder's inequality and \eqref{AnnulusCompactEmbedding}, we obtain
\begin{equation*}
\begin{aligned}
\int_{\mathbb{R}^{N}} 
|\nabla u_n|^{p-1} |\nabla\eta| |u_n| \, dx 
\leq 
\left( \int_{\mathbb{R}^{N}} |\nabla u_n|^{p} \, dx \right)^{\tfrac{p-1}{p}}
\left( \int_{\mathbb{R}^{N}} |\nabla\eta|^{p} |u_n|^{p} \, dx \right)^{\tfrac{1}{p}} \to 0 
\quad \text{as } n \to \infty.
\end{aligned}
\end{equation*}

Now, consider
\begin{align*}
I & := \iint_{\mathbb{R}^{2N}} 
|u_n(x)|^{q} \, \big| \eta(x) - \eta(y) \big|^{q} \, \frac{dy \, dx}{|x-y|^{N + qs}}\\
& = 
\iint_{B_{c,d} \times B_{c,d}} 
+ \iint_{B_{c,d}^{c} \times B_{c,d}} 
+ \iint_{B_{c,d} \times B_{c,d}^{c} } := I_{1} + I_{2} + I_{3}.
\end{align*}
Let $\varepsilon > 0$ be chosen such that $\eta$ is Lipschitz continuous whenever $|x - y| < \varepsilon$ with Lipschitz constant denoted by $C_\eta$. Then, we have
\begin{equation*}
\begin{aligned}
I_{1} 
&:= 
\iint_{B_{c,d} \times B_{c,d}} 
|u_{n}(x)|^{q} \, |\eta(x) - \eta(y)|^{q} \, \frac{dy \, dx}{|x-y|^{N + qs}} \\
&\leq 
\iint_{\left(B_{c,d} \times B_{c,d} \right)\cap \{|x - y| < \varepsilon\}} 
\frac{C_\eta \, |u_{n}(x)|^{q}}{|x - y|^{N + q s - q}} \, dy \, dx\\
&\quad+
\iint_{\left(B_{c,d} \times B_{c,d}\right) \cap \{|x - y| \geq \varepsilon\}} 
\frac{|u_{n}(x)|^{q} \, |\eta(x) - \eta(y)|^{q}}{|x - y|^{N + q s}} \, dy \, dx\\
&\leq
C(\varepsilon,q,s,N, \eta) \int_{B_{c,d}} |u_{n}(x)|^{q} \, dx
+ 
\frac{2^{q}}{\varepsilon^{N + q s}} 
|B_{c,d}| 
\int_{B_{c,d}} |u_{n}(x)|^{q} \, dx.
\end{aligned}
\end{equation*}
By compact embedding, we have $u_{n} \to 0 $ in $L^{q}(B_{c,d})$. Thus, $\lim_{n \to \infty} I_{1} = 0$. Similarly, we obtain $\lim_{n \to \infty} I_{2} = 0$. Indeed
\begin{equation*}
\begin{aligned}
I_{2} 
&\leq 
\iint_{\left(B_{c,d}^{c} \times B_{c,d}\right) \cap \{|x - y| < \varepsilon\}}
\frac{C_\eta |u_{n}(x)|^{q}}{|x - y|^{N + q s - q}} \, dy \, dx\\
&\quad+
\iint_{\left(B_{c,d}^{c} \times B_{c,d}\right) \cap \{|x - y| \geq \varepsilon\}}
\frac{|u_{n}(x)|^{q} |\eta(x) - \eta(y)|^{q}}{|x - y|^{N + q s}} \, dy \, dx.
\end{aligned}
\end{equation*}

Now, consider
\begin{align*}
&I_{3} := 
\iint_{ B_{c,d}\times B_{c,d}^{c}} 
\frac{|u_{n}(x)|^{q} |\eta(x) - \eta(y)|^{q}}{|x - y|^{N + q s}} 
\, dy \, dx \\
&\leq 
\iint_{\left(B_{c,d} \times B_{c,d}^{c} \right)\, \cap \, \{|x - y| < \varepsilon\}}
\frac{C_\eta |u_{n}(x)|^{q}}{|x - y|^{N + q s - q}} \, dy \, dx
\\ & \quad + 
\iint_{\left(B_{c,d}\times B_{c,d}^{c}\right)  \, \cap \, \{|x - y| \geq \varepsilon\}}
\frac{|u_{n}(x)|^{q} |\eta(x) - \eta(y)|^{q}}{|x - y|^{N + q s}} 
\, dy \, dx =: J_{1} + J_{2}.
\end{align*}
Further, 
for $y \in B_{c,d}$ and $x \in B_{c,d}^c$ with 
$|x-y| < \varepsilon$, the triangle inequality gives
$
|x| \leq d + \varepsilon,
$
Then 
\[
J_1 \leq C_\eta\int_{|x| \leq d+\varepsilon}|u_n(x)|^q
\left(\int_{|z|<\varepsilon}
\frac{dz}{|z|^{N+qs-q}}\right)dx
= C(\varepsilon,q,s,N,\eta)
\int_{|x| \leq d + \varepsilon}|u_n(x)|^q\,dx,
\]
which tends to zero as $n \to \infty$.\\

For any $M > \max\{\varepsilon, 2d\}$, we decompose $J_2$ into a near-field and a far-field part:
\begin{equation*}
\begin{aligned}
J_2 & \leq \iint\limits_{\left(B_{c,d}\times B_{c,d}^{c}\right)  \, \cap \, \{\varepsilon \leq |x - y| < M\}} \frac{C |u_n(x)|^q}{|x-y|^{N+qs}} \, dy\,dx \\ & \quad + \iint\limits_{\left(B_{c,d} \times B_{c,d}^{c}\right) \, \cap \, \{|x - y| \geq M\}} \frac{|u_{n}(x)|^{q} |\eta(x) - \eta(y)|^{q}}{|x-y|^{N+qs}} \, dy\,dx  := K_1 + K_2.
\end{aligned}
\end{equation*}
Since $x \in B_{c,d}^c$, $y \in B_{c,d}$ and $|x-y|<M$, the triangle inequality yields $
|x| \le |x-y| + |y| < M + d,
$
so that $x \in B_{M + d}(0)$. Hence,
\begin{equation*}
\lim_{n \to \infty} K_1 = 0 \qquad \text{for every fixed } M > \varepsilon > 0.
\end{equation*}
Next, we estimate $K_2$. 
\begin{align*}
   & K_2:= \int_{B_{c,d}^{c}}  \int_{B_{c,d}\cap B_{M}^c(x)}  \frac{|u_{n}(x)|^{q} |\eta(x) - \eta(y)|^{q}}{|x-y|^{N+qs}} \, dy\,dx\\
    &\leq \int_{B_{c}} |u_n(x)|^{q} \left(\int_{B_{M}^c(x)}\frac{2^q \; dy}{|x-y|^{N+qs}} \, \right)dx + \; \int_{B_{d}^c} \int_{B_{c,d}\cap B_{M}^c(x)} \frac{|u_{n}(x)|^{q} |\eta(x) - \eta(y)|^{q}}{|x-y|^{N+qs}}  dydx.
\end{align*}
Clearly, the first term is $o_n(1)$. For the second term, we apply Lemma~\ref{lem:radial_lemma}. 
By the choice of $M$, for all $x \in B_d^c(0)$ and $y \in B_{c,d}$ with $|x-y| \ge M$, we have
$
|x| \ge M - d > d > 0.
$. Hence
\[
|u_n(x)|^q \le C(N,p,q,s)\, |x|^{q(p-N)/p} \|u_n\|_{1,p}^q
\le C(N,p,q,s)\, (M - d)^{q(p-N)/p} \|u_n\|_{1,p}^q.
\]

Therefore,
\begin{equation*}
J_2 \le o_n(1) + C(N,s,p,q)\, (M - d)^{q(p-N)/p} [\eta]_{s,q}^q
\end{equation*}
for all $M > \max\{\varepsilon, 2d\}$.

Fixing $M$ and letting $n \to \infty$, we get
\[
\limsup_{n\to\infty} J_2 
\le C(N,s,p,q)\, (M - d)^{q(p-N)/p} [\eta]_{s,q}^q.
\]

Finally, letting $M \to \infty$ and recalling that $q(p-N)/p < 0$, we conclude
$
\lim_{n\to\infty} J_2 = 0.
$

Combining all the estimates, we obtain
\begin{equation}
\lim_{n\to\infty} I:=\lim_{n\to\infty}\iint_{\mathbb{R}^{2N}} 
|u_n(x)|^{q} \, \big| \eta(x) - \eta(y) \big|^{q} \, \frac{dy \, dx}{|x-y|^{N + qs}} = 0. \label{eq:FractionalPartAnnulusConvergence}
\end{equation}

Since $\{u_n\}$ is a $(PS)_{\tilde{c}}$ sequence, by \eqref{eq:NonLinearityAnnulusConvergence} and combining Hölder's inequality with \eqref{eq:FractionalPartAnnulusConvergence}, we can deduce
\begin{align*}
&o_n(1) = \left\langle \J'(u_n), \eta u_n \right\rangle \\
&= \int_{\mathbb{R}^N} |\nabla u_n|^{p-2} \nabla u_n \nabla (\eta u_n) + \left\langle u_n, \eta u_n \right\rangle_{s,q} - \int_{\mathbb{R}^N} |u_n|^{p^*-2} u_n (\eta u_n) \\
&= \int_{\mathbb{R}^N} |\nabla u_n|^p \eta + \int_{\mathbb{R}^N} u_n |\nabla u_n|^{p-2} \nabla u_n \cdot \nabla \eta  + \iint_{\mathbb{R}^{2N}} |u_n(x) - u_n(y)|^q\eta(y) \,\frac{dy \, dx}{|x-y|^{N + qs}}    \\
& \quad + \iint_{\mathbb{R}^{2N}} |u_n(x) - u_n(y)|^{q-2}(u_n(x)-u_n(y))u_n(x)(\eta(x)-\eta(y))\, \frac{dy \, dx}{|x-y|^{N + qs}}  + o_n(1)\\
&= \int_{\mathbb{R}^N} |\nabla u_n|^p \eta + \iint_{\mathbb{R}^{2N}} |u_n(x) - u_n(y)|^q \eta(y)\, \frac{dy \, dx}{|x-y|^{N + qs}} + o_n(1).
\end{align*}
Thus,
\begin{equation*} 
\lim_{n\to\infty} \int_{B_{c',d'}} |\nabla u_n|^p \,dx = 0,
\quad \iint_{B_{c^{\prime},d^{\prime}} \times \rnn} |u_n(x) - u_n(y)|^q \, \frac{dy \, dx}{|x-y|^{N + qs}} = 0. \qedhere
\end{equation*}
\end{proof}

For any $\delta > 0$, we introduce the quantities
\begin{align}
\kappa_1 := \lim_{n \to \infty} \int_{B_{\delta}} |\nabla u_n|^p \, dx, \quad
\kappa_2 := \lim_{n \to \infty} \iint_{B_{\delta} \times B_{\delta}} \frac{|u_n(x) - u_n(y)|^q}{|x-y|^{N+qs}} \, dy\, dx, \notag\\
\kappa_3 := \lim_{n \to \infty} \int_{B_{\delta}} |u_n|^{p^*} \, dx.\label{eq:defining kappas}
\end{align}

\begin{lemma}\label{lem:k3_estimate}
Let $\{u_n\}$ be as in Lemma~\eqref{lem:annulus_zero}. Then, for every $\delta > 0$, 
\[\text{either \quad }
\kappa_3 = 0 
\quad \text{or} \quad 
\kappa_3 \geq \max\left\{ \mathcal{S}^{\frac{p^*}{p^* - p}}, \, \mathcal{S}_f^{\frac{q_s^*}{q_s^* - q}} \right\}.
\]
\end{lemma}
\begin{proof}
Let $\phi \in$ $C_{0,r}^\infty(\mathbb{R}^N)$ be such that $\phi \equiv 1$ on $B_\delta$ and $\operatorname{supp}(\phi) \subset B_{\delta_1}$. By Lemma~\ref{lem:annulus_zero}, we have

\begin{equation}\label{eq:integral_est_1}
\int_{\mathbb{R}^N} |\nabla u_n|^{p-2} \nabla u_n \cdot \nabla (\phi u_n) \,dx
= \int_{B_\delta} |\nabla u_n|^p \,dx + o_n(1),
\end{equation}
\begin{equation}\label{eq:integral_est_2}
\int_{\mathbb{R}^N} |u_n|^{p^*} \phi \,dx
= \int_{B_\delta} |u_n|^{p^*} \,dx + o_n(1).
\end{equation}
Next, using again Lemma~\ref{lem:annulus_zero} and H\"{o}lder's Inequality, we decompose the nonlocal term as
\begin{align*}
I &:= \iint_{\mathbb{R}^{2N}} |u_n(x) - u_n(y)|^{q-2} (u_n(x) - u_n(y)) (\phi(x) u_n(x) - \phi(y) u_n(y)) \,\frac{dy \, dx}{|x-y|^{N + qs}} \\
&= \iint_{B_\delta \times B_\delta} + \iint_{B_\delta \times B_{\delta_1}^c} + \iint_{B_{\delta_1}^c \times B_\delta} + \iint_{B_{\delta, \delta_1} \times \mathbb{R}^N} + \iint_{(\mathbb{R}^N \setminus B_{\delta,\delta_1}) \times B_{\delta, \delta_1}} \\
&= \iint_{B_\delta \times B_\delta} + \iint_{B_\delta \times B_{\delta_1}^c} + \iint_{B_{\delta_1}^c \times B_\delta} +  o_n(1) := I_1 + I_2 + I_3 + o_n(1).
\end{align*}
By H\"{o}lder's inequality, we estimate
\begin{align*}
I_3 &= \iint_{B_{\delta_1}^c \times B_{\delta}} |u_n(x)-u_n(y)|^{q-2} (u_n(x)-u_n(y)) u_n(x)\phi(x) \,\frac{dy \, dx}{|x-y|^{N + qs}} \\
&\leq \left( \iint_{B_{\delta_1}^c \times B_{\delta}} |u_n(x)-u_n(y)|^q \,\frac{dy \, dx}{|x-y|^{N + qs}} \right)^{(q-1)/q} \left( \iint_{B_{\delta_1}^c \times B_{\delta}} |u_n(x)|^q\, \frac{dy \, dx}{|x-y|^{N + qs}} \right)^{1/q} \\
&\leq C [u_n]_{s,q}^{q-1} \left( \int_{B_{\delta}} |u_n|^q \,dx \right)^{1/q}.
\end{align*}
Together with $u_n \rightharpoonup 0$ in $\X^r(\mathbb{R}^N)$ and \eqref{ineq:main_embedding}, we get the boundedness of the term $[u_n]_{s,q}$  then it follows from the compact embedding that
$
u_n \to 0 \;\; \text{in } L^q(B_{\delta_1}),
$
and hence $I_3 \to 0$ as $n \to \infty$. The same argument yields $I_2 \to 0$.

Consequently,
\begin{equation} \label{eq:integral_est_3}
I = \iint_{B_\delta \times B_\delta} |u_n(x) - u_n(y)|^q \,\frac{dy \, dx}{|x-y|^{N + qs}} + o_n(1)
\end{equation}
Combining \eqref{eq:integral_est_1}, \eqref{eq:integral_est_2}, and \eqref{eq:integral_est_3}, together with the fact that $
\langle \mathcal{J}'(u_n), \phi u_n \rangle = o_n(1),
$
we obtain
\begin{equation*}
\kappa_1 + \kappa_2 - \kappa_3 = 0.
\end{equation*}
Applying the Sobolev inequalities \eqref{eq:best_constant_local} and \eqref{eq:best_constant_nonlocal} to $\phi u_n$ and using the identity above, we further deduce
\begin{equation*}
\kappa_3^{p/p*} \leq \mathcal{S}^{-1} \kappa_1 \leq \mathcal{S}^{-1} \kappa_3 \quad \text{and} \quad 
\kappa_3^{q/q_s^*} \leq \mathcal{S}_f^{-1} \kappa_2 \leq \mathcal{S}_f^{-1} \kappa_3.
\end{equation*}
This implies that
\begin{equation*}
\text{either } \kappa_3=0 \quad \text{ or } \quad \kappa_3 \geq \max \left\{\mathcal{S}^{p^*/(p^*-p)}, \mathcal{S}_f^{q_s^*/(q_s^* - q)}\right\}.
\end{equation*}
\end{proof}
\begin{lemma}\label{lem:rescaling}
There exist a $\xi_1 \in \left(0, \frac{1}{2} \max\{ \mathcal{S}^{p^*/(p^* - p)}, \mathcal{S}_f^{q_s^*/(q_s^* - 2)} \} \right)$ and a nonnegative sequence $ \{r_n\}$ such that if
\begin{equation}
    \tilde{u}_n(x) = r_n^{(N-p)/p} u_n(r_n x). \label{rescaledPSsequence}
\end{equation}
Then for all $\xi \in (0,\xi_1)$ we have
\begin{equation}
    \int_{B_1} |\tilde{u}_n|^{p^*} dx = \xi  \text{ for all } n \in \ntrl. \label{concOfMass}
\end{equation}
\end{lemma}

\begin{proof}

Under the standing assumption $p^* = q_s^*$, the rescaling \eqref{rescaledPSsequence}
preserves both the $p$-Laplacian and the $q$-fractional energies. Indeed,
\[
\int_{\mathbb{R}^N} |\nabla \tilde u_n|^p \, dx
=
\int_{\mathbb{R}^N} |\nabla u_n|^p \, dx,
\qquad
[\tilde{u}_n]_{s,q}^q 
= r_n^{-N + qs + q(N-p)/p} [u_n]_{s,q}^q,
\]
and the exponent $-N + qs + q(N-p)/p = 0$ precisely when $p^* = q_s^*$. 
Consequently,
$
\mathcal{J}(\tilde{u}_n) = \mathcal{J}(u_n),
$
and for any test function $\phi$,
$
\langle \mathcal{J}'(\tilde{u}_n), \phi \rangle 
= \left\langle \mathcal{J}'(u_n), r_n^{-(N-p)/p}\,\phi(r_n^{-1}\cdot) \right\rangle \to 0.
$
Therefore, $\{\tilde{u}_n\}$ is again a $(PS)_{\tilde{c}}$ sequence at level $\tilde{c}$.

Since $\tilde c > 0$ and $\tilde c = m$ (cf. Lemma~\ref{lem:m_equals_mtpasslevel}), it follows that
$
\kappa_\infty := \lim_{n\to\infty} \int_{\mathbb{R}^N} |u_n|^{p^*} \, dx > 0.
$
Define
\[
\xi_1 := \min\left\{
\max\left\{ \mathcal S^{p^*/(p^*-p)}, \mathcal S_f^{q_s^*/(q_s^*-q)} \right\},\; \kappa_\infty
\right\}.
\]
Fix $\xi \in (0,\xi_1)$. Then, for each $n \in \mathbb{N}$, there exists $r_n>0$ such that
$
\int_{B_{r_n}} |u_n|^{p^*} \, dx = \xi.
$ A change of variables then immediately yields \eqref{concOfMass}.
\end{proof}

\begin{proof}[Proof of Theorem \eqref{thm:existence_equality}]
     Since $\{\tilde u_n\}$ is again a $(PS)_{\tilde c}$ sequence, we have
\begin{equation*}
  \mathcal{J}\left(\tilde{u}_n\right)-\frac{1}{p^*}\left\langle  \mathcal{J}^{\prime}\left(\tilde{u}_n\right), \tilde{u}_n\right\rangle \geqslant\left(\frac1p-\frac1{p^*}\right)
\int_{\mathbb R^N}|\nabla \tilde u_n|^p\,dx. .
\end{equation*}
It follows that $\{\tilde u_n\}$ is bounded in $\X^r(\mathbb{R}^N)$. 
Hence, up to a subsequence, there exists $\tilde u \in \X^r(\mathbb{R}^N)$ such that
\begin{equation*}
\left\{\begin{array}{l}
\tilde{u}_n \rightharpoonup \tilde{u} \text { in } \X^r\left(\mathbb{R}^N\right) ; \\
\tilde{u}_n \rightharpoonup \tilde{u} \text { in } L^{p^*}\left(\mathbb{R}^N \right); \\
\tilde{u}_n(x) \rightarrow \tilde{u}(x) \text { a.e. on } \mathbb{R}^N .
\end{array}\right.
\end{equation*}

The pointwise convergence of the gradients of $u_n$ follows in the same way as in \cite[Lemma~2.2]{silva2024mixed} and \cite[Lemma~3.3]{wang2023ground}, provided that we estimate the error coming from the nonlocal term, which is given by
\begin{equation*}
J = \langle u_n,T_k(u_n -u) \eta\rangle_{s,q} - \langle u,T_k(u_n -u) \eta\rangle_{s,q},
\end{equation*}
where $T_k(t) = \begin{cases}
    t & \text{ if } |t| \leq k \\
    \frac{t}{|t|} k & \text{ if } |t| > k
\end{cases}$, and $\eta\in C_{0,r}^{\infty}(D)$ is a cutoff function with $D = B_d \text{ or } B_{a,b}$ for some $d > 0$ and $0 < a < b$. Then 
\begin{equation*}
\begin{aligned}
J & = \! \!\iint_{\rtwon}  \frac{(J_{q}(u_n(x) - u_n(y)) - J_q(u(x) - u(y)))(T_k(u_n-u)(x) - T_k(u_n-u)(y) ) \eta(x)}{|x-y|^{N+qs}}  dx dy \\
& \quad + \iint_{\rtwon} \frac{(J_{q}(u_n(x) - u_n(y)) - J_q(u(x) - u(y)))T_k(u_n-u)(y) (\eta(x) -  \eta(y))}{|x-y|^{N+qs}}  dx  dy\\
& = J_1 + J_2.
\end{aligned}
\end{equation*}
Then, clearly, $J_1 \geq 0 $ from $(2.8)$ of \cite{silva2024mixed} and $J_2$ can be estimated as 
\begin{equation*}
    J_2 \leq k 2^{q(q-1)^2} ( [u_n]_{s,q}^{q-1} + [u]_{s,q}^{q-1}) [\eta]_{s,q} \leq k C,
\end{equation*}
where $C > 0$ is some constant.

Therefore, up to a subsequence, we then have
$
\nabla \tilde u_n(x) \to \nabla \tilde u(x) \quad \text{a.e. in } \mathbb{R}^N.
$
Consequently, passing to the limit in the weak formulation and using the $(PS)_{\tilde c}$ condition, we deduce that $\tilde u$ is a critical point of $\mathcal J$, and in particular
$
\mathcal J(\tilde u) \ge 0.
$\\
Set $v_n := \tilde u_n - \tilde u$. Then $\{v_n\}$ is bounded in $\X^r(\mathbb{R}^N)$. Assume
\begin{equation*}
A\left(v_n\right):=\int_{\rnn} |\grad v_n|^p dx \rightarrow A_{\infty}, \quad B\left(v_n\right):=[v_n]_{s,q}^q \rightarrow B_{\infty}, \quad C\left(v_n\right):=\int_{\rnn} |v_n|^{p^*} dx \rightarrow C_{\infty} .
\end{equation*}
for some $A_\infty, B_\infty, C_\infty \ge 0$.
By Brezis--Lieb, we obtain
\begin{equation*}
\begin{aligned}
&   \mathcal{J}\left(v_n\right) \rightarrow \frac{1}{p} A_{\infty}+\frac{1}{q} B_{\infty}-\frac{1}{p^*} C_{\infty}=\tilde{c}-  \mathcal{J}(\tilde{u}),\\
&\left\langle  \mathcal{J}^{\prime}\left(v_n\right), v_n\right\rangle \rightarrow A_{\infty}+B_{\infty}-C_{\infty}=0.
\end{aligned}
\end{equation*}

If $A_{\infty} = 0$, 
$\tilde{u}$ is a ground state solution of \eqref{eq:main_problem}.

Assume that $A_\infty > 0$ and $\tilde u = 0$. Since $\{\tilde u_n\}$ is a $(PS)_{\tilde c}$ sequence, Lemma~\ref{lem:k3_estimate} applies to $\{\tilde u_n\}$ and yields the following
 \[\text{ either} \quad \lim_{n\to\infty} \int_{B_1} |\tilde u_n|^{p^*}\,dx = 0,
\quad \text{or} \quad
\lim_{n\to\infty} \int_{B_1} |\tilde u_n|^{p^*}\,dx 
\ge \max\left\{ \mathcal S^{p^*/(p^*-p)}, \mathcal S_f^{q_s^*/(q_s^*-q)} \right\}.
\]
Both alternatives contradict \eqref{concOfMass}. Hence, $\tilde u \not\equiv 0$.\\
If $\mathcal J(\tilde u)=\tilde c$, then $\tilde u$ is a ground state solution, since $m=\tilde c$.
Otherwise,
$
  \mathcal{J}(\tilde{u})>m=\tilde{c}.
$
On the other hand, using
\begin{equation*}
  \mathcal{J}\left(v_n\right)-\frac{1}{p^*}\left\langle  \mathcal{J}^{\prime}\left(v_n\right), v_n\right\rangle \geqslant\left(\frac{1}{p}-\frac{1}{p^*}\right) A\left(v_n\right) \geqslant 0 \implies  \mathcal{J}(\tilde{u}) \leqslant \tilde{c}.
\end{equation*}
Therefore,
$
\mathcal J(\tilde u)=\tilde c,$ and $\tilde u$ is a ground state solution of \eqref{eq:main_problem}. 
Finally, we indicate why a nonnegative ground state exists; the argument is
standard, so we only sketch the chain of reasoning. Passing to $|\tilde u|$ does not
increase the energy. In particular $|\tilde u|$ need not lie on
$\mathcal{N}$; rather, projecting it onto $\mathcal{N}$ through the unique
$t_{|\tilde u|}>0$ with $t_{|\tilde u|}|\tilde u|\in\mathcal{N}$ and using that
$t\mapsto\mathcal{J}(t\tilde u)$ is maximised at $t=1$, one obtains
$m\le\mathcal{J}(t_{|\tilde u|}|\tilde u|)\le\mathcal{J}(t_{|\tilde u|}\tilde u)\le
\mathcal{J}(\tilde u)=m$. Hence $w:=t_{|\tilde u|}|\tilde u|\ge0$ is again a
minimiser on $\mathcal{N}$. Thus
\eqref{eq:main_problem} admits a nonnegative ground state. Since the Nehari manifold is a natural constraint, $w$
is a free critical point of $\mathcal{J}$, that is, a nonnegative weak solution of
\eqref{eq:main_problem} (follows from Implicit Function Theorem and \cite[Proof of Theorem 2.3.1]{serra_book}); and by the principle of symmetric criticality (see \cite{palais_principle_symmetric}) it is a
critical point on the whole space $\mathcal{X}_0(\mathbb{R}^N)$. 
\end{proof}
\begin{remark}
Combined with the strong maximum principle (Theorem~\ref{thm:strong_comparison_mixed}), the
nonnegative ground state is in fact
strictly positive.
\end{remark}


\section{Comparison Principles}\label{sec: comparison}
In this section, we establish a weak comparison and a strong maximum principle for the operator $\mathcal{L}_{p,q}=-\Delta_p + (-\Delta)_q^s$. These results will be used in Section~\ref{sec:decayEstimates} to derive lower bounds for solutions of \eqref{eq:main_problem}, by comparing them with suitable subsolutions. All the results in this section are valid for arbitrary $p,q>1, p<N, sq<N$, with the exception of Theorem~\ref{thm:comparison_principle_mixed}.

Now, to formulate the comparison principle, we first introduce the appropriate functional setting for admissible functions.
For any open set $\Omega \subset \rnn$, we recall
\begin{equation*}
\begin{aligned}
\mathcal{D}^{s,p,q}(\Omega):=\Bigl\{&u\in L^{q-1}_{\mathrm{loc}}(\mathbb{R}^N)\cap\mathcal{D}^{1,p}(\Omega)\ :\ \\
&\exists\,E\supset\Omega,\ E^c\text{ compact},\ 
\mathrm{dist}(E^c,\Omega)>0,\ [u]_{s,q,E}<\infty\Bigr\}.
\end{aligned}
\end{equation*}

The next lemma ensures the validity of the operator $\Lo$ for functions $u \in \mathcal{D}^{s,p,q}(\om)$.

\begin{lemma}\label{dualityLemma}
For every $u \in \mathcal{D}^{s,p,q}(\Omega)$, 
$$\mathcal{D}_0^{1,p}(\Omega)\cap \mathcal{D}_0^{s,q}(\Omega) \ni \varphi \mapsto \langle\!\langle \mathcal{L}_{p,q}(u),\varphi\rangle\!\rangle$$ is well-defined and
$
\mathcal{L}_{p,q}(u)\in
\left(\mathcal{D}_0^{1,p}(\Omega)\cap \mathcal{D}_0^{s,q}(\Omega)\right)^*,
$
where $^*$ denotes the dual space.
\end{lemma}
\begin{proof}
The result follows from Proposition~2.5 in \cite{Brasco_DecayEstimate}, together with the fact that the mapping
$
\phi \mapsto \int_{\Omega} |\nabla u|^{p-2}\nabla u \cdot \nabla \phi \, dx
$
is well-defined and defines a continuous linear functional on $\mathcal{D}_{0}^{1,p}(\Omega) \cap \mathcal{D}_{0}^{s,q}(\Omega)$.
\end{proof}
\begin{proposition}[Simon-type inequality, cf. \cite{simon_inequalities}]
Let $1<p<\infty$ and let $\langle \cdot,\cdot\rangle$ denote the Euclidean inner product in $\mathbb{R}^N$, and $|\cdot|$ denotes the Euclidean norm. Then there exists a constant $C=C(p)>0$ such that for all $t_1,t_2\in\mathbb{R}^N$,
\begin{equation}\label{simon_inequality}
\big\langle |t_1|^{p-2}t_1 - |t_2|^{p-2}t_2,\; t_1 - t_2 \big\rangle
\ge
\begin{cases}
C\,|t_1-t_2|^{p}, & \text{if } p\ge 2,\\[6pt]
C\,\dfrac{|t_1-t_2|^{2}}{\big(|t_1|+|t_2|\big)^{\,2-p}}, & \text{if } 1<p\le 2.
\end{cases}
\end{equation}
\end{proposition}
\begin{theorem}[Weak Comparison principle]\label{thm:comparison_principle_mixed}
Let $\Omega \subset \mathbb{R}^N$ be an open domain, and let $u,v \in \mathcal{D}^{s,p,q}(\Omega)$ such that $\Lo (u) \leq \Lo(v)$ weakly, i.e.,
\begin{equation*}
\langle \! \langle \Lo(u), \phi \rangle \! \rangle  \leq \langle \! \langle \Lo(v), \phi \rangle \! \rangle    
\end{equation*}
for every nonnegative test function
$
\phi \in \mathcal{D}_{0}^{1,p}(\Omega) \cap \mathcal{D}_{0}^{s,q}(\Omega).
$
If
$
u \le v  \text{ a.e. in } \Omega^c,
$
then
$
u \le v \text{ a.e. in } \Omega.
$
\end{theorem}
\begin{proof}
Set $
w:=(u-v)^+.
$ Since $[u]_{1,p,\om},[v]_{1,p,\om} < \infty$ and $w=0$ a.e. in $\mathbb R^N\setminus\Omega$. This gives $w \in \mathcal{D}_{0}^{1,p}(\om)$. Moreover, $w \in \mathcal{D}_{0}^{s,q}(\om)$ by Theorem $2.7$ of \cite{Brasco_DecayEstimate}. Thus $w \in \mathcal{D}_{0}^{1,p}(\om) \cap \mathcal{D}_{0}^{s,q}(\om)$ is an admissible nonnegative test function. Testing with $w$, we obtain
\begin{align*}
0
&\geq\langle \! \langle \Lo(u) - \Lo(v), w \rangle \! \rangle \\
&=
\int_\Omega
\left(|\nabla u|^{p-2}\nabla u-|\nabla v|^{p-2}\nabla v\right)
\cdot \nabla w \;dx \\
& \qquad +
\iint_{\mathbb R^{2N}}
\frac{
\left(J_q(u(x)-u(y))-J_q(v(x)-v(y))\right)(w(x)-w(y))
}{|x-y|^{N+qs}}\,dy\,dx .
\end{align*}

On the set $\{u>v\}$, we have by the monotonicity inequality \eqref{simon_inequality}, the local term is nonnegative. Moreover, by Lemma~9 of \cite{lindgren_lindqvist_fractional_eigenvalues}, the nonlocal term is also nonnegative. Therefore both terms must vanish. Since $w = 0$ a.e.\ in $\mathbb{R}^N \setminus \Omega$, it follows that $u \leq v$ a.e.\ in $\Omega$.
\end{proof}

We now turn to the strong maximum principle for the operator.
As a preliminary step, we require a logarithmic-type estimate (see Lemma~\ref{lem:logarithmic_lemma}). We begin by recalling the following elementary inequality from \cite{CastroTuomoPalatucci}.
\begin{lemma}\label{lem:aux_lemma_log}
Let $p \geq 1$ and $\e \in (0,1]$. Then for $a, b \in \rnn $ the following inequality holds
\begin{equation*}
    |a|^p \leq |b|^p + C \e |b|^p + (1+ C \e) \e^{1-p} |a - b|^p,
\end{equation*}
where $C= C(p)$ is some non-negative constant.
\end{lemma}

\begin{lemma}\label{lem:logarithmic_lemma}
Let $\Omega \subset \mathbb{R}^N$ be an open domain, and 
 $u \in \mathcal{D}^{s,p,q}(\Omega)$ be a nonnegative weak supersolution to $\mathcal{L}_{p,q}$ in $\Omega$ (see Definiton \ref{def:sub-super}).
Let $d>0$ and let $B_R \subset \Omega$ be such that $B_{2r} \subset B_R$. Then there exists a constant $C(N,s,p,q)>0$ such that
\begin{equation*}
\begin{aligned}
&\int_{B_r} |\nabla \log(u+d)|^p \, dx 
+ \iint_{B_r \times B_r} 
\left| \log \left( \frac{u(x)+d}{u(y)+d} \right) \right|^q 
\frac{(u(y)+d)^{q-p}}{|x-y|^{N+qs}} \, dy \, dx \\
&\quad \le 
C r^{N-p} 
+ C r^{-qs} \int_{B_{2r}} (u+d)^{q-p} \, dx 
+ C \iint_{(\mathbb{R}^N \setminus B_{2r}) \times B_{3r/2}} 
\frac{(u(x)+d)^{q-p}}{|x-y|^{N+qs}} \, dy \, dx.
\end{aligned}
\end{equation*}
\end{lemma}
\begin{proof}
Taking $\eta := (u +d)^{1-p} \phi^q \in \mathcal{D}_0^{1,p}(\Omega)\cap \mathcal{D}_0^{s,q}(\Omega)$  as the test function with $\phi \equiv 1$ in $B_r$, $|D \phi | \leq \frac{C}{r}$, and $Supp(\phi) \subset B_{3r/2}$, we obtain
\begin{align*}
    0 \leq \int_{B_{\frac{3r}{2}}} | \nabla u |^{p-2} \, \nabla u \cdot \nabla \eta + \iint_{\mathbb{R}^{2N}} \frac{|u(x)-u(y)|^{q-2}(u(x)-u(y))(\eta(x)-\eta(y))}{|x-y|^{N+qs}} := J_1 + J_2.
\end{align*}
Before proceeding further, we verify the admissibility of the test function $\eta =(u+d)^{1-p}\phi^q.$ Since $(u+d)^{-p}\le d^{-p}$, it follows easily that $\eta\in \mathcal \mathcal{D}^{1,p}_0(\Omega).$

We next show that $\eta\in \mathcal D^{s,q}_0(\Omega)$. Recall that $u\in \mathcal D^{s,p,q}(\Omega)$ implies the existence of a set $E\supset\Omega$ such that $E^c$ is compact, $dist(E^c,\Omega)>0,$
and
$[u]_{s,q,E}<\infty.$ Since $\phi\in C_c^\infty(\Omega)$, 
it suffices to prove that $[\eta]_{s,q,E}<\infty.$
\begin{align*}
[\eta]_{s,q,E}^q
&\le
\iint_{E\times E}
\frac{\left|\bigl((u(x)+d)^{1-p}-(u(y)+d)^{1-p}\bigr)\phi(x)^q
\right|^q}{|x-y|^{N+sq}}\,dx\,dy\\
&\quad
+\iint_{E\times E}
\frac{\left|(u(y)+d)^{1-p}\bigl(\phi(x)^q-\phi(y)^q\bigr)
\right|^q}{|x-y|^{N+sq}}\,dx\,dy =: I_A+I_B.
\end{align*}

By the mean value theorem, it follows that
$
I_A
\le C(p,q) d^{-pq}
\iint_{E\times E}
\frac{|u(x)-u(y)|^q}{|x-y|^{N+sq}}
\,dx\,dy
<\infty,$

For the second term, 
$
I_B
\le d^{(1-p)q}
\iint_{E\times E}\frac{|\phi(x)^q-\phi(y)^q|^q}{|x-y|^{N+sq}}
\,dx\,dy < \infty.
$
Since $\phi\in C_c^\infty(\Omega)$, 
$[\phi^q]_{s,q}<\infty.$
Combining the above estimates, we conclude that
$\eta\in \mathcal D_0^{s,q}(\Omega).$\\
We now estimate $J_1$
\begin{equation*}
\begin{aligned}
J_1 &
     \leq (1-p) \int_{B_{\frac{3r}{2}}}
    |\nabla \log ( u +d)|^{p}  \varphi^{q} +  q \int_{B_{\frac{3r}{2}}}
    |\nabla \log (u +d) |^{p-1}  |\nabla \varphi| \, \varphi^{\, q-1} \\
    & \leq \frac{(1-p)}{2} \int_{B_{\frac{3r}{2}}}
    |\nabla \log ( u +d)|^{p}  \varphi^{q}  + C(p,q) \int_{B_{\frac{3r}{2}}}
    |\nabla \varphi|^{p}\varphi^{ q - p},
\end{aligned}
\end{equation*}
where the last inequality follows from Young's inequality.
For the non-local term, we break it as follows
\begin{equation*}
\begin{aligned}
J_2 = 
\iint_{B_{2r} \times B_{2r}} 
\;+\;
2 \iint_{(\rnn \setminus B_{2r}) \times B_{2r}} 
\;:=\;
I_{1} + I_{2}.
\end{aligned}
\end{equation*}

For estimating $I_{1}$, we first consider $u(x) > u(y)$. Using Lemma \ref{lem:aux_lemma_log}, we get

\begin{equation*}
\begin{aligned}
&(u(x) - u(y))^{q-1}
\left[ (u(x)+d)^{1-p} \, \varphi^{q}(x) -(u(y)+d)^{1-p} \, \varphi^{q}(y)\right]\\
&\leq
(u(x) - u(y))^{q-1} (u(x) + d)^{1-p}
\Big[
    \varphi^q(y)
    + \varepsilon C(q) \varphi^{q}(y)
    - \frac{(u(y)+d)^{1-p}}{(u(x)+d)^{1-p}} \varphi^{q}(y)
\Big]
\\
&\quad
+ (u(x) - u(y))^{q-1} (u(x)+d)^{1-p}
(1 + C(q)\varepsilon) \, \varepsilon^{1-q}
|\varphi(x) - \varphi(y)|^{q}
:= T_{1} + T_{2}.
\end{aligned}
\end{equation*}

We now choose $\varepsilon = \delta \, \frac{u(x) - u(y)}{u(x) + d}\le 1$ where $\delta \in (0,1)$. Then
\begin{equation}\label{eq:T1-expanded}
\begin{aligned}
T_{1}
&=
(u(x)-u(y))^{q} (u(x)+d)^{-p} \varphi^{q}(y)
\left[
    C(q)\delta 
    + 
    \frac{
        1 - \left( \dfrac{u(y)+d}{u(x)+d} \right)^{1-p}
    }{
        1 - \dfrac{u(y)+d}{u(x)+d}
    }
\right].
\end{aligned}
\end{equation}

To analyse the ratio appearing in \eqref{eq:T1-expanded}, define the function
\begin{equation*}
    g(t) = \frac{1 - t^{1-p}}{1 - t}, \qquad t \in (0,1).
\end{equation*}

Since $g$ is increasing on $(0,1)$, it follows that
\begin{equation}\label{eq:g-bounds}
g(t)
\le
\begin{cases}
\dfrac{1-p}{2^{p}} \, \dfrac{t^{1-p}}{1 - t} & 0 < t \le \frac{1}{2}, \\[6pt]
1-p & \frac{1}{2} \le t \le 1.
\end{cases}
\end{equation}
To estimate $T_1$, we distinguish two cases. First, suppose that $t := \frac{u(y)+d}{u(x)+d} \in (0,\tfrac{1}{2}]$. Using the bound \eqref{eq:g-bounds} for $g(t)$, we obtain
\begin{equation*}
\begin{aligned}
T_{1}
&\le 
(u(x)-u(y))^{q} (u(x)+d)^{-p} \varphi^{q}(y)
\left[
    C(q)\delta 
    + 
    \frac{1-p}{2^{p}}
    \frac{ \left( \dfrac{u(y)+d}{u(x)+d} \right)^{1-p} }
         { \dfrac{u(x)-u(y)}{u(x)+d} }
\right]\\
&\leq
(u(x)-u(y))^{q-1}
(u(y)+d)^{1-p}
\varphi^{q}(y)\left[  C(q)\delta + \frac{1-p}{2^{p}}\right].
\end{aligned}
\end{equation*}
where the last inequality follows from the elementary bound 
$
    \frac{u(x)-u(y)}{u(x)+d}
    \left( \frac{u(y)+d}{u(x)+d} \right)^{p-1}
    \le 1.
$
Choosing 
\begin{equation*}
    \delta = \frac{p-1}{ C(q) \, 2^{p-1}} \frac{1}{2^2}
\end{equation*}
where $C(q)>1$, we obtain the final estimate
\begin{equation}\label{eq:T1-final}
T_{1}
\le 
-\varphi^{q}(y)
\frac{p-1}{2^{p+1}}
\frac{(u(x)-u(y))^{\,q-1}}{(u(y)+d)^{\,q-1}}
(u(y)+d)^{q-p}.
\end{equation}
Now if 
$
t = \frac{u(y)+d}{u(x)+d} \in \left(\tfrac{1}{2},1\right),
$
then using the bound on $g(t)$, we obtain
\begin{equation}\label{eq:T1-case2-final}
\begin{aligned}
T_{1}
&\le 
(u(x)-u(y))^{q} (u(x)+d)^{-p} \varphi^{q}(y) 
(  C(q)\delta + 1-p)\\
&\le -
\frac{(u(x)-u(y))^{q}}{(u(x)+d)^{p}} 
\varphi^{q}(y)\,
(p-1)\left(\frac{2^{p+1}-1}{2^{p+1}}\right).
\end{aligned}
\end{equation}
We note that if $2(u(y)+d) \le u(x)+d$, then
\begin{equation}\label{eq:intermediate-est1}
\begin{aligned}
\left[
    \log \left( \frac{u(x)+d}{u(y)+d} \right)
\right]^{q}\le 
C(q) \left(
    \frac{u(x)-u(y)}{u(y)+d}
\right)^{q-1}
\end{aligned}
\end{equation}
because $\lim_{t\to\infty} \frac{(\log t)^{q}}{(t-1)^{q-1}}= 0$. On the other hand, if $2(u(y)+d) \ge u(x)+d$, then, using the inequality $\log(1+\xi)\le \xi$ for all $\xi\ge0$, we deduce
\begin{equation}\label{eq:intermediate-est2}
\begin{aligned}
\Bigg[
    \log\!\left( \frac{u(x)+d}{u(y)+d} \right) 
\Bigg]^{q} (u(y) +d)^{q-p}
&\le 
    \frac{(u(x)-u(y))^q}{(u(y)+d)^p} \leq 2^p \frac{(u(x) - u(y))^q}{(u(x) + d)^p}.
\end{aligned}
\end{equation}

Substituting \eqref{eq:intermediate-est1} and \eqref{eq:intermediate-est2} in \eqref{eq:T1-final} and \eqref{eq:T1-case2-final} respectively, we get
\begin{equation}\label{eq:T1_bound}
T_1 \leq - C(p,q) \varphi^{q}(y)
\left(
    \log \left( \frac{u(x)+d}{u(y)+d} \right)
\right)^{q}
(u(y)+d)^{q-p}.
\end{equation}


We next estimate the contribution of the term from $T_2$ using the Lipschitz continuity of $\varphi$ as
\begin{equation}\label{eq:phi-diff-term}
\begin{aligned}
\iint_{B_{2r}\times B_{2r}}
\frac{(u(x)+d)^{q-p}\,
|\varphi(x) - \varphi(y)|^{q}}{|x-y|^{N+qs}}
\, dy\, dx
&\le
\frac{C}{r^q}
\iint_{B_{2r}\times B_{2r}}
\frac{(u(x)+d)^{q-p}}
{|x-y|^{\,N + q(s-1)}}
\, dy\, dx\\
&\leq  \frac{C}{r^q} \int_{B_{2r}} (u(x) +d )^{q-p} \, dx \int_{B_{4r}} \frac{dz}{|z|^{N- q(1-s)}}   \\
& \leq \frac{C(s,q,N)}{r^{qs}} \int_{B_{2r}} (u(x) +d )^{q-p} \, dx.     
\end{aligned}
\end{equation}
Noting that the estimate holds trivially when $u(x)=u(y)$, and that in the case $u(y)>u(x)$ the same computations apply after interchanging the roles of $x$ and $y$, we combine \eqref{eq:T1_bound} and \eqref{eq:phi-diff-term} to conclude the following estimate of $I_1$ as follows, where $C = C(p,q,s,N)$:
\begin{align*}
    I_1 \leq -\iint_{B_{2r}\times B_{2r}} C \varphi^{q}(y)
\left|
    \log \left( \frac{u(x)+d}{u(y)+d} \right)
\right|^{q}
\frac{(u(y)+d)^{q-p}}{|x-y|^{N+qs}}dy \;dx+\frac{C}{r^{qs}} \int_{B_{2r}} (u(x) +d )^{q-p} \, dx. 
\end{align*}
Finally, we estimate the contribution of $I_2$. Since $u\ge 0$, we have
$
\left(
\frac{u(x)-u(y)}{u(x)+d}
\right)^{q-1}
\le 1.
$
Using this, together with the fact that $\varphi$ is supported in $B_{3r/2}$, we obtain
\begin{equation*}
\begin{aligned}
I_2
&\le
C(q)
\int_{\mathbb{R}^N\setminus B_{2r}}
\int_{B_{3r/2}}
\frac{\varphi^q(x)(u(x)+d)^{q-p}}{|x-y|^{N+qs}}
\,dy\,dx .
\end{aligned}
\end{equation*}
Combining this estimate with the bound for $I_1$, we obtain the corresponding estimate for $J_2$. Finally, combining the estimates for $J_1$ and $J_2$ yields the desired logarithmic estimate.
\qedhere
\end{proof}

With the logarithmic estimate established, we are now in a position to prove the strong maximum principle.
\begin{remark}
    We note that the weak maximum principle follows readily by employing the standard argument of taking $u^- := \max\{-u,0\}$ as a test function in the weak formulation.
    \end{remark}
\begin{theorem}[Strong Maximum Principle]\label{thm:strong_comparison_mixed}
 Let $\om \subset \rnn$ be a connected open set, $1<p\leq q<\infty, p<N, sq<N$, and $u \in \mathcal{D}^{s,p,q}(\om)$ be a weak supersolution of the operator $\Lo$ with $u = 0$ in $\om^c$. Then either $u>0$ a.e. in $\om$ or $u=0$ a.e. in $\om$.
\end{theorem}

\begin{proof}
    In view of the weak maximum principle, $u \geqslant 0$ a.e. in $\mathbb{R}^N$. Thus we assume $u\not\equiv0$ a.e. in $\om$ and prove $u>0$ a.e. in $\om$. We aim at proving that for any connected compact set $K \subset \subset \om$, if $u \not \equiv 0$ in $K$ then $u>0$ a.e. in $K$. Observe that $K$ can be covered by a finite number of balls $B_{r }\left(x_1\right), \ldots, B_{r }\left(x_k\right)$ contained in $\om$ such that $x_i \in K$ and

\begin{equation}
\left|B_{r}\left(x_i\right) \cap B_{r}\left(x_{i+1}\right)\right|>0, \quad i=1, \ldots, k-1 . \label{eq:IntersectionPositive}
\end{equation}

Suppose by contradiction that
$
|\{x\in K:u(x)=0\}|>0.
$
Then there exists some $i \in\{1, \ldots, k-1\}$ such that
\[
Z:=\{x\in B_{r}(x_i):u(x)=0\}\] has positive measure.

Now $Z$ has positive Lebesgue measure and $u=0$ a.e. in $Z$. Set $v_\varepsilon := \log\left(1+\frac{u}{\varepsilon}\right)$. Then $v_\varepsilon=0$ a.e. in $Z$, and hence
$
(v_\varepsilon)_Z
=0.
$
By the Poincar\'e inequality  \cite[Theorem 13.27]{giovanni_book_sobolev}, we obtain
\[
\int_{B_{r}(x_i)}
\left|\log\left(1+\frac{u}{\varepsilon}\right)\right|^p dx
=
\int_{B_{r}(x_i)}
|v_\varepsilon-(v_\varepsilon)_Z|^p dx
\le
C (p, Z,B_{r}(x_i) )
\int_{B_{r}(x_i)}
|\nabla v_\varepsilon|^p dx.
\]

Now using Lemma \ref{lem:logarithmic_lemma}, we have 
\begin{align*}
     &\int_{ B_{r}\left(x_i\right)}\left|\log \left(1+\frac{u}{\varepsilon}\right)\right|^p d x \\
     &\leq C r^{N-p} + C r^{-qs}  \int_{B_{2r}} (u(x) +\varepsilon)^{q-p} + C \iint_{\rnn \setminus B_{2r} \times B_{3r/2}} \frac{(u(x) + \varepsilon)^{q-p}}{|x-y|^{N+qs}} dy \, dx\\
     & \leq C r^{N-p} + C r^{-qs}  \int_{B_{2r}} (u(x) +\varepsilon)^{q-p} .
\end{align*}
where $C>0$ is independent of $\varepsilon$.

If $u\not\equiv0$ in $B_{r}(x_i)$, then there exists $t>0$ such that
\[
A_t:=\{x\in B_{r}(x_i):u(x)\ge t\}
\]
has positive measure. Hence
\[
\int_{B_{r}(x_i)}
\left|\log\left(1+\frac{u}{\varepsilon}\right)\right|^p dx
\ge
|A_t|
\left|\log\left(1+\frac{t}{\varepsilon}\right)\right|^p
\to \infty
\quad\text{as }\varepsilon\to0,
\]
which contradicts the uniform bound. Therefore, if the zero set has positive measure in a ball, then $u\equiv0$ a.e. in that ball. In view of \eqref{eq:IntersectionPositive}, we can follow the same argument to get $u \equiv 0$ a.e. in K which contradicts our assumption.

For the case when $\Omega$ is bounded and connected, one can construct a sequence of connected compact sets $\{A_n\}_{n\in\mathbb N}$ satisfying
$
A_n \Subset \Omega
\text{ and }
|\Omega\setminus A_n|<\frac1n.
$ Applying the previous argument on each $A_n$ and passing to the limit as $n\to\infty$ yields the desired conclusion.

Next, suppose that $\Omega$ is connected but unbounded. Then there exists a sequence of bounded connected open sets $\{\Omega_n\}_{n\in \mathbb N }$ such that
$
\Omega_n\subset \Omega_{n+1}\subset \Omega,\;
\Omega=\bigcup_{n=1}^{\infty}\Omega_n.
$
Since $u\not\equiv0$ in $\Omega$, there exists $\tilde{n}\in\mathbb N$ for which $u\not\equiv0$ in $\Omega_n$ whenever $n\ge \tilde{n}$. Since $u$ is a non-negative weak supersolution of $\mathcal L_{p,q}$ in each $\Omega_n$, it follows from the bounded-domain case that $u>0$ a.e. in $\Omega_n$ for every $n\ge n_0$, and hence $u>0$ a.e. in $\Omega$.
This completes the proof.
\end{proof}

\section{Harnack Estimates}\label{sec: Harnack}
In this subsection, we establish Harnack-type estimates for non-negative supersolutions of the operator $\mathcal{L}_{p,q}$ when \eqref{eq:isocritical-intro} holds.
\begin{lemma}\label{lem:expansion of positivity}
Let $u \in \X(\rnn) \cap L^{\infty}(\rnn)$ be a non-negative weak supersolution of $\Lo$ (Section \ref{sec: notion of solution}). Assume $k >0$ and there exists $\tau \in (0,1]$ such that 
\begin{equation*}
| B_r(x_0) \cap \{ u \geq k \} | \geq \tau |B_r|
\end{equation*}
for $r \in (0,1]$ where $0 < r < \frac{R}{16}$ with $R>0$. Then there exists a positive constant $\delta \equiv \delta(\tau,p,q,N,s, \|u\|_{L_{loc}^\infty(\rnn)}) \in (0, \frac{1}{4})$ such that 
\begin{equation*}
    \essinf_{B_{4r}(x_0)} u \geq \delta k .
\end{equation*}
\end{lemma}
\begin{proof}
The proof is divided into two main steps. \\
\textbf{Step 1:}
We claim that there exists a constant $C (N,s,p,q,\lv u \rv_{L^{\infty}(\rnn)})> 0$ such that
\begin{equation}\label{eq:level_set_estimate}
\left| B_{6 r}(x_0) \cap \left\{ u \le 2\delta k- \varepsilon \right\} \right|
\le 
\frac{C}{\tau \log\!\left(\tfrac{1}{2\delta}\right)} | B_{6r}(x_0) |,
\end{equation}
for every $\delta \in (0,\tfrac{1}{4})$, and $\varepsilon > 0$.\\
To start with, let $\psi$ be a standard cutoff function such that
\begin{equation*}
\begin{aligned}
0 \le \psi \le 1 &\quad \text{ and } \quad |\nabla \psi| \le \frac{8}{r} \quad \text{in } B_{7r}(x_0), \text{ and} \qquad 
\psi \equiv 1 &\quad \text{in } B_{6 r}(x_0). 
\end{aligned}
\end{equation*}
Define the auxiliary function $w = u + \varepsilon$, and choose the test function $\varphi = w^{1-p}\, \psi^q$. Working as in Logarithmic Lemma \ref{lem:logarithmic_lemma}, we obtain
\begin{equation}\label{eq:log_estimate}
\begin{aligned}
&\int_{B_{6 r}(x_0)} \left| \nabla \log(u+\varepsilon) \right|^{p} \, dx + \iint_{B_{6 r}(x_0) \times B_{6 r}(x_0)}
\left| \log \left( \frac{u(x)+\varepsilon}{u(y)+\varepsilon} \right) \right|^q
\frac{(u(y)+\varepsilon)^{q-p}}{|x-y|^{N + sq}}
\, dy \, dx \\
&\le C r^{N-p}
+ C r^{-qs}
\int_{B_{8 r}(x_0)} (u+\varepsilon)^{q-p} \, dx 
+ C \iint_{ \mathbb{R}^N \setminus B_{8 r}(x_0) \times B_{7 r}(x_0)}
\frac{(u(x)+\varepsilon)^{q-p}}{|x-y|^{N + sq}}
\, dy \, dx.
\end{aligned}
\end{equation}

Next, define the truncated logarithmic function for $\delta \in (0,\frac{1}{4})$
\begin{equation*}
v
=
\left[
\min \left\{
\log\!\left(\frac{1}{2\delta}\right),\;
\log\!\left( \frac{k+\varepsilon}{u+\varepsilon} \right)
\right\}
\right]_+.
\end{equation*}
Then using the Poincar\'{e} inequality in local form \cite[Theorem $2$ of Chapter $5$, Section 5.8]{evans_pde_book} and logarithmic estimate \eqref{eq:log_estimate}, we obtain
\begin{equation}\label{eq:v_energy_estimate}
\begin{aligned}
&\int_{B_{6r}(x_0)} |v - (v)_{B_{6r}(x_0)}| \, dx
\le 
C(N,p) r|B_{6r}(x_0)|^{1 -\frac{1}{p}}
\left( \int_{B_{6r}(x_0)} |\nabla v|^p \, dx \right)^{\frac{1}{p}} \\
&\le 
C r|B_{6r}|^{1-\frac{1}{p}} \left( \int_{B_{6r}(x_0)} \left| \nabla \log(u+\varepsilon) \right|^p \, dx \right)^{\frac{1}{p}} \\
&\le C(N,s,p,q) r |B_{6r}|^{1 - \frac{1}{p}} \left[ r^{N-p}
+  r^{-qs}
\int_{B_{8r}(x_0)} (u(x)+\varepsilon)^{q-p} \, dx \right. \\
& \qquad \qquad \qquad \qquad \qquad \qquad  \left. +  \iint_{\mathbb{R}^N \setminus B_{8r}(x_0) \times B_{7 r}(x_0)}
\frac{(u(x)+\varepsilon)^{q-p}}{|x-y|^{N + sq}}
\, dy \, dx\right]^{\frac{1}{p}}.
\end{aligned}
\end{equation}
By assumption, we have the measure estimate
\begin{equation*}
\left| B_{6 r}(x_0) \cap \{ v = 0 \} \right|
\ge 
\frac{\tau}{6^N} |B_{6 r}(x_0)|.
\end{equation*}
Using this and the definition of $v$, we obtain
\begin{equation*}
\begin{aligned}
\log\!\left(\frac{1}{2\delta}\right)
& =
\frac{1}{|B_{6r}(x_0) \cap \{v=0\}|}
\int_{B_{6r}(x_0) \cap \{v=0\}}
\big( \log\!\left(\tfrac{1}{2\delta}\right) - v(x) \big)\, dx \\
& \le 
\frac{6^N}{\tau}
\frac{1}{|B_{6r}(x_0)|}
\int_{B_{6r}(x_0)}
\big( \log\!\left(\tfrac{1}{2\delta}\right) - v(x) \big)\, dx  \le 
\frac{6^N}{\tau}
\left(
\log\!\left(\frac{1}{2\delta}\right) - (v)_{B_{6r}(x_0)}
\right).
\end{aligned}
\end{equation*}
Integrating the previous estimate, we obtain
\begin{equation}\label{eq:level_measure_intermediate}
\left|  B_{6r}(x_0) \cap  \left\{v = \log\!\left(\tfrac{1}{2\delta}\right) \right\}  \right|
\, \log\!\left(\tfrac{1}{2\delta}\right)
\le 
\frac{6^N}{\tau}
\int_{B_{6r}(x_0)} |v - (v)_{B_{6r}(x_0)}| \, dx.
\end{equation}
Using estimate \eqref{eq:level_measure_intermediate} in \eqref{eq:v_energy_estimate}, the fact $u\in L^{\infty}(\rnn)$, and the inequality $|x-y| \geq C|x_0 - y|$, for some constant $C>0$ with $x \in B_{7 r}(x_0)$ and $y \in \rnn \setminus B_{8r}(x_0)$, we deduce 
\begin{equation*}
\begin{aligned}
\left|  B_{6r}(x_0) \cap \left\{ v = \log\!\left(\tfrac{1}{2\delta}\right) \right\} \right| & \le 
\frac{C}{\tau \log\!\left(\tfrac{1}{2\delta}\right)}
\, r^{1+\frac{(p-1)N}{p}}
\Bigg[
r^{N-p}
+ r^{-qs}
\int_{B_{8r}(x_0)} (u(x)+\varepsilon)^{q-p} \, dx  \\
&\qquad \qquad +  \iint_{(\mathbb{R}^N \setminus B_{8r}(x_0))\times B_{7 r}(x_0)}
\frac{(u(x)+\varepsilon)^{q-p}}{|x-y|^{N+sq}}
\, dy \, dx
\Bigg]^{\frac{1}{p}}\\
&\leq 
\frac{C(N,s,p,q, \|u\|_{L^\infty_{loc}(\rnn)})}{\tau \log\!\left(\tfrac{1}{2\delta}\right)}
\, r^{1+\frac{(p-1)N}{p}} \left[ C r^{N-p} + C r^{N-sq}\right]^{1/p}.
\end{aligned}
\end{equation*}
Recalling the definition of $v$, we equivalently obtain
\begin{equation*}
\begin{aligned}
\left| B_{6r}(x_0) \cap \left\{ u \le 2\delta k - \varepsilon \right\} \right| &\leq \left| B_{6r}(x_0) \cap \left\{ u \le 2\delta (k+\varepsilon) - \varepsilon \right\} \right| =\left| B_{6r}(x_0) \cap \left\{ v = \log\!\left(\tfrac{1}{2\delta}\right) \right\} \right| \\
&\leq 
\frac{C}{\tau \log\!\left(\tfrac{1}{2\delta}\right)}
\, r^{1+\frac{(p-1)N}{p}} \left[  r^{N-p} + r^{N-sq} \right]^{\frac{1}{p}} \leq \frac{C}{\tau \log\!\left(\tfrac{1}{2\delta}\right)} |B_{6r}(x_0)|.\\
\end{aligned}
\end{equation*}
where $C = C(N,s,p,q,\|u\|_{L^\infty(\mathbb{R}^N)})$. Here we have crucially used that $p > sq$ and $r \in (0,1]$.

\smallskip
\textbf{Step 2:}
In this step we will establish the following: there exists $\delta^* > 0$ such that, for every $\delta \in (0,\delta^*)$,
\begin{equation*}
    \essinf_{B_{4r}(x_0)} u \geq \delta k - 2 \varepsilon.
\end{equation*}
Without loss of generality, we may assume that 
\begin{equation}\label{epsilonDeltaRelation}
    \delta k > 2 \varepsilon.
\end{equation}

Let $\rho \in [r,6r]$ and let $\varphi \in C_c^\infty(B_{\rho}(x_0))$ be a cutoff function such that $0 \leq \varphi \leq 1$ in $B_{\rho}(x_0)$. Define
$
w := (\ell - u)_+,
$
for some $\ell \in (\delta k, 2\delta k)$, where $\delta$ is as in Step~1. We take $\eta := w \varphi^q$ as a test function. Then, using the weak formulation, we get
\begin{equation}
\begin{aligned}
0 &\leq  \int_{\rnn} |\nabla u|^{p-2} \nabla u \cdot \nabla (w \varphi^q) \, dx  + \iint_{\mathbb{R}^N \times \mathbb{R}^N} 
\frac{J_q(u(x)-u(y)) \big(w\varphi^q(x) - w\varphi^q(y)\big)}{|x-y|^{N + sq}} \, dx\,dy\\
&:= I_1 + I_2.
\end{aligned}
\label{eq:weak_formulation_s2}
\end{equation}

Applying Young's inequality, we obtain
\begin{equation}\label{eq:I1_estimate_s2}
I_1
\leq -\frac{1}{2} \int_{B_{\rho}(x_0)} |\nabla w|^p \varphi^q \, dx 
+ C(p,q) \int_{B_{\rho}(x_0)} |\nabla \varphi|^p w^p \varphi^{q-p} \, dx.
\end{equation}

To estimate $I_2$, we borrow the fractional term estimate in Lemma $3.2$ of \cite{Castro_NonLocalHarnackInequalities} to get

\begin{equation}
\begin{aligned}
I_2 \leq  & - C \iint_{B_{\rho}(x_0)\times B_{\rho}(x_0)} 
\frac{|w(x)\varphi(x) - w(y)\varphi(y)|^q}{|x-y|^{N + sq}} \, dx\,dy \\
& + C \iint_{B_{\rho}(x_0)\times B_{\rho}(x_0)} 
\max\{w(x), w(y)\}^q \frac{|\varphi(x)-\varphi(y)|^q}{|x-y|^{N + sq}} \, dx\,dy \\
& +C \ell \, |B_\rho (x_0) \cap \{u < \ell \}|  \, \sup_{x \in \mathrm{supp}(\varphi)} 
\int_{\mathbb{R}^N \setminus B_\rho (x_0)} 
\frac{\ell^{\,q-1}}{|x-y|^{N + sq}} \, dy. 
\end{aligned}
\label{eq:I2_estimate_s2}
\end{equation}
Substituting  \eqref{eq:I1_estimate_s2} and \eqref{eq:I2_estimate_s2} into \eqref{eq:weak_formulation_s2}, we get, for some constant $C\equiv C(N,s,p,q)>0$,
\begin{equation}
\begin{aligned}
& \int_{B_\rho (x_0)} |\nabla w|^p \varphi^q \, dx 
+ \iint_{B_\rho (x_0)\times B_\rho (x_0)} 
\frac{|w(x)\varphi(x) - w(y)\varphi(y)|^q}{|x-y|^{N + sq}} \, dx\,dy \\
& \leq C \left[\int_{B_{\rho} (x_0)} |\nabla \varphi|^p w^p \varphi^{q-p} \, dx \right.\\
& \quad +  \iint_{B_{\rho} (x_0)\times B_{\rho} (x_0)} 
\max\{w(x), w(y)\}^q 
\frac{|\varphi(x)-\varphi(y)|^q}{|x-y|^{N + sq}} \, dx\,dy \\
& \quad + \left. \ell \, |B_\rho (x_0) \cap \{u < \ell \}|  \, \sup_{x \in \mathrm{supp}(\varphi)} 
\int_{\mathbb{R}^N \setminus B_\rho (x_0)} 
\frac{\ell^{\,q-1}}{|x-y|^{N + sq}} \, dy\right] := C\left[J_1 + J_2 + J_3\right].
\end{aligned}
\label{eq:key_energy_split}
\end{equation}
Now, we are ready to setup the iteration process. For each $j = 0,1,2,\dots$, we define the following sequences.
\begin{align*}
\ell = k_j & := \delta k + 2^{-j-1}\,\delta k \in [\delta k,\, 2\delta k], \qquad \rho= \rho_j := 4r + 2^{1-j} r \in (4r,\, 6r), \\
\hat{\rho}_j &:= \frac{\rho_j + \rho_{j+1}}{2} \in (4r,\, 6r), \qquad B_j := B_{\rho_j}(x_0), \qquad  \hat{B}_j := B_{\hat{\rho}_j}(x_0),\\
&\hspace{3cm}w_j := (k_j - u)_+.
\end{align*}
Since $k_j \leq 2\delta k$, the difference between successive levels satisfies
\begin{equation}
k_j - k_{j+1} = 2^{-j-2}\,\delta k \geq 2^{-j-3} k_j.
\label{eq:kj_gap_repeat}
\end{equation}
Also, on the set $\{u < k_{j+1}\}$, we estimate
\begin{equation}
\begin{aligned}
w_j = (k_j - u)_+ 
\geq (k_j - u)\,\chi_{\{u < k_{j+1}\}} \geq (k_j - k_{j+1})\,\chi_{\{u < k_{j+1}\}} \geq 2^{-j-3} k_j \,\chi_{\{u < k_{j+1}\}}.
\end{aligned}
\label{eq:wj_lower_repeat}
\end{equation}
Finally, let $(\psi_j)_{j \geq 0} \subset C_c^\infty(\hat{B}_j)$ be a sequence such that
\begin{equation*}
0 \leq \psi_j \leq 1 \quad \text{in } \hat{B}_j, \quad \psi_j \equiv 1 \quad \text{in } B_{j+1}, \text{ and} \quad |\nabla \psi_j| \leq \frac{2^{j+3}}{r}.
\end{equation*}
We now take $\varphi = \psi_j$ and $w = w_j$ in \eqref{eq:key_energy_split} and estimate each term on the right-hand side separately. We begin with the term $J_1$, taking into account that $p<q$.
\begin{equation}
J_1 :=  \int_{B_j} w_j^p |\nabla \psi_j|^p \psi_j^{\,q-p} \, dx \leq C(p)\, 2^{jp} k_j^p r^{-p} \, |B_j \cap \{u < k_j\}|.
\label{eq:J1_estimate_s2}
\end{equation}
Using the bounds in Lemma 3.2 of \cite{Castro_NonLocalHarnackInequalities}, we obtain
\begin{equation}
J_2 \leq C(N,q,s)\, 2^{jq} k_j^q r^{-sq} \, |B_j \cap \{u < k_j\}|,
\label{eq:J2_estimate_s2}
\end{equation}
and 
\begin{equation}
J_3 \leq C(N,q,s)\, 2^{j(N+qs)} k_j^q r^{-sq} \, |B_j \cap \{u < k_j\}|.
\label{eq:J3_estimate_s2}
\end{equation}

Combining \eqref{eq:J1_estimate_s2}, \eqref{eq:J2_estimate_s2}, and \eqref{eq:J3_estimate_s2} into \eqref{eq:key_energy_split}, we deduce
\begin{equation*}
\begin{aligned}
\int_{B_j} |\nabla w_j|^p \psi_j^q \, dx &+ \iint_{B_{j} \times B_{j}} \frac{|w_j (x) \psi_j(x) - w_j(y)\psi_j(y)|^{q}}{|x-y|^{N+qs}} \, dx \, dy \\
& \leq C \, |B_j \cap \{u < k_j\}| \left[
2^{jp} k_j^p r^{-p}
+ 2^{jq} k_j^q r^{-sq}
+ 2^{j(N+qs)} k_j^q r^{-sq} \right] \\
& \leq  C(N,s,p,q)\, |B_j \cap \{u < k_j\}|  2^{j(p+q+N+qs)}  \left( \frac{k_j^p}{r^p} + \frac{k_j^q}{r^{qs}} \right).
\end{aligned}
\end{equation*}
Let $\kappa = q_s^*/q= p^*/q$. Combining the Sobolev inequality on balls (see, e.g., \cite[Corollary 1.57]{LocalSobolevOnBall} for the local case;\;\cite[Corollary 4.10]{cozzi_harnack_fractional}, and \;\cite[Proposition 2.2]{brasco_second_eigenvalue} for the nonlocal case),\;\eqref{eq:key_energy_split}, and the preceding estimate, 
\begin{align}
&\left[\left(\frac{(k_j-k_{j+1})}{r}\right)^p+ \left(\frac{(k_j-k_{j+1})}{r^s}\right)^q \right] \left( \frac{|B_{j+1} \cap \{u < k_{j+1}\}|}{|B_{j+1}|} \right)^{\frac{1}{\kappa}} \notag \\
& \leq  \frac{1}{r^p} \left( (k_j - k_{j+1})^{p^*} 
\frac{|B_{j+1} \cap \{u < k_{j+1}\}|}{|B_{j+1}|} \right)^{\frac{p}{p^*}} + 
\frac{1}{r^{qs}}\left( (k_j - k_{j+1})^{q_s^*} 
\frac{|B_{j+1} \cap \{u < k_{j+1}\}|}{|B_{j+1}|} \right)^{\frac{q}{q_s^*}}
 \notag\\
&  \leq \frac{1}{r^p}\left( \frac{1}{|B_{j+1}|} 
\int_{B_{j+1} \cap \{u < k_{j+1}\}} (k_j - u)^{p^*}  dx \right)^{\frac{p}{p^*}}
+ \frac{1}{r^{qs}}\left( \frac{1}{|B_{j+1}|} 
\int_{B_{j+1} \cap \{u < k_{j+1}\}} (k_j - u)^{q_s^*}  dx \right)^{\frac{q}{q_s^*}} \notag\\
& \leq \frac{1}{r^p}\left( \frac{1}{|B_{j+1}|} 
\int_{B_{j+1}} w_j^{p^*} \psi_j^{q p^*/p} \, dx \right)^{\frac{p}{p^*}} +
\frac{1}{r^{qs}}\left( \frac{1}{|B_{j+1}|} 
\int_{B_{j+1}} w_j^{q_s^*} \psi_j^{q_s^*} \, dx \right)^{\frac{q}{q_s^*}} \notag\\
& \leq  \frac{C(N)}{r^p} \left( \frac{1}{|B_{j}|} 
\int_{B_{j}} w_j^{p^*} \psi_j^{q p^*/p} \, dx \right)^{\frac{p}{p^*}} +
 \frac{C(N)}{r^{qs}}\left( \frac{1}{|B_{j}|} 
\int_{\hat{B}_{j}} w_j^{q_s^*} \psi_j^{q_s^*} \, dx \right)^{\frac{q}{q_s^*}} \notag\\
& \leq  \frac{C}{|B_{j}|} \int_{B_j} |\nabla (w_j \psi_j^{q/p})|^p \, dx +
C(N,q,s) \frac{2^{j(1 + qs)}}{|B_{j}|} \iint_{B_{j} \times B_{j}} \frac{|w_j (x) \psi_j(x) - w_j(y) \psi_j(y)|^{q}}{|x-y|^{N+qs}}  dx  dy  \notag\\
&  \leq C(N,p,q,s)\, \frac{|B_j \cap \{u < k_j\}|}{|B_j|} \; 2^{j(p+q+N+2qs+1)} \left( \frac{k_j^p}{r^p} + \frac{k_j^q}{r^{qs}} \right). \label{eq:final_recursive_estimate_s2}
 \end{align}
Note that, in view of \eqref{eq:wj_lower_repeat} and the condition $p<q$, we have
\begin{equation}
\begin{aligned}
\left( \frac{k_j^p}{r^p} + \frac{k_j^q}{r^{qs}} \right)&\left[\left(\frac{(k_j-k_{j+1})}{r}\right)^p+ \left(\frac{(k_j-k_{j+1})}{r^s}\right)^q \right]^{-1} \\
&\leq \left( \frac{k_j^p}{r^p} + \frac{k_j^q}{r^{qs}} \right) \left[\frac{2^{(-j-3)p}k_j^p}{r^p} + \frac{2^{(-j-3)q}k_j^q}{r^{qs}} \right]^{-1}\\
& \leq 2^{(j+3)q}\left( \frac{k_j^p}{r^p} + \frac{k_j^q}{r^{qs}} \right) \left(\frac{k_j^p}{r^p} + \frac{k_j^q}{r^{qs}} \right)^{-1} = 2^{(j+3)q}.
\end{aligned}\label{eq:kj_bound}
\end{equation}
Define the normalized measure sequence
\begin{equation*}
Y_j := \frac{|B_j \cap \{u < k_j\}|}{|B_j|}, 
\qquad j = 0,1,2,\dots
\end{equation*}
From \eqref{eq:final_recursive_estimate_s2}, \eqref{eq:kj_gap_repeat} and \ref{eq:kj_bound}, we deduce the following recursive inequality
\begin{equation*}
\begin{aligned}
Y_{j+1}^{\frac{q}{q^*_s}} \leq C\, \, Y_j \, 2^{j(p+q(1+2s)+N+1)} 2^{(j+3)q} \leq  C(N,s,p,q)\,Y_j \, 2^{j(2q(1+s)+N+p+1)}.
\end{aligned}
\end{equation*}

Taking the power $q^*_s/q$, for each $j = 0, 1,2, \dots$, we arrive at the key iteration inequality 
\begin{equation*}
Y_{j+1} 
\leq c_0 \, b^j \, Y_j^{1+\alpha},
\end{equation*}
where
\begin{equation}
c_0 := C>0,
 \quad b := 2^{\frac{q^*_s}{q}(2q(1+s)+N+p+1)} > 1, \text{ and } \quad \alpha := \frac{q^*_s}{q} - 1 > 0.
\label{eq:contants_iteration_def}
\end{equation}
Finally, we estimate the initial quantity $Y_0$. By the definition of $Y_0$, together with \eqref{epsilonDeltaRelation} and \eqref{eq:level_set_estimate},
\begin{equation}
\begin{aligned}
Y_0 
&=  \frac{|B_{6r}(x_0) \cap \{u < \tfrac{3}{2}\delta k\}|}{|B_{6r}|}\leq \frac{|B_{6r}(x_0) \cap \{u < 2\delta k - \varepsilon\}|}{|B_{6r}(x_0)|} \leq \frac{C}{\tau \log\!\left(\tfrac{1}{2\delta}\right)},
\end{aligned}
\label{eq:Y0_def}
\end{equation}
for all $\delta \in (0,\frac{1}{4})$.

To apply Lemma \ref{lem:iteration_decay}, we need  $Y_0 \leq c_0^{-\frac{1}{\alpha}}\, b^{-\frac{1}{\alpha^2}}$. Using the expressions of $c_0$ and $b$ from \eqref{eq:contants_iteration_def}, this requirement becomes
\begin{equation}
\begin{aligned}
Y_0 \leq {} &
C^{-\frac{1}{\alpha}} 
\left( 2^{\frac{q^*_s}{q}(2q(1+s) + N + p +1)} \right)^{-\frac{1}{\alpha^2}}.
\end{aligned}
\label{eq:Y0_target_expanded}
\end{equation}
Combining \eqref{eq:Y0_def} and \eqref{eq:Y0_target_expanded}, it suffices to require
\begin{equation*}
\begin{aligned}
\frac{C}{\tau \log\!\left(\tfrac{1}{2\delta}\right)} 
\leq {} &
C^{-\frac{1}{\alpha}} 
\left( 2^{\frac{q^*_s}{q}(2q(1+s) + N + p +1)} \right)^{-\frac{1}{\alpha^2}}.
\end{aligned}
\end{equation*}
This clearly holds for all $\delta < \frac{1}{2} \exp\left( -\frac{C(N,s,p,q)}{\tau} \left( 2^{\frac{q^*_s}{q}(2q(1+s) + N + p +1)} \right)^{\frac{1}{\alpha^2}} \right)$.
Thus, by Lemma \ref{lem:iteration_decay}, $Y_j \to 0$ as $j \to \infty$.
\end{proof}

The following result, and then Proposition \eqref{Weak Harnack inequality} can be proved analogously to Lemma~6.7 and Proposition~6.8 of \cite{cozzi_harnack_fractional}.

\begin{lemma}
Let $
u \in \X(\mathbb{R}^{N})
    \cap L^{\infty}(\mathbb{R}^{N})$
be a non-negative weak supersolution of $\mathcal{L}_{p,q}$.
Assume that for some $R\in(0,1]$, $\tau\in(0,1)$, $t>0$, and
$k\in\mathbb{N}$,
\[
|B_R\cap\{u\ge t\}|
\ge \tau^{k}|B_R|.
\]
Then there exists a constant
$
\delta
=
\delta\bigl(
N,p,q,s,\tau,
\|u\|_{L^\infty_{\mathrm{loc}}(\mathbb{R}^{N})}
\bigr)
\in \left(0,\frac18\right)
$
such that
\[
u \ge \delta^{k} t
\qquad \text{in } B_R.
\]
\end{lemma}


\section{Decay Estimates}\label{sec:decayEstimates}
In this subsection, we establish the precise asymptotic behavior at infinity of positive radial solutions to \eqref{eq:main_problem}, i.e., we give the proof of Lemma \ref{prop:decay-main}. Throughout this subsection, we work in the isocritical regime \eqref{eq:isocritical-intro}, namely
$
p^*=q_s^*.
$
Combining the Harnack-type estimates proved in the previous subsection with suitable comparison arguments, we derive the optimal decay rate at infinity for positive radially symmetric decreasing solutions. In particular, these solutions exhibit the same asymptotic decay as the fundamental solution of the \(p\)-Laplacian.

\noindent
The upper bound \textup{(i)} of Theorem \ref{prop:decay-main} is proved in Section~\ref{sec:upper-bound},
and the lower bound \textup{(ii)} in Section~\ref{sec:lower-bound}, treating the regimes
$q>\frac{p(N-1)}{N-p}$ and $q<\frac{p(N-1)}{N-p}$ separately. Set
$
\alpha := \frac{N-p}{p-1} 
$ throughout this section.

\subsection{Upper Bound}\label{sec:upper-bound}
We first establish a borderline Lorentz space estimate in $L^{\beta,\infty}(\mathbb{R}^N)$. The desired upper bound then follows from the Radial Lemma~\eqref{Radial Lemma for Lorentz spaces}.\\

\begin{proposition}[Borderline Lorentz estimate]\label{Borderline Lorentz estimate}
Let $u\in \mathcal{D}^{1,p}(\mathbb R^N)\cap \mathcal{D}^{s,q}(\mathbb R^N)$ be a nonnegative weak solution of
\eqref{eq:main_problem} for $1<p<N, q>1, sq<N$. Then
\begin{equation}
u\in L^\eta(\mathbb R^N)
\qquad \text{for all }\; \eta>\frac{(p-1)N}{N-p}:=\eta_0.
\label{increasinIntegrability}
\end{equation}
Furthermore, $u$ belongs to the weak Lorentz space
$L^{\eta_0,\infty}(\mathbb R^N)$, and satisfies
\begin{equation}
\sup_{t>0}
t\,|\{x\in\mathbb R^N:u(x)>t\}|^{1/\eta_0}
\le
\mathcal S^{-1/(p-1)}
\|u\|_{L_{p^*-1}}^{(p^*-1)/(p-1)}.
\label{LorentzSpaceCondition}
\end{equation}
\end{proposition}

\begin{proof}
The argument is split into two steps. We begin by proving \eqref{increasinIntegrability}. Once this is established, \eqref{LorentzSpaceCondition} follows as a consequence.

\textbf{Step 1:} Let $0<\theta<1$ and $\varepsilon>0$. We introduce the Lipschitz, increasing function $\psi_{\varepsilon}:[0,+\infty)\to[0,+\infty)$ defined by

\begin{equation*}
\psi_{\varepsilon}(t)=\int_0^t\left[(\varepsilon+\tau)^{\frac{\theta-1}{p}}+\frac{\theta-1}{p} \tau(\varepsilon+\tau)^{\frac{\theta-1-p}{p}}\right]^p d \tau .
\end{equation*}

Since, $0<\theta<1$, we observe that 

\begin{equation}
0 \leq \psi_{\varepsilon}(t) \leq \int_0^t(\varepsilon+\tau)^{\theta-1} d \tau=\frac{1}{\theta}\left[(\varepsilon+t)^\theta-\varepsilon^\theta\right] \leq \frac{t^\theta}{\theta}. \label{TestFunctionBound}
\end{equation}
Also, define
\begin{equation*}
\Psi_{\varepsilon}(t):=\int_0^t \psi_{\varepsilon}^{\prime}(\tau)^{\frac{1}{p}} d \tau=t(\varepsilon+t)^{\frac{\theta-1}{p}}.
\end{equation*}

We now use the test function $\varphi=\psi_\varepsilon(u)\in \mathcal{D}^{1,p}(\mathbb R^N)\cap \mathcal{D}^{s,q}(\mathbb R^N)$, where the inclusion follows from the Lipschitz continuity of $\psi_\varepsilon$ and the fact that $\psi_\varepsilon(0)=0$, in the weak formulation \eqref{eq:main_problem}, which yields
\begin{equation*}
\int_{\mathbb{R}^N} |\nabla u|^{p-2} \nabla u \cdot\nabla  (\psi_{\varepsilon}(u)) +\iint_{\mathbb{R}^{2 N}} \frac{J_q(u(x)-u(y))\left(\psi_{\varepsilon}(u(x))-\psi_{\varepsilon}(u(y))\right)}{|x-y|^{N+s q}} =\int_{\mathbb{R}^N} u^{p^*-1} \psi_{\varepsilon}(u). 
\end{equation*}
We analyse the two sides separately, beginning with the right-hand side, for which it follows immediately from Proposition~\ref{prop:global_boundedness_mixed} and \eqref{TestFunctionBound} that
\begin{equation*}
\begin{aligned}
\int_{\mathbb{R}^N} u^{p^*-1} \psi_{\varepsilon}(u(x)) d x \leq & \frac{1}{\theta}\|u\|_{L^{\infty}}^{p^*+\theta-1}\left|\left\{u>M_0\right\}\right| \\
& +\left(\int_{\left\{u \leq M_0\right\}} u^{p^*} d x\right)^{\frac{p^*-p}{p^*}}\left(\int_{\left\{u \leq M_0\right\}}\left(\psi_{\varepsilon}(u) u^{p-1}\right)^{\frac{p^*}{p}} d x\right)^{\frac{p}{p^*}},
\end{aligned}
\end{equation*}

for some constant $M_0>0$.
Using \eqref{TestFunctionBound}, we obtain
\begin{equation*}
0 \leq \psi_{\varepsilon}(t) t^{p-1} \leq \frac{1}{\theta}(\varepsilon+t)^{\theta-1} t^p=\frac{1}{\theta} \Psi_{\varepsilon}(t)^p .
\end{equation*}
Combining the previous estimates, we deduce

\begin{equation*}
\begin{aligned}
\int_{\mathbb{R}^N} u^{p^*-1} \psi_{\varepsilon}(u(x)) d x \leq & \frac{1}{\theta}\|u\|_{L^{\infty}}^{p^*+\theta-1}\left|\left\{u>M_0\right\}\right| \\
& +\frac{1}{\theta}\left(\int_{\left\{u \leq M_0\right\}} u^{p^*} d x\right)^{\frac{p^*-p}{p^*}}\left(\int_{\mathbb{R}^N} \Psi_{\varepsilon}(u)^{p^*} d x\right)^{\frac{p}{p^*}}.
\end{aligned}
\end{equation*}

We now turn to the LHS.
\begin{equation*}
\underbrace{\int_{\mathbb{R}^N} |\nabla u|^{p-2} \nabla u \cdot\nabla  (\psi_{\varepsilon}(u(x))) \,dx}_{I_1}+\underbrace{\iint_{\mathbb{R}^{2 N}} \frac{J_q(u(x)-u(y))\left(\psi_{\varepsilon}(u(x))-\psi_{\varepsilon}(u(y))\right)}{|x-y|^{N+s q}} \,d x \,d y}_{I_2}. 
\end{equation*}
Consider first $I_1$
\begin{align*}
    I_1&:= \int_{\mathbb{R}^N} |\nabla u|^{p-2} \nabla u \cdot\nabla  (\psi_{\varepsilon}(u(x))) \,dx= \int_{\mathbb{R}^N}\left|\nabla \Psi_{\varepsilon}(u) \right|^p\geq \mathcal{S} \left(\int_{\mathbb{R}^N}\left|\Psi_{\varepsilon}(u)\right|^{p^*}\right)^{\frac{p}{p^*}}.
\end{align*}

Next, we analyse the term $I_2$. By Lemma A.2 in \cite{brasco_second_eigenvalue}, and assuming without loss of generality that $u(x)>u(y)$, we estimate the integrand as follows:
\begin{align*}
    J_q(u(x)-u(y))\left(\psi_{\varepsilon}(u(x))-\psi_{\varepsilon}(u(y))\right) &= (u(x)-u(y))^{q-1}\int_{u(y)}^{u(x)}\psi^{\prime}_{\varepsilon}(\tau) d \tau\\
    &\geq (u(x)-u(y))^{q-p} \left(\int_{u(y)}^{u(x)}\Psi^{\prime}_{\varepsilon}(\tau)d \tau\right)^{p}.
\end{align*}
Consequently, we obtain
\begin{align*}
    I_2&:=\iint_{\mathbb{R}^{2 N}} \frac{J_q(u(x)-u(y))\left(\psi_{\varepsilon}(u(x))-\psi_{\varepsilon}(u(y))\right)}{|x-y|^{N+s q}} \,d x \,d y\\
    &\geq \iint_{\mathbb{R}^{2 N}}\frac{\left|u(x)-u(y)\right|^{q-p} \left|\Psi_{\varepsilon}(u(x))- \Psi_{\varepsilon}(u(y)) \right|^{p}}{|x-y|^{N+sq}}\,d x \,d y \geq 0.
    \end{align*}
Putting together all the appropriate estimates, we deduce
\begin{align*}
    \mathcal{S} \lv \Psi_{\varepsilon}(u) \rv_{L^{p^*}}^p&\leq \frac{\|u\|_{L^{\infty}}^{p^*+\theta-1}}{\theta}\left|\left\{u>M_0\right\}\right|+\frac{1}{\theta}\left(\int_{\left\{u \leq M_0\right\}} u^{p^*} d x\right)^{\frac{p^*-p}{p^*}}\left(\int_{\mathbb{R}^N} \Psi_{\varepsilon}(u)^{p^*} d x\right)^{\frac{p}{p^*}}.
\end{align*}

We now choose the level $M_0 = M_0(\theta,u)>0$ such that
\begin{equation*}
\left(\int_{\left\{u \leq M_0\right\}} u^{p^*} d x\right)^{\frac{p^*-p}{p^*}} \leq \frac{\theta \mathcal{S}}{2}.
\end{equation*}
With this choice, we infer
\begin{equation*}
\left(\int_{\mathbb{R}^N}\left(u(u+\varepsilon)^{\frac{\theta-1}{p}}\right)^{p^*} d x\right)^{\frac{p}{p^*}} \leq \frac{2}{\theta \mathcal{S}}\|u\|_{L^{\infty}}^{p^*+\theta-1}\left|\left\{u>M_0\right\}\right|, \quad \text{for every $0<\theta<1$.}
\end{equation*}
Finally, letting $\varepsilon \to 0$, we obtain the desired integrability result \eqref{increasinIntegrability}.

\smallskip
\textbf{Step 2:} We now establish \eqref{LorentzSpaceCondition}. For any $t>0$, define $g_t(s)=\min\{t,s\}$ and set
\begin{equation*}
G_t(s)=\int_0^s g_t^{\prime}(\tau)^{\frac{1}{q}} d \tau=g_t(s).
\end{equation*}

Choosing $g_t(u)$ as a test function in \eqref{eq:main_problem}, we obtain
\begin{align*}
\int_{\mathbb{R}^N} |\nabla u|^{p-2} \nabla u \cdot\nabla  \left(g_t(u)\right) &+ \langle u, g_t(u)\rangle_{s,q}  = \int_{\mathbb{R}^N} u^{p^*-1} \left(g_t(u)\right).
\end{align*}

Again invoking \cite[Lemma~A.2]{brasco_second_eigenvalue}, together with the fractional Sobolev inequality \eqref{eq:best_constant_nonlocal}, we deduce that
\begin{equation*}
\begin{aligned}
0\leq\mathcal{S}_{f}\left\|g_t(u)\right\|_{L_{q^*_s}}^q & \leq\left[g_t(u)\right]_{s, q}^q \leq \iint_{\mathbb{R}^{2 N}} \frac{J_q(u(x)-u(y))\left(g_t(u(x))-g_t(u(y))\right)}{|x-y|^{N+s q}} \,d x \,d y .
\end{aligned}
\end{equation*}

Moreover,
\begin{equation*}
    \int_{\mathbb{R}^N} |\nabla u|^{p-2} \nabla u \cdot\nabla  \left(g_t(u)\right) \,dx =  \int_{\mathbb{R}^N} |\nabla g_t(u)|^{p} \,dx \geq \mathcal{S} \left\|g_t(u)\right\|_{L^{p^*}}^p.
\end{equation*}

Combining the above two estimates, we get
\begin{equation*}
    \mathcal{S} \left\|g_t(u)\right\|_{L^{p^*}}^p \leq  \int_{\mathbb{R}^N} u^{p^*-1} g_t(u) d x . 
\end{equation*}
Recalling from \eqref{increasinIntegrability} that $u\in L^{p^*-1}(\mathbb{R}^N)$, we deduce
\begin{equation*}
t|\{u>t\}|^{\frac{1}{p^*}} \leq\left\|g_t(u)\right\|_{L^{p^*}} \leq\mathcal{S}^{-\frac{1}{p}}\left(\int_{\mathbb{R}^N} u^{p^*-1} g_t(u) d x\right)^{\frac{1}{p}} \leq \mathcal{S}^{-\frac{1}{p}}t^{\frac{1}{p}}\|u\|_{L_{p^*-1}}^{\frac{p^*-1}{p}} .
\end{equation*}
The estimate above readily implies \eqref{LorentzSpaceCondition}.
\end{proof}

\subsection{Lower Bound }\label{sec:lower-bound}
We prove the above proposition by considering separately the two cases introduced in Sections~\ref{case 1} and~\ref{Case 2}.

\subsubsection{Proof of \textup{(ii)} of Theorem \ref{prop:decay-main} for the regime $q> \frac{p(N-1)}{N-p}$}\}\label{case 1}
We first construct a smooth function that serves as a subsolution of
$\mathcal{L}_{p,q}$ outside a large ball.

\begin{lemma}[Regularised subsolution]\label{lem:subsolution_gamma}
  Define $\Upsilon : \mathbb{R}^N \to \mathbb{R}$ by
  $\Upsilon(x) = \rho(|x|)$, where
  \[
    \rho(r) =
    \begin{cases}
      h(r) & 0 \le r \le 1, \\[4pt]
      r^{-\alpha} & r \ge 1,
    \end{cases}
  \]
  and $h(r) = 1+\frac{\alpha(\alpha+5)}{6} -  \frac{\alpha(\alpha+3)}{2}r^2 + \frac{\alpha(\alpha+2)}{3} r^3$ is a polynomial satisfying
  \begin{equation*}
    h(1) = 1, \quad
    h'(1) = -\alpha, \quad
    h''(1) = \alpha(\alpha+1), \quad
    h'(0) = 0,
  \end{equation*}
  with $h > 0$ and $h' \le 0$
  on $[0,1]$.

  Then the following hold:
  \begin{enumerate}
    \item $\Upsilon \in C^2(\mathbb{R}^N)$, and $\Upsilon$ is radially
          symmetric, positive, and strictly decreasing in $|x|$.
    \item There exist constants $k \gg 1$ and $c_* (N,p,q,s,\|\Upsilon\|_{C^2(\rnn)})> 0$ such that
          \[
            \mathcal{L}_{p,q}(\Upsilon)(x)
              := -\Delta_p \Upsilon(x) + (-\Delta)_q^s \Upsilon(x)
              < 0
            \qquad \text{for all } |x| \ge k.
          \]
  \end{enumerate}
\end{lemma}

\begin{proof}

\medskip
The proof of $(1)$ is immediate. We now determine the sign of
$\mathcal{L}_{p,q}(\Upsilon)$ for sufficiently large $|x|$.

\smallskip
\textit{Local term.}
Since $\alpha=(N-p)/(p-1)$ which makes $|x|^{-\alpha}$ $p$-harmonic in
$\mathbb{R}^N\setminus\{0\}$ (see \cite[p.~37]{lindqvist_notes}),
\begin{equation*}
  -\Delta_p \Upsilon = 0
  \qquad \text{in } \overline{B_1}^c := \{x\in\mathbb{R}^N : |x|>1\}.
\end{equation*}

\smallskip
\textit{Nonlocal term.}
Using Lemma 7.1 of \cite{PezzoQuaas} under the assumption $q> \frac{p(N-1)}{N-p}\iff\alpha(q-1)>N$, there exists $k\gg 1$ and $c_*>0$ such that
\begin{equation*}
  (-\Delta)_q^s \Upsilon(x)
  \le -\frac{c_*}{|x|^{N+sq}}
  \qquad \text{for all } |x| \geq k.
\end{equation*}
Thus the claim follows.
\end{proof}
\begin{lemma}\label{lem:antisymmetry}
Let $\phi\in C^2(\mathbb{R}^N)
\cap\widetilde{\mathcal{D}}^{s,q}(\overline{B_1}^c)$
satisfy the following conditions:
\begin{enumerate}[label=(\roman*)]
  \item 
        $\overline{B_1}^c\Subset\{\nabla\phi\neq 0\}$;
  \item 
        $\phi$ is positive, radially symmetric
        and radially decreasing for $|x|\ge 1$;
  \item 
        there exist $C_1,C_2,C_3>0$ such that
        for all $|x|\ge 1$
        \[
          |\phi(x)|\le\frac{C_1}{|x|^\alpha},
          \quad
          |\nabla\phi(x)|\le\frac{C_2}{|x|^{\alpha+1}},
          \quad
          |D^2\phi(x)|\le\frac{C_3}{|x|^{\alpha+2}},
        \]
\end{enumerate}
where for any $R>0$
\begin{equation*}
\begin{aligned}
\widetilde{\mathcal{D}}^{s,q}(\overline{B_R}^c) := \{ u \in L^{q-1}_{\mathrm{loc}}(\mathbb{R}^N) \cap L^{q_s^*}(\overline{B_R}^c) :&  \text{ there exists }\, E \supset \overline{B_R}^c \text{ with } E^c \text{ compact, } \\
&\mathrm{dist}(E^c, \overline{B_R}^c) > 0 \text{ and } [u]_{s,q,E} < +\infty \}.
\end{aligned}
\end{equation*}
Let $\varphi\in \mathcal{D}_0^{s,q}(\overline{B_1}^c)$.
Then the following identity holds
\begin{equation}
  \iint_{\mathbb{R}^{2N}}
    \frac{J_q(\phi(x)-\phi(y))(\varphi(x)-\varphi(y))}
         {|x-y|^{N+sq}}\,dx\,dy
  = \int_{\mathbb{R}^N}
      (-\Delta)_q^s\phi(x)\;\varphi(x)\,dx, \label{eq:antisymmetry}
\end{equation}
where $(-\Delta)_q^s\phi(x)
= \lim_{\delta\downarrow 0}T_\delta\phi(x)$
and $T_\delta\phi(x)
:= 2\int_{|x-y|>\delta}
\frac{J_q(\phi(x)-\phi(y))}{|x-y|^{N+sq}}\,dy$
is well-defined and finite for every $x\in \overline{B_1}^c$.

\end{lemma}

\begin{proof}
The proof is divided into two parts.
In Part~A, we prove the identity for
$\varphi\in C^\infty_c(\overline{B_1}^c)$.
In Part~B, we extend to general
$\varphi\in \mathcal{D}_0^{s,q}(\overline{B_1}^c)$
by a density argument.

Throughout, $E\supset \overline{B_1}^c$ denotes the open set
with $E^c$ compact, $d:=\mathrm{dist}(E^c,\overline{B_1}^c)>0$,
and $[\phi]_{s,q,E}<\infty$, from the definition
of $\widetilde{\mathcal{D}}^{s,q}(\overline{B_1}^c)$.

\medskip
\noindent\textit{Part A: Proof for
$\varphi\in C^\infty_c(\overline{B_1}^c)$
with compact support $\mathcal{K}\subset \overline{B_1}^c$.}

For $\delta>0$, define
\[
  A_\delta
  := \bigl\{(x,y)\in\mathbb{R}^{2N} : |x-y|>\delta\bigr\},
\]
\[
  I
  := \langle \phi, \varphi \rangle_{s,q},
  \quad
  I_\delta
  := \iint_{A_\delta}
       \frac{J_q(\phi(x)-\phi(y))(\varphi(x)-\varphi(y))}
            {|x-y|^{N+sq}}dxdy.
\]

\medskip
\textbf{Step 1:} $I$ is finite.

\begin{equation*}
  I
  = \underbrace{
      \iint_{E\times E}
        \frac{J_q(\phi(x)-\phi(y))(\varphi(x)-\varphi(y))}
             {|x-y|^{N+sq}}\,dx\,dy
    }_{:=\,I_E}
    + \underbrace{
        2\iint_{\overline{B_1}^c\times E^c}
          \frac{J_q(\phi(x)-\phi(y))\varphi(x)}
               {|x-y|^{N+sq}}\,dx\,dy
      }_{:=\,I_{\overline{B_1}^c\times E^c}},
\end{equation*}
\smallskip
Applying  H\"{o}lder's inequality, we get 
\begin{align*}
  |I_E|
  & \le [\phi]_{s,q,E}^{q-1}\,[\varphi]_{s,q}
  <\infty,
\end{align*}
\smallskip
Since $E^c$ is compact and $d=\mathrm{dist}(E^c,\overline{B_1}^c)>0$,
 \[|x-y|\ge C_1(1+|x|)
  \qquad\forall\,x\in \overline{B_1}^c,\;y\in E^c,\]

for some $C_1=C_1(E,\overline{B_1}^c)>0$.
Using 
$(|a|+|b|)^{q-1}\le \max\{1,2^{q-2}\}(|a|^{q-1}+|b|^{q-1})$
and then Fubini, we can deduce
\begin{align}
  |I_{\overline{B_1}^c\times E^c}|
  &\le
 C(E,\overline{B_1}^c,q)
  \Biggl[
    |E^c|
    \int_{\overline{B_1}^c}
      \frac{|\phi(x)|^{q-1}|\varphi(x)|}{(1+|x|)^{N+sq}}\,dx
    +
    \|\phi\|_{L^{q-1}(E^c)}^{q-1}
    \int_{\overline{B_1}^c}
      \frac{|\varphi(x)|}{(1+|x|)^{N+sq}}\,dx
  \Biggr].\label{eq:Icross_split}
\end{align}
For the first integral, apply H\"{o}lder's again
\begin{align*}
  \int_{\overline{B_1}^c}
    \frac{|\phi(x)|^{q-1}|\varphi(x)|}
         {(1+|x|)^{N+sq}}\,dx
  &\le
  \left(
    \int_{\overline{B_1}^c}
      \frac{|\phi(x)|^q}{(1+|x|)^{N+sq}}\,dx
  \right)^{\!\!(q-1)/q}
  \left(
    \int_{\overline{B_1}^c}
      \frac{|\varphi(x)|^q}{(1+|x|)^{N+sq}}\,dx
  \right)^{\!\!1/q}.
\end{align*}
The first factor in the above inequality
is finite by \cite[Lemma $2.4$]{Brasco_DecayEstimate},
and for any ball $B_R\subset \overline{B_1}^c$
\begin{equation*}
  \int_{\overline{B_1}^c}\frac{|\phi(x)|^q}{(1+|x|)^{N+sq}}\,dx
  \le C(N,s,q,R)\Bigl([\phi]_{s,q,E}^q
             +\|\phi\|_{L^{q-1}(B_R)}^q\Bigr)
  < \infty.
\end{equation*}
For the second factor, using Hardy's inequality \cite{FractionalHaryInequality}
\begin{equation*}
  \int_{\overline{B_1}^c}\frac{|\varphi(x)|^q}{(1+|x|)^{N+sq}}\,dx
  \le\int_{\overline{B_1}^c}\frac{|\varphi(x)|^q}{|x|^{sq}}\,dx
  \le C(N,q,s)[\varphi]_{s,q}^q
  <\infty.
\end{equation*}

For the second integral in \eqref{eq:Icross_split}, we perform similar calculations and finally obtain
\begin{equation}
 |I|
  \le C\Bigl([\phi]_{s,q,E}
            +\|\phi\|_{L^{q-1}(B_{R_0})}\Bigr)^{q-1}
     [\varphi]_{s,q}
  <\infty,\label{eq:I_finite_final}
\end{equation}
where $C=C(N,q,s,E,\overline{B_1}^c,R)>0$, and for some $R_0>0$.
\begin{remark}\label{rem:step1_general_varphi}
The bound \eqref{eq:I_finite_final} established
in Step~1 is valid for any
$\varphi\in \mathcal{D}_0^{s,q}(\overline{B_1}^c)$.
\end{remark}

\textbf{Step 2:} We prove that 
$I_\delta = \int_{\mathbb{R}^N}T_\delta\phi(x)\,\varphi(x)\,dx
\text{ for every } \delta>0.$

\begin{align*}
     &\iint_{A_\delta}
     \frac{|J_q(\phi(x)-\phi(y))||\varphi(x)-\varphi(y)|}
          {|x-y|^{N+sq}}\,dx\,dy \leq 2\int_{\mathcal{K}}|\varphi(x)|
    \int_{\{|x-y|>\delta\}}
      \frac{|\phi(x)-\phi(y)|^{q-1}}
           {|x-y|^{N+sq}}\,dy\,dx.
\end{align*}
For each $x\in\mathcal{K}$, split 
\[
  \int_{\{|x-y|>\delta\}} \!
    \frac{|\phi(x)-\phi(y)|^{q-1}}{|x-y|^{N+sq}}dy
  =
  \underbrace{
    \int_{\delta<|x-y|\le 1} \!
      \frac{|\phi(x)-\phi(y)|^{q-1}}{|x-y|^{N+sq}}dy
  }_{=:\,N_\delta(x)}
  +
  \underbrace{
    \int_{|x-y|>1} \!
      \frac{|\phi(x)-\phi(y)|^{q-1}}{|x-y|^{N+sq}}dy
  }_{=:\,F(x)}.
\]
Since $\phi\in C^1(\mathbb{R}^N)$, for each fixed $\delta>0$, we have
\[
  N_\delta(x)
  \le C(\|\nabla\phi\|_{L^\infty(\mathcal{K}+B_1)},q)
      \int_{\delta<|x-y|\le 1} \frac{dy}{|x-y|^{N + sq -q+1}}
  = C\omega_N\int_\delta^1 r^{q-2-sq}\,dr
  =: N_\delta(\mathcal{K}).
\]
\begin{align*}
  F(x)
  &\le C(q)\int_{|x-y|>1}
        \frac{|\phi(x)|^{q-1}+|\phi(y)|^{q-1}}
             {|x-y|^{N+sq}}\,dy\\
  &=
  C(q)|\phi(x)|^{q-1}
    \int_{|z|>1} \frac{dz}{|z|^{N+sq}}
  +
  C(q)\int_{|x-y|>1}
       \frac{|\phi(y)|^{q-1}}{|x-y|^{N+sq}}\,dy\leq C(N,s,q)\|\phi\|_{L^\infty}^{q-1}.
\end{align*}
Thus 
\begin{align*}
  &\iint_{A_\delta}
     \frac{|J_q(\phi(x)-\phi(y))||\varphi(x)-\varphi(y)|}
          {|x-y|^{N+sq}}\,dx\,dy < \infty \qquad \text{for every $\delta>0$.}
          \end{align*}
Thus, the application of Fubini's theorem is justified, and it is now straightforward to obtain the claimed identity in Step 2.

\smallskip
\textbf{Step 3: $I_\delta\to I$ as $\delta\downarrow 0$.}
\[
  I-I_\delta
  =\iint_{\{|x-y|\le\delta\}}
     \frac{J_q(\phi(x)-\phi(y))(\varphi(x)-\varphi(y))}
          {|x-y|^{N+sq}}\,dx\,dy.
\]
Since $\varphi\in C^\infty_c(\overline{B_1}^c)$ has compact support
$\mathcal{K}\subset \overline{B_1}^c\subset E$, and $E^c\subset B_1$,
we have $\mathcal{K}\cap E^c=\emptyset$.
Since both $\mathcal{K}$ and $E^c$ are compact and disjoint:
$
  d_\mathcal{K}
  := \mathrm{dist}(\mathcal{K},E^c)
  \ge \mathrm{dist}(\overline{B_1}^c,E^c)
  = d > 0.
$
For $0<\delta<d_\mathcal{K}$ and any pair $(x,y)$
with $|x-y|\le\delta$ and $\varphi(x)-\varphi(y)\ne 0$:
either $x\in\mathcal{K}$ or $y\in\mathcal{K}$. If $x\in\mathcal{K}$ and $|x-y|\le\delta<d_\mathcal{K}$,
then $y\notin E^c$. Therefore $\{|x-y|\le\delta,\,\varphi(x)-\varphi(y)\ne 0\}
\subseteq E\times E$. Finally, using H\"older's inequality, we obtain
\begin{align*}
  &|I-I_\delta|\\
  &\le
  \left(
    \iint_{\{|x-y|\le\delta\}\cap(E\times E)}
      \frac{|\phi(x)-\phi(y)|^q}{|x-y|^{N+sq}}\,dx\,dy
  \right)^{\frac{q-1}{q}}
  \left(
    \iint_{\{|x-y|\le\delta\}\cap(E\times E)}
      \frac{|\varphi(x)-\varphi(y)|^q}{|x-y|^{N+sq}}\,dx\,dy
  \right)^{\!\!1/q}
  \notag\\
  &\leq
  [\phi]_{s,q,E}^{q-1}
  \cdot
  [\varphi]_{s,q,\{|x-y|\le\delta\}}.
\end{align*}
The integrand
$\frac{|\varphi(x)-\varphi(y)|^q}{|x-y|^{N+sq}}
\chi_{\{|x-y|\le\delta\}}\to 0$
pointwise a.e.\ as $\delta\downarrow 0$,
and is dominated by
$\frac{|\varphi(x)-\varphi(y)|^q}{|x-y|^{N+sq}}
\in L^1(\mathbb{R}^{2N})$
since $[\varphi]_{s,q} \in \mathcal{D}_0^{s,q}(\overline{B_1}^c)$.
Dominated convergence theorem then proves the claim.

\smallskip
\textbf{Step 4:}
$\int_{\mathbb{R}^N}T_\delta\phi(x)\varphi(x)\,dx
\to \int_{\mathbb{R}^N}(-\Delta)_q^s\phi(x)\varphi(x)\,dx$
as $\delta\downarrow 0$.

We apply the dominated convergence theorem
to $\int_{\mathcal{K}}T_\delta\phi(x)\,\varphi(x)\,dx$,
where $\mathcal{K}:=\mathrm{supp}(\varphi)\Subset \overline{B_1}^c$.
Since $d_0:=\mathrm{dist}(\mathcal{K},\partial B_1)>0$,
fix $0<R<d_0/3$,
so that $B_{3R/2}(x)\Subset \overline{B_1}^c$
for every $x\in\mathcal{K}$. For each fixed $x\in\mathcal{K}$ and $\delta\in(0,R)$,
write:

\[
  T_\delta\phi(x)
  =
 2 \underbrace{
    \int_{\delta<|x-y|<R}
      \frac{J_q(\phi(x)-\phi(y))}{|x-y|^{N+sq}}\,dy
  }_{=:\,S_\delta(x)}
  +
 2 \underbrace{
    \int_{|x-y|\ge R}
      \frac{J_q(\phi(x)-\phi(y))}{|x-y|^{N+sq}}\,dy
  }_{=:\,T_R\phi(x)}.
\]
The far term $T_R\phi(x)$ is independent of $\delta$ and finite. Moreover, \cite[Lemma~3.6]{FractionalLaplacianEquivalenceOfSolutions} applies appropriately to the first integral, in view of the assumption
$
\overline{B_1}^c \Subset \{\nabla \phi \neq 0\}.
$
Hence
\[
(-\Delta)^s_q\phi(x)
=
2\operatorname{P.V.}\int_{\mathbb R^N}
\frac{J_q(\phi(x)-\phi(y))}{|x-y|^{N+sq}}\,dy
\]
is well defined and finite. Hence, the dominated convergence theorem proves the claim of Step~4.  Finally, we have established the identity \eqref{eq:antisymmetry} for Part~A:
\[
  I
  = \lim_{\delta\downarrow 0}I_\delta
  = \lim_{\delta\downarrow 0}
      \int_{\mathbb{R}^N}T_\delta\phi(x)\,\varphi(x)\,dx
  = \int_{\mathbb{R}^N}(-\Delta)_q^s\phi(x)\,\varphi(x)\,dx.
\]

\medskip
\noindent\textit{Part B: Extension to general
$\varphi\in \mathcal{D}_0^{s,q}(\overline{B_1}^c)$.}

Using \cite[Theorem~2.1]{Brasco_DecayEstimate}, there exists
$\{\varphi_n\}_{n\in\mathbb{N}}\subset C^\infty_c(\overline{B_1}^c)$
such that $[\varphi_n-\varphi]_{s,q}\to 0$
as $n\to\infty$.
By Part~A, for each $n\in\mathbb{N}$:
\begin{equation*}
  \iint_{\mathbb{R}^{2N}}
    \frac{J_q(\phi(x)-\phi(y))
          (\varphi_n(x)-\varphi_n(y))}
         {|x-y|^{N+sq}}\,dx\,dy
  = \int_{\mathbb{R}^N}
      (-\Delta)_q^s\phi(x)\,\varphi_n(x)\,dx.
\end{equation*}
Denoting $\psi_n:=\varphi_n-\varphi\in \mathcal{D}_0^{s,q}(\overline{B_1}^c)$,
and Remark~\ref{rem:step1_general_varphi}, we get
\begin{align*}
  &\left|
    \iint_{\mathbb{R}^{2N}}
      \frac{J_q(\phi(x)-\phi(y))
            (\varphi_n(x)-\varphi_n(y))}
           {|x-y|^{N+sq}}\,dx\,dy
    -\iint_{\mathbb{R}^{2N}}
       \frac{J_q(\phi(x)-\phi(y))
             (\varphi(x)-\varphi(y))}
            {|x-y|^{N+sq}}\,dx\,dy
  \right|\\
  &=
  \left|
    \iint_{\mathbb{R}^{2N}}
      \frac{J_q(\phi(x)-\phi(y))
            [\psi_n(x)-\psi_n(y)]}
           {|x-y|^{N+sq}}\,dx\,dy
  \right|\\
  &\le
  C\Bigl([\phi]_{s,q,E}
        +\|\phi\|_{L^{q-1}(B_{R_0})}\Bigr)^{q-1}
  [\psi_n]_{s,q}
  \to 0
  \quad\text{as }n\to\infty.
\end{align*}
We now show
\begin{equation}\label{eq:RHS_to_show}
  \int_{\mathbb{R}^N}(-\Delta)_q^s\phi(x)\,\varphi_n(x)\,dx
  \to
  \int_{\mathbb{R}^N}(-\Delta)_q^s\phi(x)\,\varphi(x)\,dx.
\end{equation}
Since $\psi_n:=\varphi_n-\varphi\in \mathcal{D}_0^{s,q}(\overline{B_1}^c)$,
by H\"{o}lder's inequality
and the fractional Hardy inequality
\cite{FractionalHaryInequality}, we deduce
\begin{equation*}
  \left|
    \int_{\mathbb{R}^N}
      (-\Delta)_q^s\phi(x)\,\psi_n(x)\,dx
  \right|
  \le
  C(N,s,q)
  \left\|(-\Delta)_q^s\phi\;|x|^s\right\|_{L^{q'}(\overline{B_1}^c)}
  [\psi_n]_{s,q},
\end{equation*}
where $\frac{1}{q}+\frac{1}{q^{'}}=1$. So \eqref{eq:RHS_to_show} follows once we establish
$
  (-\Delta)_q^s\phi\;|x|^s\in L^{q'}(\overline{B_1}^c).
$
By \cite[Lemma~7.1]{PezzoQuaas}, since $\alpha(q-1)>N$,
there exist $\widetilde{R}>1$ and
$C=C(\alpha,N,q,s,\|\phi\|_{C^2(\mathbb{R}^N)})>0$ such that
\[
  \int_{B_{\widetilde{R}}^c}
    |(-\Delta)_q^s\phi(x)|^{q'}|x|^{sq'}\,dx
  \le C
  \int_{B_{\widetilde{R}}^c}
    |x|^{-(N+sq-s)q'}\,dx
  < \infty,
\]
The finiteness of the integral over
$\{1<|x|<\widetilde{R}\}$
follows again from
\cite[Lemma~3.6]{FractionalLaplacianEquivalenceOfSolutions}.
This completes the proof of the lemma.
\qedhere
\end{proof}

\medskip
\begin{proof}[Proof of \textup{(ii)} Theorem \ref{prop:decay-main}]
Let $\Upsilon$, and $k$ be as in Lemma~\ref{lem:subsolution_gamma}.

Define $\phi(x):=K\Upsilon(kx)$, where $K$ will be chosen later, then
\begin{align*}
  (-\Delta)_q^s \phi(x)= K^{q-1}k^{sq}\,(-\Delta)_q^s \Upsilon(kx), \quad  -\Delta_p\phi(x)
  = K^{p-1}k^p\,(-\Delta_p \Upsilon)(kx).
\end{align*}

The aforementioned fractional term $(-\Delta)_q^s \phi(x) \in L_{loc}^{\infty}(\overline{B_R}^c)$ for any $R > 0$ by Lemma $3.6$ of \cite{FractionalLaplacianEquivalenceOfSolutions}. 
Then by Lemma~\ref{lem:subsolution_gamma}, we get
\begin{equation}
  \mathcal{L}_{p,q}(\phi)(x)
  = K^{p-1}k^p\,(-\Delta_p\Upsilon)(kx)
    + K^{q-1}k^{sq}\,(-\Delta)_q^s\Upsilon(kx) <0 \qquad\forall\,|x|\geq 1. \label{eq:phi_subsol}
\end{equation}

Since $\Upsilon\in C^2(\mathbb{R}^N)$, and using the Harnack inequality from Proposition \ref{Weak Harnack inequality}, we have
\begin{equation*}
  m := \inf_{x\in B_1} U(x) > 0,  \; M := \sup_{x\in B_1}\Upsilon(kx) =\sup_{|z|\le k}\Upsilon(z)
     = \Upsilon(0) < \infty, \;   K := \frac{m}{M} > 0.
\end{equation*}

Hence, 
\begin{equation*}
  \phi(x) \le U(x) \qquad \text{for a.e.\ } x\in B_1.
\end{equation*}

\medskip
We now verify the weak inequality
$\langle \! \langle \mathcal{L}_{p,q}(\phi),\varphi\rangle \! \rangle
\le \langle \! \langle \mathcal{L}_{p,q}(U),\varphi\rangle \! \rangle$
for all nonnegative functions $\varphi\in \X(\overline{B_1}^c)$, where the notations are as in Lemma~\ref{dualityLemma}.
Firstly, both $U, \phi \in \mathcal D^{s,p,q}(\overline{B_1}^c)$. Indeed, by Proposition~\ref{prop:global_boundedness_mixed},
$
U\in L^\infty(\mathbb R^N)\cap \X(\rnn)
\subset L^{q-1}_{\rm loc}(\mathbb R^N)\cap L^{p^*}(\overline{B_1}^c),
$
with $\|\nabla U\|_{L^p(\rnn)}<\infty$ and
$[U]_{s,q}<\infty$ by Theorem~\ref{ineq:main_embedding}. Moreover,
$
\phi(x)=K\Upsilon(kx)\sim |x|^{-\alpha}
\;\;\text{for }|x|\ge1,
$
with $\alpha=\frac{N-p}{p-1}$, so that
$
\phi\in L^{q-1}_{\rm loc}(\mathbb R^N)\cap L^{p^*}(\overline{B_1}^c)
\cap \mathcal{D}^{1,p}(\mathbb R^N),
$
and again $[\phi]_{s,q}<\infty$.
\smallskip
Then, since $\phi=K\Upsilon(k\,\cdot\,)\in C^2(\mathbb{R}^N)$, using \eqref{eq:phi_subsol} and $\varphi\ge 0$, we have
\begin{align}
  \langle \! \langle \mathcal{L}_{p,q}(\phi),\varphi\rangle \! \rangle
  &= \int_{\mathbb{R}^N}|\nabla\phi|^{p-2}\nabla\phi\cdot\nabla\varphi\,dx
    + \iint_{\mathbb{R}^{2N}}
        \frac{J_q(\phi(x)-\phi(y))(\varphi(x)-\varphi(y))}
             {|x-y|^{N+sq}}\,dx\,dy \notag\\
  \label{eq:phi_weak_subsol}
  &= -\int_{\overline{B_1}^c}(\Delta_p\phi)\,\varphi\,dx+  \int_{\mathbb{R}^N}(-\Delta)_q^s\phi(x)\,\varphi(x)\,dx \leq 0,
\end{align}
where the identity for the nonlocal term used in the second-to-last step above follows from Lemma \ref{lem:antisymmetry}.
 Further as $U \geq 0$ is a weak solution of \eqref{eq:main_problem} in $\mathbb{R}^N$ and for $\varphi\in \X(\overline{B_1}^c)\subset \X(\rnn)$, we get
\begin{equation}\label{eq:U_weak_sol}
  \langle \! \langle \mathcal{L}_{p,q}(U),\varphi\rangle \! \rangle
  = \int_{\mathbb{R}^N} U^{p^*-1}\varphi\,dx \ge 0.
\end{equation}

Combining \eqref{eq:phi_weak_subsol} and \eqref{eq:U_weak_sol}
\begin{equation*}
  \langle \! \langle \mathcal{L}_{p,q}(\phi),\varphi\rangle \! \rangle
  \le 0
  \le \langle \! \langle \mathcal{L}_{p,q}(U),\varphi\rangle \! \rangle
  \qquad
  \forall\,\varphi\in \X(\overline{B_1}^c),\;\varphi\ge 0.
\end{equation*}

Therefore all the hypotheses of
Theorem~\ref{thm:comparison_principle_mixed} are satisfied on $\Omega=\overline{B_1}^c$, thus
\[
  U(x)
  \ge \phi(x)
  = K\Upsilon(kx)
  = \frac{m}{M}\cdot k^{-\alpha}|x|^{-\alpha}
  = c_1\,|x|^{-\frac{N-p}{p-1}}, \qquad \text{for }|x|\geq 1
\]
where
$
  c_1 := \frac{m\,k^{-\alpha}}{M} > 0.
$
This completes the proof.

\end{proof}

\subsubsection{Proof of \textup{(ii)} of Theorem \ref{prop:decay-main} for the regime  $q< \frac{p(N-1)}{N-p}$}\label{Case 2}

\medskip
 
We first record the upper bound on $(-\Delta)_q^s\widetilde{\Gamma}$
under the condition $\alpha(q-1)<N \iff q< \frac{p(N-1)}{N-p}$ where $\widetilde{\Gamma}$ is as defined in the subsequent Lemma.
 
\begin{lemma}[Upper bound for $\mathcal{L}_{p,q}(\widetilde{\Gamma})$]
\label{lem:nonlocal_upper}
For $x\in \rnn$, define $\widetilde{\Gamma}(x) = \min\{1, |x|^{-\alpha}\}$.
Then for every $R > 1$, $\widetilde{\Gamma} \in \mathcal{D}^{s,p,q}(\overline{B_R}^c)$
and $\mathcal{L}_{p,q}(\widetilde{\Gamma}) \leq C/|x|^{N+sq}$
weakly in $\overline{B_R}^c$, i.e.,
\begin{equation*}
  \langle \! \langle \mathcal{L}_{p,q}(\widetilde{\Gamma}),\varphi\rangle \! \rangle
  \leq \int_{\overline{B_R}^c}
         \frac{C(N,q,s,\alpha,R)}{|x|^{N+sq}}
         \varphi(x)\,dx
\end{equation*}
for all nonnegative $\varphi \in \X(\overline{B_1}^c)$,
where $C(N,q,s,\alpha,R) > 0$ is a constant. 
\end{lemma}
\begin{proof}

It is easy to see that $\widetilde{\Gamma}\in \mathcal{D}^{s,p,q}(\overline{B_R}^c)$. Then for any nonnegative $\varphi\in  \X(\overline{B_1}^c)$:
\begin{equation*}
  \langle \! \langle \mathcal{L}_{p,q}(\widetilde{\Gamma}),\varphi\rangle \! \rangle
  = \underbrace{\int_{\mathbb{R}^N}|\nabla\widetilde{\Gamma}|^{p-2}
      \nabla\widetilde{\Gamma}\cdot\nabla\varphi\,dx}_{=0}
    + \iint_{\mathbb{R}^{2N}}
        \frac{J_q(\widetilde{\Gamma}(x)-\widetilde{\Gamma}(y))
              (\varphi(x)-\varphi(y))}
             {|x-y|^{N+sq}}\,dx\,dy.
\end{equation*}

To evaluate $ (-\Delta)_q^s\widetilde{\Gamma}(x)$ we want to invoke Proposition $2.8$ of \cite{Brasco_DecayEstimate} with $\Omega=\overline{B_R}^c$, $u=\Gamma$ with $\Gamma(x)=|x|^{-\alpha}$, $f=(-\Delta)_q^s\Gamma$, and $v=-(\Gamma-1)_+$.  \\
Note that by Lemma~A.1 and Lemma~A.2 of \cite{Brasco_DecayEstimate} gives $\Gamma \in \tilde{\mathcal{D}}^{s,q}(\overline{B_R}^c)$ (this space is defined in Lemma \ref{lem:antisymmetry}) and
\begin{equation*}
  (-\Delta)_q^s\Gamma(x)
  = C(\alpha)\,|x|^{-\alpha(q-1)-sq}
  \quad\text{weakly in }\overline{B_R}^c,
\end{equation*}
where $C(\alpha) = 2 \int_{0}^1 t^{sq -1} (1 - t^{N - s q  - \alpha (q-1)}) | 1 - t^{\alpha} |^{q-1} \Phi(t) \, dt$ with the kernel given by $\Phi(s)
=
\mathcal{H}^{N-2}(\mathbb{S}^{N-2})
\int_{-1}^{1}
\frac{
(1-t^2)^{\frac{N-3}{2}}
}{
(1 - 2ts + s^2)^{\frac{N+qs}{2}}
}
\, dt$. Also, since $p < q$, we have $N - sq - \alpha (q - 1) < 0$. This yields $C(\alpha) < 0$. \\
Moreover, the function $f = C(\alpha)|x|^{-\alpha(q-1)-sq}$ satisfies $f \in L^1_{\rm loc}(\overline{B_R}^c)$ trivially since $|x| \geq R > 0$. For the dual membership $f \in (\mathcal{D}_0^{s,q}(\overline{B_R}^c))^*$, by the fractional Sobolev inequality
\[
  \left|\int_{\overline{B_R}^c} f\varphi\,dx\right|
  \leq \|f\|_{L^{(q^*_s)'}(\overline{B_R}^c)}\|\varphi\|_{L^{q^*_s}(\overline{B_R}^c)}
  \leq C\|f\|_{L^{(q^*_s)'}(\overline{B_R}^c)}[\varphi]_{s,q},
\quad \varphi \in \mathcal{D}_0^{s,q}(\overline{B_R}^c)\]
and $\|f\|_{L^{(q^*_s)'}(\overline{B_R}^c)} < \infty$ since $(q^*_s)' = qN/(N(q-1)+sq)$ and the integrability condition $(\alpha(q-1)+sq)(q^*_s)' > N$ reduces to $\alpha > (N-sq)/q$, which holds as it reduces to $\alpha = (N-p)/(p-1) > (N-sq)/q$. Hence $f \in (\mathcal{D}_0^{s,q}(\overline{B_R}^c))^*$. Finally, the function $v := -(\Gamma-1)_+$ satisfies $\text{supp}(v)\subseteq\overline{B_1}$ and $v\in L^{q-1}(\mathbb{R}^N)$.
Thus, for $|x|>R>1$, we obtain
\begin{equation*}
\langle (-\Delta)_q^s\widetilde{\Gamma}(x) ,\varphi\rangle_{s,q}
  =  \left\langle \frac{C(\alpha)}{|x|^{\alpha(q-1) +sq}},\varphi \right\rangle_{p}
    + 2 \left\langle \int_{B_1}
        \frac{J_q(\Gamma(y)-\Gamma(x)) - J_q(1-\Gamma(x))}
             {|x-y|^{N+sq}}\,dy, \varphi \right\rangle_{s,q}.
\end{equation*}

To estimate the second term, using $|x-y|\ge\frac{R-1}{R}|x|$ for $|x|>R>1$ and $y\in B_1$, we deduce
\begin{equation*}
  2\int_{B_1}
    \frac{J_q(\Gamma(y)-\Gamma(x)) - J_q(1-\Gamma(x))}
         {|x-y|^{N+sq}}\,dy
  \le \frac{2\bigl(\frac{R}{R-1}\bigr)^{N+sq}}{|x|^{N+sq}}
      \int_{B_1}\Gamma(y)^{q-1}\,dy
  \leq \frac{C(N,s,\alpha,q,R)}{|x|^{N+sq}},
\end{equation*}
where we used $
\int_{B_1}|y|^{-\alpha(q-1)}dy  < \infty$, since $\alpha(q-1)<N$. Thus the claim follows.
\end{proof}

Next, we establish a lower bound for the solution $U$ of \eqref{eq:main_problem}. To this end, we begin by defining an auxiliary problem. By Harnack's inequality in Proposition \ref{Weak Harnack inequality}, there exists a constant $ c^* > 0$ such that $\inf_{B_1} U \geq  c^*$. We then define the following minimization problem.
\begin{equation}\label{eq:auxiliary_pb_mixed}\tag{$A$}
\mathcal{A} := \inf_{u \in \X(\rnn)} \left\{ [u]_{1,p}^p + [u]_{s,q}^q : u \geq \frac{c^*}{2} \;
\text{ on }\;B_1 \right\}
\end{equation}

\begin{proposition}
The problem \eqref{eq:auxiliary_pb_mixed} has a unique positive solution $U_{0} \in  \X(\rnn) \cap C_{loc}^{0,\alpha}(\overline{B_1}^c)$. Moreover, $U_{0}$ is radial and non-increasing, and $U_{0} \in \X(\rnn)$ solves in the weak sense
\begin{equation*}
    \begin{cases}
        -\Delta_{p}\, U_{0} + (-\Delta_{q})^{s} U_{0} = 0
            & \text{in } \overline{B_{1}}^{c}, \\[4pt]
        U_{0} \equiv \frac{c^*}{2}
            & \text{in } \overline{B}_1.
    \end{cases}
\end{equation*}
\end{proposition}
\begin{proof}
The existence of a minimizer follows directly as in \cite[Proposition $3.5$]{Brasco_DecayEstimate}. Note that, since the constraint forces nontriviality, we have $U_{0} \not \equiv 0$. \\
Define the functional $\mathcal{E}(u) := \frac{[u]_{{1,p}}^{p}}{p} + \frac{[u]_{{s,q}}^{q}}{q}$. Since the functional $\mathcal{E}$ is strictly convex on $\X(\rnn)$, and the admissible set $\mathcal{K}
:=
\left\{
u\in \X(\mathbb R^N):
u \geq \frac{c^*}{2} \;
\text{ on }\;B_1
\right\}
$ is convex and closed. Uniqueness of the minimizer follows immediately. Positivity follows from the maximum principle.

The Schwarz decreasing rearrangement $U_{0}^{\#}$ of $U_{0}$ defined via the layer-cake representation: for each $t > 0$,\begin{equation*}
    \{ U_{0}^{\#} > t \} = B_{r(t)},
    \qquad
    r(t) := \left(\frac{|\{ U_{0} > t \}|}{\omega_{N}}\right)^{1/N},
\end{equation*}
where $\omega_{N} = |B_{1}|$, so that $|B_{r(t)}| = |\{ U_{0} > t \}|$. It satisfies
\[
    [U_{0}^{\#}]_{{1,p}} \leq [U_{0}]_{{1,p}}
    \qquad\text{and}\qquad
    [U_{0}^{\#}]_{{s,q}} \leq [U_{0}]_{{s,q}},
\]
by the P\'olya-Szeg\H{o} inequality for $\mathcal{D}_0^{1,p}$ (see \cite[Theorem $2.1.3$]{henrot_extremum_problems}) and its fractional analogue for $\mathcal{D}_0^{s,q}$ (see \cite[Theorem 6.2]{malhotra2025eigenvalues}). Moreover, $U_{0}^{\#} \geq \frac{c^*}{2}\chi_{B_1}$. Indeed, fix any $
0\le t<\frac{ c^*}{2}.
$ Since $U_{0} \geq \frac{c^*}{2}$ on $B_{1}$, we have
$B_{1} \subseteq \{ U_{0} > t \}$, and therefore
$
    |\{ U_{0} > t \}| \geq |B_{1}| = \omega_{N},
$
which gives $B_{1} \subseteq B_{r(t)} = \{ U_{0}^{\#} > t \}$.
Since this holds for every $t \in [0,\frac{c^*}{2})$, letting $t \nearrow \frac{ c^*}{2}$ yields
$U_{0}^{\#} \geq \frac{c^*}{2}$ on $B_{1}$, hence $U_{0}^{\#} \geq \frac{c^*}{2} \chi_{B_{1}}$. 
Thus $U_{0}^{\#}$ is admissible and achieves energy no larger than $U_{0}$, uniqueness forces $U_{0}^{\#} = U_{0}$. Hence $U_{0}$ is radial and non-increasing.

Now, for every $u \in \X(\rnn)$ and $t \in \mathbb{R}$, it follows from \cite[Remark 3.3]{mosconi_LS_ineq} that
\begin{equation}\label{eq:sublattice}
    \mathcal{E}(\max\{u,t\}) + \mathcal{E}(\min\{u,t\}) \leq \mathcal{E}(u).
\end{equation}
Applying \eqref{eq:sublattice} shows that $\min\left\{U_{0}, \frac{c^*}{2} \right\} \geq  \frac{c^*}{2} \chi_{B_{1}}$ is admissible and $\mathcal{E}(\min\{U_{0}, \frac{c^*}{2}\}) \leq \mathcal{E}(U_{0})$. By minimality and uniqueness, $\min\{U_{0},\frac{c^*}{2}\} = U_{0}$, i.e.\ $U_{0} \leq \frac{c^*}{2}$. Combined with the constraint $U_{0} \geq \frac{c^*}{2}$ on $B_{1}$, we conclude $U_{0} \equiv \frac{c^*}{2}$ in $\overline{B}_{1}$. Since $U_{0}$ minimizes $\mathcal{E}$ freely on $\overline{B_{1}}^{\,c}$, it satisfies the Euler-Lagrange equation for $\mathcal{E}$ in the weak sense.

Finally, applying Theorem $5$ of \cite{CristianaGiuseppe}, we get $U_{0} \in C_{loc}^{0,\alpha}(\overline{B_{1}}^{\,c})$ for some $\alpha \in (0,1)$.
\end{proof}
\begin{lemma}\label{lem: strict dercresing U_0}
Let $U_{0}$ be the unique solution of \eqref{eq:auxiliary_pb_mixed}. If $R > 1$ then $U_{0}(R) < \frac{c^*}{2}$.
\end{lemma}
\begin{proof}
Define 
\vspace{-1.5em}
\begin{equation*}
    x^* = \sup \{ |z| \in [0,\infty) : U_{0}(|y|) = \frac{c^*}{2} \text{ for all } y \in [0,|z|) \}.
\end{equation*}
Since, $U_{0}$ is continuous, $x^*$ is achieved. We have to show that $x^* = 1$. On contrary, assume that $x^* > 1$. Then we define $U_1(x) = U_{0}(x^* x )$. Then clearly $U_1$ is an admissible function. Consider
\begin{align*}
  \mathcal E(U_1)
&=\frac{1}{p}[U_1]_{1,p}^p
+
\frac1q [U_1]_{s,q}^q =\frac{1}{p} (x^*)^{p-N}[U_0]_{1,p}^p
+
\frac1q (x^*)^{sq-N}[U_0]_{s,q}^q < \mathcal{E}(U_{0}).
\end{align*}
But this is contrary to the definition of $U_{0}$. Therefore, $U_{0}(R) < \frac{c^*}{2}$.
\end{proof}

\begin{lemma}[Nonlocal computation for $\bar{U}$]
\label{lem:nonlocal_barU}
Let $U_{0}\in \X(\mathbb{R}^N)$ be a positive radially symmetric
decreasing solution of \eqref{eq:auxiliary_pb_mixed}.
Define
\begin{equation*}
  \bar{U}(x) := \min\{U_{0}(x),\, U_{0}(2)\},
\end{equation*}
Then $\bar{U}\in \mathcal{D}^{s,p,q}(\overline{B_3}^c)$ and for a.e.\ $|x|>3$:
\begin{equation}
  \langle \! \langle \mathcal{L}_{p,q}(\bar{U}),\varphi\rangle \! \rangle
  \ge \int_{\overline{B_3}^c}\frac{c_1}{|x|^{N+sq}}\varphi(x)\,dx
  > 0 \label{Eq:LoweBoundOfTruncationOfAuxillaryProblemSolution}
\end{equation}
for all nonnegative $\varphi\in \X(\overline{B_3}^c)$,
where
\begin{equation*}
  c_1 :=
  \frac{2\omega_N\cdot 3^{N+sq}}{4^{N+sq}}\cdot c_0 > 0,
  \qquad
  c_0 :=
  \begin{cases}
    2^{2-q}(U_{0}(1)-U_{0}(2))^{q-1} & \text{if } q\ge 2, \\[4pt]
    \dfrac{(U_{0}(1)-U_{0}(2))^{q-1}}{2^{q-1}} & \text{if } 1 < q < 2.
  \end{cases}
\end{equation*}
\end{lemma}

\begin{proof}
    It is easy to see that $\bar{U}\in \mathcal{D}^{s,p,q}(\overline{B_3}^c)$. 
Following as in Lemma $2.8$ of \cite{HolderRegularityFractional}, for all nonnegative $\varphi\in \X(\overline{B_3}^c)$, we get 
 \begin{align}
  &\langle \! \langle \mathcal{L}_{p,q}(\bar{U}),\varphi\rangle \! \rangle
  =  \langle -\Delta_p\bar{U}(x),\varphi\rangle_{p} +   \langle (-\Delta)_q^s\bar{U}(x) ,\varphi\rangle_{s,q}
  \notag\\
  &= \langle -\Delta_p U_{0}(x),\varphi\rangle_{p} + \langle (-\Delta)_q^s U_{0}(x) ,\varphi\rangle_{s,q}
    \\
    & \quad + 2\int_{\overline{B_3}^c} \int_{B_2}
        \frac{J_q(U_{0}(x)-U_{0}(2)) - J_q(U_{0}(x)-U_{0}(y))}
             {|x-y|^{N+sq}}  \varphi(x)\,dy \, dx
  \notag\\
  &= \langle \! \langle \mathcal{L}_{p,q}(U_{0})(x),\varphi\rangle \! \rangle
    + 2\int_{\overline{B_3}^c} \int_{B_2}
        \frac{J_q(U_{0}(x)-U_{0}(2)) - J_q(U_{0}(x)-U_{0}(y))}
             {|x-y|^{N+sq}} \varphi(x)\,dy \, dx
  \notag\\
  &=  2\int_{\overline{B_3}^c}  \int_{B_2}
        \frac{J_q(U_{0}(y)-U_{0}(x))-J_q(U_{0}(2)-U_{0}(x))}
             {|x-y|^{N+sq}} \varphi(x)\,dy \, dx,\notag\\
  &\geq  2\int_{\overline{B_3}^c}  \int_{B_1}
        \frac{J_q(U_{0}(y)-U_{0}(x))-J_q(U_{0}(2)-U_{0}(x))}
             {|x-y|^{N+sq}} \varphi(x)\,dy \, dx.    \label{eq:Lbarv_pointwise} 
\end{align}
The above computation is valid for each $\varphi\in \X(\overline{B_3}^c)$ since  $\langle \! \langle \mathcal{L}_{p,q}({U_{0}}),\varphi\rangle \! \rangle < \infty $ and 
\begin{align*}
  &\iint_{\overline{B_3}^c\times\overline{B_2}}
     \frac{|J_q(U_0(x)-U_0(2))|
           +|J_q(U_0(x)-U_0(y))|}
          {|x-y|^{N+sq}}
     |\varphi(x)|\,dx\,dy\\
  &\le
   C(N,s,q)(c^*/2)^{q-1} \int_{\overline{B_3}^c}
    \frac{|\varphi(x)|}{|x|^{N+sq}}\,dx
  \le C \mathcal{S}_f^{-1/q}[\varphi]_{s,q} <\infty.
\end{align*}

Now, to get a lower bound on the RHS of 
\eqref{eq:Lbarv_pointwise}, we consider two cases. 
For $q \geq 2$, we apply \cite[eq.~(2.2)]{Brasco_DecayEstimate}, which states that for all $a\in\mathbb{R}$, $b\ge 0$, $q\ge 2$, we have
\begin{equation}\label{eq:brasco22}
  J_q(a) - J_q(a+b) \le -2^{2-q}b^{q-1}.
\end{equation}
Set $a:=U_{0}(x)-U_{0}(y)<0$ and $b:=U_{0}(y)-U_{0}(2)\ge U_{0}(1)-U_{0}(2)>0$ for $y\in B_1$. Applying \eqref{eq:brasco22}, we deduce
\begin{equation*}
\begin{aligned}
 J_q(U_{0}(y)-U_{0}(x))-J_q(U_{0}(2)-U_{0}(x))
  \ge 2^{2-q}(U_{0}(y)-U_{0}(2))^{q-1} \\
  \ge 2^{2-q}(U_{0}(1)-U_{0}(2))^{q-1}
  =: c_0 > 0.
\end{aligned}
\end{equation*}
For the case $1 < q < 2$, we apply \cite[eq.~(2.4)]{Brasco_DecayEstimate},
which states that for all $a\in[0,A]$, $b\ge 0$, $q\in(1,2]$, we have
\begin{equation}\label{eq:brasco24}
  J_q(a) - J_q(a-b)
  \ge \max\left\{
        J_q(A) - J_q(A-b),\;
        \frac{b^{q-1}}{2^{q-1}}
      \right\}.
\end{equation}
Set $a:=U_{0}(y)-U_{0}(x)>0$ and $b:=U_{0}(y)-U_{0}(2)\ge U_{0}(1)-U_{0}(2)>0$ for $y\in B_1$, and $A:=\|U_{0}\|_{L^\infty(\rnn)}$.
Applying \eqref{eq:brasco24}, we get
\begin{equation*}
  J_q(U_{0}(x)-U_{0}(2)) - J_q(U_{0}(x)-U_{0}(y))
  \ge \frac{(U_{0}(y)-U_{0}(2))^{q-1}}{2^{q-1}}
  \ge \frac{(U_{0}(1)-U_{0}(2))^{q-1}}{2^{q-1}}
  =: \! c_0 > 0,
\end{equation*}
Therefore, using the above estimates together with the fact that for $y\in B_1$ and $|x|\geq 3$, $|x-y|\le\frac{4}{3}|x|$, we obtain
\begin{align}
  \int_{B_1}
    \frac{J_q(U_{0}(x)-U_{0}(2))-J_q(U_{0}(x)-U_{0}(y))}{|x-y|^{N+sq}}dy
  &\ge
  \int_{B_1}
    \frac{c_0}{\left(\frac{4}{3}\right)^{N+sq}|x|^{N+sq}}dy
  \notag=\frac{c_0\,\omega_N}{\left(\frac{4}{3}\right)^{N+sq}|x|^{N+sq}}.
\end{align}
Substituting the above computation into
\eqref{eq:Lbarv_pointwise} yields
\eqref{Eq:LoweBoundOfTruncationOfAuxillaryProblemSolution}.
\end{proof}

\begin{proof}[Proof of \textup{(ii)} Theorem \ref{prop:decay-main}]
Let $\bar{U}$, $c_1$, and $C$ be as in Lemmas~\ref{lem:nonlocal_upper} and 
\ref{lem:nonlocal_barU}, and set \[\varepsilon := \min\{\left(\frac{c_1}{C}\right)^{\!1/(q-1)}, U_0(3)\} > 0.\]
By Lemma~\ref{lem:nonlocal_upper} and the scaling properties of $\mathcal{L}_{p,q}$, it follows that
\begin{equation*}
 \langle \! \langle  \mathcal{L}_{p,q}(\varepsilon\widetilde{\Gamma}), \varphi \rangle \! \rangle 
  = \varepsilon^{p-1}\underbrace{\langle -\Delta_p\widetilde{\Gamma},  \varphi \rangle_{p}  }_{=\,0}
    + \varepsilon^{q-1} \langle (-\Delta)_q^s\widetilde{\Gamma} , \varphi \rangle_{s,q}
  \le  \int_{\overline{B_3}^c} \frac{\varepsilon^{q-1}C}{|x|^{N+sq}} \varphi \, dx 
\end{equation*}
for all $\varphi \in \X(\overline{B_3}^c)$.
Thus, by the choice of $\varepsilon$ together with Lemma~\ref{lem:nonlocal_barU}, it follows that
$\mathcal{L}_{p,q}(\varepsilon\widetilde{\Gamma}) \le \mathcal{L}_{p,q}(\bar{U})$.
For the boundary condition on $B_3$, note that $\widetilde{\Gamma}(x) \le 1$ 
for $x \in \rnn$, and $\inf_{x \in \overline{B_3}} \bar{U}(x)= U_{0}(3).$ 
Hence, 
\begin{equation*}
  \varepsilon\widetilde{\Gamma}(x)
  \le \varepsilon
  \le U_{0}(3)
  \le \bar{U}(x)
  \qquad \forall\, x \in B_3.
\end{equation*}
An application of 
Theorem~\ref{thm:comparison_principle_mixed} then yields
\[
  \varepsilon\widetilde{\Gamma}(x) \le \bar{U}(x)
  \qquad \text{for a.e.\ } x \in \overline{B_3}^c,
\]
and consequently, for a.e.\ $x \in \overline{B_3}^c$,
\begin{equation}\label{eq:aux_lower_bound}
  U_{0}(x)
  = \bar{U}(x)
  \ge \varepsilon\widetilde{\Gamma}(x)
  = \varepsilon|x|^{-\alpha}
  = \varepsilon|x|^{-(N-p)/(p-1)}.
\end{equation}
It remains to compare $U_{0}$ and $U$ on $\overline{B_1}^c$. For any 
$\varphi \in \X(\overline{B_1}^c)$, one has
\[
  \langle \! \langle \mathcal{L}_{p,q}(U_{0}), \varphi \rangle \! \rangle
  = 0
  \le \langle \! \langle U^{p^*-1}, \varphi \rangle \! \rangle
   = \langle \! \langle \mathcal{L}_{p,q}(U), \varphi \rangle \! \rangle.
\]
Moreover, for each $x \in \overline{B}_1$,
$
  U_{0}(x) = \frac{c^*}{2} < c^* \le U(x).
$
Since $U_{0}, U \in \mathcal{D}^{1,p}(\rnn) \subset \mathcal{D}^{s,p,q}(\overline{B_1}^c)$, a further application of Theorem~\ref{thm:comparison_principle_mixed} 
therefore gives
\begin{equation}\label{eq:aux_upper_bound}
  U_{0}(x) \le U(x)
  \qquad \text{for a.e.\ } x \in \overline{B_1}^c.
\end{equation}
Combining \eqref{eq:aux_lower_bound} and \eqref{eq:aux_upper_bound} completes 
the proof.
\end{proof}

\section*{Acknowledgements}

The authors would like to express their sincere gratitude to
Prof. Jacques Giacomoni for several stimulating discussions
and valuable suggestions. D.~Gupta is supported by the Jang Young Sil Fellowship Program. S.~Malhotra is supported by the Prime Minister's Research Fellowship
(PMRF ID-1403266), Government of India.
\bibliographystyle{abbrv}
\bibliography{references}

@book {giovanni_fractional_book,
    AUTHOR = {Leoni, Giovanni},
     TITLE = {A first course in fractional {S}obolev spaces},
    SERIES = {Graduate Studies in Mathematics},
    VOLUME = {229},
 PUBLISHER = {American Mathematical Society, Providence, RI},
      YEAR = {2023},
     PAGES = {xv+586},
      ISBN = {[9781470468989]; [9781470472535]; [9781470472528]},
   MRCLASS = {46-01 (30H05 35R11 42Bxx 42C40 46E35)},
  MRNUMBER = {4567945},
MRREVIEWER = {E.\ S.\ Dubtsov},
       DOI = {10.1090/gsm/229},
       URL = {https://doi.org/10.1090/gsm/229},
}

@article {willem_radial_cpt_embeddingp>2,
    AUTHOR = {Su, Jiabao and Wang, Zhi-Qiang and Willem, Michel},
     TITLE = {Weighted {S}obolev embedding with unbounded and decaying
              radial potentials},
   JOURNAL = {J. Differential Equations},
  FJOURNAL = {Journal of Differential Equations},
    VOLUME = {238},
      YEAR = {2007},
    NUMBER = {1},
     PAGES = {201--219},
      ISSN = {0022-0396,1090-2732},
   MRCLASS = {35J60 (35J20)},
  MRNUMBER = {2334597},
       DOI = {10.1016/j.jde.2007.03.018},
       URL = {https://doi.org/10.1016/j.jde.2007.03.018},
}

@article {lions_strauss_lemma,
    AUTHOR = {Lions, Pierre-Louis},
     TITLE = {Sym\'etrie et compacit\'e{} dans les espaces de {S}obolev},
   JOURNAL = {J. Functional Analysis},
  FJOURNAL = {Journal of Functional Analysis},
    VOLUME = {49},
      YEAR = {1982},
    NUMBER = {3},
     PAGES = {315--334},
      ISSN = {0022-1236},
   MRCLASS = {46E35},
  MRNUMBER = {683027},
       DOI = {10.1016/0022-1236(82)90072-6},
       URL = {https://doi.org/10.1016/0022-1236(82)90072-6},
}

@article {Brasco_DecayEstimate,
    AUTHOR = {Brasco, Lorenzo and Mosconi, Sunra and Squassina, Marco},
     TITLE = {Optimal decay of extremals for the fractional {S}obolev
              inequality},
   JOURNAL = {Calc. Var. Partial Differential Equations},
  FJOURNAL = {Calculus of Variations and Partial Differential Equations},
    VOLUME = {55},
      YEAR = {2016},
    NUMBER = {2},
     PAGES = {Art. 23, 32},
      ISSN = {0944-2669,1432-0835},
   MRCLASS = {35A23 (35A15 35R11 46E35 49K21)},
  MRNUMBER = {3461371},
MRREVIEWER = {Evangelos\ A.\ Latos},
       DOI = {10.1007/s00526-016-0958-y},
       URL = {https://doi.org/10.1007/s00526-016-0958-y},
}

@incollection {simon_inequalities,
    AUTHOR = {Simon, J.},
     TITLE = {R\'egularit\'e{} de la solution d'une \'equation non
              lin\'eaire dans {${\mathbb{R}}\sp{N}$}},
 BOOKTITLE = {Journ\'ees d'{A}nalyse {N}on {L}in\'eaire ({P}roc. {C}onf.,
              {B}esan\c con, 1977)},
    SERIES = {Lecture Notes in Math.},
    VOLUME = {665},
     PAGES = {205--227},
 PUBLISHER = {Springer, Berlin},
      YEAR = {1978},
      ISBN = {3-540-08922-5},
   MRCLASS = {35D10 (35J60)},
  MRNUMBER = {519432},
MRREVIEWER = {Wolf\ von Wahl},
}

@article {lindgren_lindqvist_fractional_eigenvalues,
    AUTHOR = {Lindgren, Erik and Lindqvist, Peter},
     TITLE = {Fractional eigenvalues},
   JOURNAL = {Calc. Var. Partial Differential Equations},
  FJOURNAL = {Calculus of Variations and Partial Differential Equations},
    VOLUME = {49},
      YEAR = {2014},
    NUMBER = {1-2},
     PAGES = {795--826},
      ISSN = {0944-2669,1432-0835},
   MRCLASS = {35P30 (35J60 35R11)},
  MRNUMBER = {3148135},
MRREVIEWER = {Xavier\ Ros-Oton},
       DOI = {10.1007/s00526-013-0600-1},
       URL = {https://doi.org/10.1007/s00526-013-0600-1},
}

@article {garain_juha_transactions_regularity_mixed,
    AUTHOR = {Garain, Prashanta and Kinnunen, Juha},
     TITLE = {On the regularity theory for mixed local and nonlocal
              quasilinear elliptic equations},
   JOURNAL = {Trans. Amer. Math. Soc.},
  FJOURNAL = {Transactions of the American Mathematical Society},
    VOLUME = {375},
      YEAR = {2022},
    NUMBER = {8},
     PAGES = {5393--5423},
      ISSN = {0002-9947,1088-6850},
   MRCLASS = {35B45 (35B65 35D30 35J92 35R11)},
  MRNUMBER = {4469224},
       DOI = {10.1090/tran/8621},
       URL = {https://doi.org/10.1090/tran/8621},
}

@book {evans_pde_book,
    AUTHOR = {Evans, Lawrence C.},
     TITLE = {Partial differential equations},
    SERIES = {Graduate Studies in Mathematics},
    VOLUME = {19},
   EDITION = {Second},
 PUBLISHER = {American Mathematical Society, Providence, RI},
      YEAR = {2010},
     PAGES = {xxii+749},
      ISBN = {978-0-8218-4974-3},
   MRCLASS = {35-01},
  MRNUMBER = {2597943},
MRREVIEWER = {Diego\ M.\ Maldonado},
       DOI = {10.1090/gsm/019},
       URL = {https://doi.org/10.1090/gsm/019},
}

@article {CristianaGiuseppe,
    AUTHOR = {De Filippis, Cristiana and Mingione, Giuseppe},
     TITLE = {Gradient regularity in mixed local and nonlocal problems},
   JOURNAL = {Math. Ann.},
  FJOURNAL = {Mathematische Annalen},
    VOLUME = {388},
      YEAR = {2024},
    NUMBER = {1},
     PAGES = {261--328},
      ISSN = {0025-5831,1432-1807},
   MRCLASS = {49N60 (35J60 35R11 49J10)},
  MRNUMBER = {4693935},
MRREVIEWER = {Xiaodong\ Yan},
       DOI = {10.1007/s00208-022-02512-7},
       URL = {https://doi.org/10.1007/s00208-022-02512-7},
}

@article {Castro_NonLocalHarnackInequalities,
    AUTHOR = {Di Castro, Agnese and Kuusi, Tuomo and Palatucci, Giampiero},
     TITLE = {Nonlocal {H}arnack inequalities},
   JOURNAL = {J. Funct. Anal.},
  FJOURNAL = {Journal of Functional Analysis},
    VOLUME = {267},
      YEAR = {2014},
    NUMBER = {6},
     PAGES = {1807--1836},
      ISSN = {0022-1236,1096-0783},
   MRCLASS = {35R11 (35B45 35B65 35D30 35J25 35R09)},
  MRNUMBER = {3237774},
       DOI = {10.1016/j.jfa.2014.05.023},
       URL = {https://doi.org/10.1016/j.jfa.2014.05.023},
}

@article{brezis_nirenberg,
  author    = {Brezis, Ha\"im and Nirenberg, Louis},
  title     = {Positive solutions of nonlinear elliptic equations involving 
               critical {S}obolev exponents},
  journal   = {Comm. Pure Appl. Math.},
  volume    = {36},
  number    = {4},
  pages     = {437--477},
  year      = {1983}
}

@article{brezis_lieb,
  author    = {Brezis, Ha\"im and Lieb, Elliott},
  title     = {A relation between pointwise convergence of functions and 
               convergence of functionals},
  journal   = {Proc. Amer. Math. Soc.},
  volume    = {88},
  number    = {3},
  pages     = {486--490},
  year      = {1983}
}

@article{aubin,
  author    = {Aubin, Thierry},
  title     = {Probl\`emes isop\'erim\'etriques et espaces de {S}obolev},
  journal   = {J. Differential Geometry},
  volume    = {11},
  number    = {4},
  pages     = {573--598},
  year      = {1976}
}

@article{talenti,
  author    = {Talenti, Giorgio},
  title     = {Best constant in {S}obolev inequality},
  journal   = {Ann. Mat. Pura Appl. (4)},
  volume    = {110},
  pages     = {353--372},
  year      = {1976}
}

@article{cotsiolis_tavoularis,
  author    = {Cotsiolis, Athanase and Tavoularis, Nikolaos K.},
  title     = {Best constants for {S}obolev inequalities for higher order 
               fractional derivatives},
  journal   = {J. Math. Anal. Appl.},
  volume    = {295},
  number    = {1},
  pages     = {225--236},
  year      = {2004}
}

@article{biagi_dipierro_valdinoci_vecchi_2021,
  author    = {Biagi, Stefano and Dipierro, Serena and Valdinoci, Enrico 
               and Vecchi, Eugenio},
  title     = {Semilinear elliptic equations involving mixed local and nonlocal 
               operators},
  journal   = {Proc. Roy. Soc. Edinburgh Sect. A},
  volume    = {151},
  number    = {5},
  pages     = {1611--1641},
  year      = {2021}
}

@article{biagi_dipierro_valdinoci_vecchi_2022,
  author    = {Biagi, Stefano and Dipierro, Serena and Valdinoci, Enrico 
               and Vecchi, Eugenio},
  title     = {Mixed local and nonlocal elliptic operators: regularity and 
               maximum principles},
  journal   = {Comm. Partial Differential Equations},
  volume    = {47},
  number    = {3},
  pages     = {585--629},
  year      = {2022}
}

@book{willem,
  author    = {Willem, Michel},
  title     = {Minimax Theorems},
  series    = {Progress in Nonlinear Differential Equations and their 
               Applications},
  volume    = {24},
  publisher = {Birkh\"auser},
  address   = {Boston},
  year      = {1996}
}

@book {WillemFunctionalAnalysis,
    AUTHOR = {Willem, Michel},
     TITLE = {Functional analysis---fundamentals and applications},
    SERIES = {Cornerstones},
   EDITION = {Second},
 PUBLISHER = {Birkh\"auser/Springer, Cham},
      YEAR = {[2022] \copyright 2022},
     PAGES = {xv+251},
      ISBN = {978-3-031-09148-3; 978-3-031-09149-0},
   MRCLASS = {46-01 (35B05 35B50 35J25 35P05)},
  MRNUMBER = {4696491},
       DOI = {10.1007/978-3-031-09149-0},
       URL = {https://doi.org/10.1007/978-3-031-09149-0},
}

@article{Ponce,
	author = {Ponce, Augusto C. },
	journal = {Calculus of Variations and Partial Differential Equations},
	number = {3},
	pages = {229--255},
	title = {A new approach to Sobolev spaces and connections to {$\mathbf{\Gamma}$-convergence}},
	volume = {19},
	year = {2004}}

@inproceedings{JeanBrezisMironescu,
  title={Another look at Sobolev spaces},
  author={Jean Bourgain and Haim Brezis and Petru Mironescu},
  year={2001},
  url={https://api.semanticscholar.org/CorpusID:125723890}
}

@article{BrascoCharacterization,
	author = {Brasco, Lorenzo and G{\'o}mez-Castro, David and V{\'a}zquez, Juan Luis},
	journal = {Calculus of Variations and Partial Differential Equations},
	number = {2},
	pages = {60},
	title = {Characterisation of homogeneous fractional Sobolev spaces},
	volume = {60},
	year = {2021}}

@article {brasco_salort_spaces,
    AUTHOR = {Brasco, Lorenzo and Salort, Ariel},
     TITLE = {A note on homogeneous {S}obolev spaces of fractional order},
   JOURNAL = {Ann. Mat. Pura Appl. (4)},
  FJOURNAL = {Annali di Matematica Pura ed Applicata. Series IV},
    VOLUME = {198},
      YEAR = {2019},
    NUMBER = {4},
     PAGES = {1295--1330},
      ISSN = {0373-3114,1618-1891},
   MRCLASS = {46E35 (46B70)},
  MRNUMBER = {3987216},
MRREVIEWER = {Oscar\ Dom\'inguez},
       DOI = {10.1007/s10231-018-0817-x},
       URL = {https://doi.org/10.1007/s10231-018-0817-x},
}

@book {brezis_book,
    AUTHOR = {Brezis, Haim},
     TITLE = {Functional analysis, {S}obolev spaces and partial differential
              equations},
    SERIES = {Universitext},
 PUBLISHER = {Springer, New York},
      YEAR = {2011},
     PAGES = {xiv+599},
      ISBN = {978-0-387-70913-0},
   MRCLASS = {35-01 (46-01 46E35 46N20 47F05)},
  MRNUMBER = {2759829},
MRREVIEWER = {Vicen\c tiu\ D.\ R\u adulescu},
}

@article{malhotra2025eigenvalues,
  title={On the Eigenvalues and {F}u{\v{c}}{\'\i}k Spectrum of $p$-Laplace Local and Nonlocal Operator With Mixed Interpolated {H}ardy Term},
  author={Malhotra, Shammi and Goyal, Sarika and Sreenadh, Konijeti},
  journal={Asymptotic Analysis},
  pages={09217134251339280},
  year={2025},
  publisher={SAGE Publications Sage UK: London, England}
}

@article{silva2024mixed,
  title={Mixed local-nonlocal quasilinear problems with critical nonlinearities},
  author={da Silva, Jo{\~a}o Vitor and Fiscella, Alessio and Viloria, Victor A Blanco},
  journal={Journal of Differential Equations},
  volume={408},
  pages={494--536},
  year={2024},
  publisher={Elsevier}
}

@article{wang2023ground,
  title={The ground states of {H}{\'e}non equations for $p$-{L}aplacian in $\mathbb{R}^N$ involving upper weighted critical exponents},
  author={Wang, Cong and Su, Jiabao},
  journal={Communications in Nonlinear Science and Numerical Simulation},
  volume={120},
  pages={107146},
  year={2023},
  publisher={Elsevier}
}

@article {CastroTuomoPalatucci,
    AUTHOR = {Di Castro, Agnese and Kuusi, Tuomo and Palatucci, Giampiero},
     TITLE = {Local behavior of fractional {$p$}-minimizers},
   JOURNAL = {Ann. Inst. H. Poincar\'e{} C Anal. Non Lin\'eaire},
  FJOURNAL = {Annales de l'Institut Henri Poincar\'e{} C. Analyse Non
              Lin\'eaire},
    VOLUME = {33},
      YEAR = {2016},
    NUMBER = {5},
     PAGES = {1279--1299},
      ISSN = {0294-1449,1873-1430},
   MRCLASS = {35R11 (35B45 35B65 35D30 47G20)},
  MRNUMBER = {3542614},
MRREVIEWER = {Erwin\ Topp},
       DOI = {10.1016/j.anihpc.2015.04.003},
       URL = {https://doi.org/10.1016/j.anihpc.2015.04.003},
}

@book {giusti_book,
    AUTHOR = {Giusti, Enrico},
     TITLE = {Direct methods in the calculus of variations},
 PUBLISHER = {World Scientific Publishing Co., Inc., River Edge, NJ},
      YEAR = {2003},
     PAGES = {viii+403},
      ISBN = {981-238-043-4},
   MRCLASS = {49-02 (35J50 49J10 49K10 49N60)},
  MRNUMBER = {1962933},
MRREVIEWER = {Giovanni\ Alberti},
       DOI = {10.1142/9789812795557},
       URL = {https://doi.org/10.1142/9789812795557},
}

@book {giovanni_book_sobolev,
    AUTHOR = {Leoni, Giovanni},
     TITLE = {A first course in {S}obolev spaces},
    SERIES = {Graduate Studies in Mathematics},
    VOLUME = {181},
   EDITION = {Second},
 PUBLISHER = {American Mathematical Society, Providence, RI},
      YEAR = {2017},
     PAGES = {xxii+734},
      ISBN = {978-1-4704-2921-8},
   MRCLASS = {46E35 (26Axx 26B30 28A78 46-01)},
  MRNUMBER = {3726909},
       DOI = {10.1090/gsm/181},
       URL = {https://doi.org/10.1090/gsm/181},
}

@book {lindqvist_notes,
    AUTHOR = {Lindqvist, Peter},
     TITLE = {Notes on the {$p$}-{L}aplace equation},
    SERIES = {Report. University of Jyv\"askyl\"a{} Department of
              Mathematics and Statistics},
    VOLUME = {102},
 PUBLISHER = {University of Jyv\"askyl\"a, Jyv\"askyl\"a},
      YEAR = {2006},
     PAGES = {ii+80},
      ISBN = {951-39-2586-2},
   MRCLASS = {35J60},
  MRNUMBER = {2242021},
MRREVIEWER = {Jorge\ Garc\'ia-Meli\'an},
}

@article {PezzoQuaas,
    AUTHOR = {Del Pezzo, Leandro M. and Quaas, Alexander},
     TITLE = {Spectrum of the fractional {$p$}-{L}aplacian in {$\mathbb{R}^N$}
              and decay estimate for positive solutions of a {S}chr\"odinger
              equation},
   JOURNAL = {Nonlinear Anal.},
  FJOURNAL = {Nonlinear Analysis. Theory, Methods \& Applications. An
              International Multidisciplinary Journal},
    VOLUME = {193},
      YEAR = {2020},
     PAGES = {111479, 24},
      ISSN = {0362-546X,1873-5215},
   MRCLASS = {35R11 (35B40 35P05 46E35 47A10)},
  MRNUMBER = {4062974},
       DOI = {10.1016/j.na.2019.03.002},
       URL = {https://doi.org/10.1016/j.na.2019.03.002},
}

@article {HolderRegularityFractional,
    AUTHOR = {Iannizzotto, Antonio and Mosconi, Sunra and Squassina, Marco},
     TITLE = {Global {H}\"older regularity for the fractional
              {$p$}-{L}aplacian},
   JOURNAL = {Rev. Mat. Iberoam.},
  FJOURNAL = {Revista Matem\'atica Iberoamericana},
    VOLUME = {32},
      YEAR = {2016},
    NUMBER = {4},
     PAGES = {1353--1392},
      ISSN = {0213-2230,2235-0616},
   MRCLASS = {35R11 (35B65 35J92 47G20)},
  MRNUMBER = {3593528},
MRREVIEWER = {Pablo\ Ra\'ul\ Stinga},
       DOI = {10.4171/RMI/921},
       URL = {https://doi.org/10.4171/RMI/921},
}

@article {FractionalLaplacianEquivalenceOfSolutions,
    AUTHOR = {Korvenp\"a\"a, Janne and Kuusi, Tuomo and Lindgren, Erik},
     TITLE = {Equivalence of solutions to fractional {$p$}-{L}aplace type
              equations},
   JOURNAL = {J. Math. Pures Appl. (9)},
  FJOURNAL = {Journal de Math\'ematiques Pures et Appliqu\'ees. Neuvi\`eme
              S\'erie},
    VOLUME = {132},
      YEAR = {2019},
     PAGES = {1--26},
      ISSN = {0021-7824,1776-3371},
   MRCLASS = {35R11 (35B51 35D30 35D40 35R09)},
  MRNUMBER = {4030247},
MRREVIEWER = {Joana\ Terra},
       DOI = {10.1016/j.matpur.2017.10.004},
       URL = {https://doi.org/10.1016/j.matpur.2017.10.004},
}

@article{anthal2025pohozaev,
  title={Pohozaev-type identities for classes of quasilinear elliptic local and nonlocal equations and systems, with applications},
  author={Anthal, Gurdev Chand and Garain, Prashanta},
  journal={arXiv preprint arXiv:2506.08667},
  year={2025}
}

@article {cozzi_harnack_fractional,
    AUTHOR = {Cozzi, Matteo},
     TITLE = {Regularity results and {H}arnack inequalities for minimizers
              and solutions of nonlocal problems: a unified approach via
              fractional {D}e {G}iorgi classes},
   JOURNAL = {J. Funct. Anal.},
  FJOURNAL = {Journal of Functional Analysis},
    VOLUME = {272},
      YEAR = {2017},
    NUMBER = {11},
     PAGES = {4762--4837},
      ISSN = {0022-1236,1096-0783},
   MRCLASS = {49N60 (35B45 35B65 35R11 47G20)},
  MRNUMBER = {3630640},
MRREVIEWER = {Mark\ Allen},
       DOI = {10.1016/j.jfa.2017.02.016},
       URL = {https://doi.org/10.1016/j.jfa.2017.02.016},
}

@article {FractionalHaryInequality,
    AUTHOR = {Frank, Rupert L. and Seiringer, Robert},
     TITLE = {Non-linear ground state representations and sharp {H}ardy
              inequalities},
   JOURNAL = {J. Funct. Anal.},
  FJOURNAL = {Journal of Functional Analysis},
    VOLUME = {255},
      YEAR = {2008},
    NUMBER = {12},
     PAGES = {3407--3430},
      ISSN = {0022-1236,1096-0783},
   MRCLASS = {46E35 (35A25)},
  MRNUMBER = {2469027},
MRREVIEWER = {Niko\ M.\ Marola},
       DOI = {10.1016/j.jfa.2008.05.015},
       URL = {https://doi.org/10.1016/j.jfa.2008.05.015},
}

@book {henrot_extremum_problems,
    AUTHOR = {Henrot, A.},
     TITLE = {Extremum problems for eigenvalues of elliptic operators},
    SERIES = {Frontiers in Mathematics},
 PUBLISHER = {Birkh\"auser Verlag, Basel},
      YEAR = {2006},
     PAGES = {x+202},
      ISBN = {978-3-7643-7705-2; 3-7643-7705-4},
   MRCLASS = {35P05 (35J05 35J15 35P15 47F05 49R50)},
  MRNUMBER = {2251558},
MRREVIEWER = {G\"unter\ Berger},
}

@article {mosconi_LS_ineq,
    AUTHOR = {Gigli, Nicola and Mosconi, Sunra},
     TITLE = {The abstract {L}ewy-{S}tampacchia inequality and applications},
   JOURNAL = {J. Math. Pures Appl. (9)},
  FJOURNAL = {Journal de Math\'ematiques Pures et Appliqu\'ees. Neuvi\`eme
              S\'erie},
    VOLUME = {104},
      YEAR = {2015},
    NUMBER = {2},
     PAGES = {258--275},
      ISSN = {0021-7824,1776-3371},
   MRCLASS = {49J40 (35J86 35J87)},
  MRNUMBER = {3365829},
MRREVIEWER = {Dimitri\ Mugnai},
       DOI = {10.1016/j.matpur.2015.02.007},
       URL = {https://doi.org/10.1016/j.matpur.2015.02.007},
}

@article{LocalSobolevOnBall,
author = {Adams, David R.},
title = {FINE REGULARITY OF SOLUTIONS OF ELLIPTIC PARTIAL DIFFERENTIAL EQUATIONS (Mathematical Surveys and Monographs 51)},
journal = {Bulletin of the London Mathematical Society},
volume = {31},
number = {2},
pages = {248-250},
doi = {https://doi.org/10.1112/S0024609398215025},
url = {https://londmathsoc.onlinelibrary.wiley.com/doi/abs/10.1112/S0024609398215025},
eprint = {https://londmathsoc.onlinelibrary.wiley.com/doi/pdf/10.1112/S0024609398215025},
year = {1999}
}

@article {brasco_second_eigenvalue,
    AUTHOR = {Brasco, Lorenzo and Parini, Enea},
     TITLE = {The second eigenvalue of the fractional {$p$}-{L}aplacian},
   JOURNAL = {Adv. Calc. Var.},
  FJOURNAL = {Advances in Calculus of Variations},
    VOLUME = {9},
      YEAR = {2016},
    NUMBER = {4},
     PAGES = {323--355},
      ISSN = {1864-8258,1864-8266},
   MRCLASS = {35R11 (35P30 35R09 47J10)},
  MRNUMBER = {3552458},
MRREVIEWER = {Eduardo\ Colorado},
       DOI = {10.1515/acv-2015-0007},
       URL = {https://doi.org/10.1515/acv-2015-0007},
}

@book {grisvard_book,
    AUTHOR = {Grisvard, Pierre},
     TITLE = {Elliptic problems in nonsmooth domains},
    SERIES = {Classics in Applied Mathematics},
    VOLUME = {69},
      NOTE = {Reprint of the 1985 original [MR0775683],
              With a foreword by Susanne C. Brenner},
 PUBLISHER = {Society for Industrial and Applied Mathematics (SIAM),
              Philadelphia, PA},
      YEAR = {2011},
     PAGES = {xx+410},
      ISBN = {978-1-611972-02-3},
   MRCLASS = {35J25 (01A75 35-02)},
  MRNUMBER = {3396210},
       DOI = {10.1137/1.9781611972030.ch1},
       URL = {https://doi.org/10.1137/1.9781611972030.ch1},
}

@article{chakraborty2025global,
  title={Global Compactness Result for a {B}r\'ezis-{N}irenberg-Type Problem Involving Mixed Local Nonlocal Operator},
  author={Chakraborty, Souptik and Gupta, Diksha and Malhotra, Shammi and Sreenadh, Konijeti},
  journal={arXiv preprint arXiv:2504.15968},
  year={2025}
}

@article {mosconi_BN_pfrac,
    AUTHOR = {Mosconi, Sunra and Perera, Kanishka and Squassina, Marco and
              Yang, Yang},
     TITLE = {The {B}rezis-{N}irenberg problem for the fractional
              {$p$}-{L}aplacian},
   JOURNAL = {Calc. Var. Partial Differential Equations},
  FJOURNAL = {Calculus of Variations and Partial Differential Equations},
    VOLUME = {55},
      YEAR = {2016},
    NUMBER = {4},
     PAGES = {Art. 105, 25},
      ISSN = {0944-2669,1432-0835},
   MRCLASS = {35R11 (35A15 35B33 35J92)},
  MRNUMBER = {3530213},
MRREVIEWER = {Kai\ Diethelm},
       DOI = {10.1007/s00526-016-1035-2},
       URL = {https://doi.org/10.1007/s00526-016-1035-2},
}

@article {peetre_lorentz_spaces,
    AUTHOR = {Peetre, Jaak},
     TITLE = {Espaces d'interpolation et th\'eor\`eme de {S}oboleff},
   JOURNAL = {Ann. Inst. Fourier (Grenoble)},
  FJOURNAL = {Universit\'e{} de Grenoble. Annales de l'Institut Fourier},
    VOLUME = {16},
      YEAR = {1966},
     PAGES = {279--317},
      ISSN = {0373-0956,1777-5310},
   MRCLASS = {46.38 (42.00)},
  MRNUMBER = {221282},
MRREVIEWER = {C.\ Goulaouic},
       URL = {http://www.numdam.org/item?id=AIF_1966__16_1_279_0},
}

@article {cassani_Sobolevineq_lorentz,
    AUTHOR = {Cassani, Daniele and Ruf, Bernhard and Tarsi, Cristina},
     TITLE = {Optimal {S}obolev type inequalities in {L}orentz spaces},
   JOURNAL = {Potential Anal.},
  FJOURNAL = {Potential Analysis. An International Journal Devoted to the
              Interactions between Potential Theory, Probability Theory,
              Geometry and Functional Analysis},
    VOLUME = {39},
      YEAR = {2013},
    NUMBER = {3},
     PAGES = {265--285},
      ISSN = {0926-2601,1572-929X},
   MRCLASS = {46E35 (35A23 35B38 35B65 35J20)},
  MRNUMBER = {3102987},
MRREVIEWER = {Dachun\ Yang},
       DOI = {10.1007/s11118-012-9329-2},
       URL = {https://doi.org/10.1007/s11118-012-9329-2},
}

@article {cassani_Hardy_lorentz,
    AUTHOR = {Cassani, D. and Ruf, B. and Tarsi, C.},
     TITLE = {Equivalent and attained version of {H}ardy's inequality in
              {$\mathbb{R}^n$}},
   JOURNAL = {J. Funct. Anal.},
  FJOURNAL = {Journal of Functional Analysis},
    VOLUME = {275},
      YEAR = {2018},
    NUMBER = {12},
     PAGES = {3303--3324},
      ISSN = {0022-1236,1096-0783},
   MRCLASS = {42B30 (46E40)},
  MRNUMBER = {3864503},
MRREVIEWER = {Oscar\ Blasco},
       DOI = {10.1016/j.jfa.2018.09.008},
       URL = {https://doi.org/10.1016/j.jfa.2018.09.008},
}

@article {alvino_best_Lorentz_Sobolev,
    AUTHOR = {Alvino, Angelo},
     TITLE = {Sulla diseguaglianza di {S}obolev in spazi di {L}orentz},
   JOURNAL = {Boll. Un. Mat. Ital. A (5)},
  FJOURNAL = {Bollettino della Unione Matematica Italiana. A. Serie 5},
    VOLUME = {14},
      YEAR = {1977},
    NUMBER = {1},
     PAGES = {148--156},
      ISSN = {0392-4033},
   MRCLASS = {46E35},
  MRNUMBER = {438106},
MRREVIEWER = {Pavel\ Doktor},
}

@article {glowinski_plap_glaciology,
    AUTHOR = {Glowinski, Roland and Rappaz, Jacques},
     TITLE = {Approximation of a nonlinear elliptic problem arising in a
              non-{N}ewtonian fluid flow model in glaciology},
   JOURNAL = {M2AN Math. Model. Numer. Anal.},
  FJOURNAL = {M2AN. Mathematical Modelling and Numerical Analysis},
    VOLUME = {37},
      YEAR = {2003},
    NUMBER = {1},
     PAGES = {175--186},
      ISSN = {0764-583X,1290-3841},
   MRCLASS = {86A40 (65N15 65N30 76A05 76M10)},
  MRNUMBER = {1972657},
MRREVIEWER = {Eun-Jae\ Park},
       DOI = {10.1051/m2an:2003012},
       URL = {https://doi.org/10.1051/m2an:2003012},
}

@inproceedings{kuijper_plap_image_processing,
  title={$p$-{L}aplacian driven image processing},
  author={Kuijper, Arjan},
  booktitle={2007 IEEE International conference on image processing},
  volume={5},
  pages={V--257},
  year={2007},
  organization={IEEE}
}

@article {cuccu_plap_nonlinear_elasticity,
    AUTHOR = {Cuccu, Fabrizio and Emamizadeh, Behrouz and Porru, Giovanni},
     TITLE = {Nonlinear elastic membranes involving the {$p$}-{L}aplacian
              operator},
   JOURNAL = {Electron. J. Differential Equations},
  FJOURNAL = {Electronic Journal of Differential Equations},
      YEAR = {2006},
     PAGES = {No. 49, 10},
      ISSN = {1072-6691},
   MRCLASS = {35J60 (35J25 49Q10 74K15)},
  MRNUMBER = {2226922},
}

@article {pellacci_logistic_equation,
    AUTHOR = {Montefusco, E. and Pellacci, B. and Verzini,
              G.},
     TITLE = {Fractional diffusion with {N}eumann boundary conditions: the
              logistic equation},
   JOURNAL = {Discrete Contin. Dyn. Syst. Ser. B},
  FJOURNAL = {Discrete and Continuous Dynamical Systems. Series B. A Journal
              Bridging Mathematics and Sciences},
    VOLUME = {18},
      YEAR = {2013},
    NUMBER = {8},
     PAGES = {2175--2202},
      ISSN = {1531-3492,1553-524X},
   MRCLASS = {35R11 (35B32 35J25 92D40)},
  MRNUMBER = {3082317},
MRREVIEWER = {Xavier\ Ros-Oton},
       DOI = {10.3934/dcdsb.2013.18.2175},
       URL = {https://doi.org/10.3934/dcdsb.2013.18.2175},
}

@article {valdinoci_logistic_equation,
    AUTHOR = {Dipierro, S. and P. Lippi, E. and Valdinoci,
              E.},
     TITLE = {({N}on)local logistic equations with {N}eumann conditions},
   JOURNAL = {Ann. Inst. H. Poincar\'e{} C Anal. Non Lin\'eaire},
  FJOURNAL = {Annales de l'Institut Henri Poincar\'e{} C. Analyse Non
              Lin\'eaire},
    VOLUME = {40},
      YEAR = {2023},
    NUMBER = {5},
     PAGES = {1093--1166},
      ISSN = {0294-1449,1873-1430},
   MRCLASS = {35Q92 (35R11 60G22 92D25)},
  MRNUMBER = {4651677},
       DOI = {10.4171/aihpc/57},
       URL = {https://doi.org/10.4171/aihpc/57},
}

@article {pellacci_best_dispersal_strategies,
    AUTHOR = {Pellacci, B. and Verzini, G.},
     TITLE = {Best dispersal strategies in spatially heterogeneous
              environments: optimization of the principal eigenvalue for
              indefinite fractional {N}eumann problems},
   JOURNAL = {J. Math. Biol.},
  FJOURNAL = {Journal of Mathematical Biology},
    VOLUME = {76},
      YEAR = {2018},
    NUMBER = {6},
     PAGES = {1357--1386},
      ISSN = {0303-6812,1432-1416},
   MRCLASS = {35R11 (35P15 47A75 92D25)},
  MRNUMBER = {3771424},
       DOI = {10.1007/s00285-017-1180-z},
       URL = {https://doi.org/10.1007/s00285-017-1180-z},
}

@article {valdinoci_ecological_niche,
    AUTHOR = {Dipierro, S. and Valdinoci, E.},
     TITLE = {Description of an ecological niche for a mixed local/nonlocal
              dispersal: an evolution equation and a new {N}eumann condition
              arising from the superposition of {B}rownian and {L}\'evy
              processes},
   JOURNAL = {Phys. A},
  FJOURNAL = {Physica A. Statistical Mechanics and its Applications},
    VOLUME = {575},
      YEAR = {2021},
     PAGES = {Paper No. 126052, 20},
      ISSN = {0378-4371,1873-2119},
   MRCLASS = {60G50 (35Q92 92B05)},
  MRNUMBER = {4249816},
       DOI = {10.1016/j.physa.2021.126052},
       URL = {https://doi.org/10.1016/j.physa.2021.126052},
}

@article{blazevski_anisotropic_heat_transport,
  title={Local and nonlocal anisotropic transport in reversed shear magnetic fields: Shearless Cantori and nondiffusive transport},
  author={Blazevski, D. and del-Castillo-Negrete, D.},
  journal={Physical Review E—Statistical, Nonlinear, and Soft Matter Physics},
  volume={87},
  number={6},
  pages={063106},
  year={2013},
  publisher={APS}
}

@article {byun_regularity_mixed,
    AUTHOR = {Byun, Sun-Sig and Lee, Ho-Sik and Song, Kyeong},
     TITLE = {Regularity results for mixed local and nonlocal double phase
              functionals},
   JOURNAL = {J. Differential Equations},
  FJOURNAL = {Journal of Differential Equations},
    VOLUME = {416},
      YEAR = {2025},
     PAGES = {1528--1563},
      ISSN = {0022-0396,1090-2732},
   MRCLASS = {49N60 (35B65 35R05 35R11 47G20)},
  MRNUMBER = {4819006},
MRREVIEWER = {Eugen\ Viszus},
       DOI = {10.1016/j.jde.2024.10.028},
       URL = {https://doi.org/10.1016/j.jde.2024.10.028},
}

@article {biagi_Hong_krahn_ineq_mixed,
    AUTHOR = {Biagi, Stefano and Dipierro, Serena and Valdinoci, Enrico and
              Vecchi, Eugenio},
     TITLE = {A {H}ong-{K}rahn-{S}zeg\"o{} inequality for mixed local and
              nonlocal operators},
   JOURNAL = {Math. Eng.},
  FJOURNAL = {Mathematics in Engineering},
    VOLUME = {5},
      YEAR = {2023},
    NUMBER = {1},
     PAGES = {Paper No. 014, 25},
      ISSN = {2640-3501},
   MRCLASS = {35P30},
  MRNUMBER = {4391102},
MRREVIEWER = {Tetsutaro\ Shibata},
       DOI = {10.3934/mine.2023014},
       URL = {https://doi.org/10.3934/mine.2023014},
}

@article {biagi_faber_krahn_ineq_mixed,
    AUTHOR = {Biagi, Stefano and Dipierro, Serena and Valdinoci, Enrico and
              Vecchi, Eugenio},
     TITLE = {A {F}aber-{K}rahn inequality for mixed local and nonlocal
              operators},
   JOURNAL = {J. Anal. Math.},
  FJOURNAL = {Journal d'Analyse Math\'ematique},
    VOLUME = {150},
      YEAR = {2023},
    NUMBER = {2},
     PAGES = {405--448},
      ISSN = {0021-7670,1565-8538},
   MRCLASS = {35P15 (31B10)},
  MRNUMBER = {4645045},
MRREVIEWER = {Richard\ S.\ Laugesen},
       DOI = {10.1007/s11854-023-0272-5},
       URL = {https://doi.org/10.1007/s11854-023-0272-5},
}

@article {pezzo_ferreira_rossi_eigenvalues_mixed,
    AUTHOR = {Del Pezzo, Leandro M. and Ferreira, Ra\'ul and Rossi, Julio
              D.},
     TITLE = {Eigenvalues for a combination between local and nonlocal
              {$p$}-{L}aplacians},
   JOURNAL = {Fract. Calc. Appl. Anal.},
  FJOURNAL = {Fractional Calculus and Applied Analysis. An International
              Journal for Theory and Applications},
    VOLUME = {22},
      YEAR = {2019},
    NUMBER = {5},
     PAGES = {1414--1436},
      ISSN = {1311-0454,1314-2224},
   MRCLASS = {35R11 (26A33 35J92 35P30 47G20)},
  MRNUMBER = {4044581},
       DOI = {10.1515/fca-2019-0074},
       URL = {https://doi.org/10.1515/fca-2019-0074},
}

@article {divya_eigenvalue_mixed,
    AUTHOR = {Goel, Divya and Sreenadh, K.},
     TITLE = {On the second eigenvalue of combination between local and
              nonlocal {$p$}-{L}aplacian},
   JOURNAL = {Proc. Amer. Math. Soc.},
  FJOURNAL = {Proceedings of the American Mathematical Society},
    VOLUME = {147},
      YEAR = {2019},
    NUMBER = {10},
     PAGES = {4315--4327},
      ISSN = {0002-9939,1088-6826},
   MRCLASS = {35P15 (35P30 35R11 47J10 49Q10)},
  MRNUMBER = {4002544},
       DOI = {10.1090/proc/14542},
       URL = {https://doi.org/10.1090/proc/14542},
}

@article{bhakta_mixed_multiplicity,
  title={Quasilinear problems with mixed local-nonlocal operator and concave-critical nonlinearities: Multiplicity of positive solutions},
  author={Bhakta, Mousomi and Biswas, Nirjan and Das, Paramananda},
  journal={arXiv preprint arXiv:2504.15000},
  year={2025}
}

@article{dhanya_mixed_multiplicity_weights,
  title={Multiplicity Results for Mixed Local Nonlocal Equations With Indefinite Concave-Convex Type Nonlinearity},
  author={Dhanya, R and Giacomoni, Jacques and Jana, Ritabrata},
  journal={arXiv preprint arXiv:2503.00365},
  year={2025}
}

@article {biagi_vecchi_abc_mixed,
    AUTHOR = {Biagi, Stefano and Vecchi, Eugenio},
     TITLE = {On the existence of a second positive solution to mixed
              local-nonlocal concave-convex critical problems},
   JOURNAL = {Nonlinear Anal.},
  FJOURNAL = {Nonlinear Analysis. Theory, Methods \& Applications. An
              International Multidisciplinary Journal},
    VOLUME = {256},
      YEAR = {2025},
     PAGES = {Paper No. 113795, 27},
      ISSN = {0362-546X,1873-5215},
   MRCLASS = {35J75 (35B09 35B33 35D30 35J25)},
  MRNUMBER = {4878869},
MRREVIEWER = {Abdelkrim\ Moussaoui},
       DOI = {10.1016/j.na.2025.113795},
       URL = {https://doi.org/10.1016/j.na.2025.113795},
}

@article {chen_kim_song_uniform_harnack_mixed,
    AUTHOR = {Chen, Zhen-Qing and Kim, Panki and Song, Renming and Vondra\v{c}ek, Zoran},
     TITLE = {Boundary {H}arnack principle for {$\Delta+\Delta^{\alpha/2}$}},
   JOURNAL = {Trans. Amer. Math. Soc.},
  FJOURNAL = {Transactions of the American Mathematical Society},
    VOLUME = {364},
      YEAR = {2012},
    NUMBER = {8},
     PAGES = {4169--4205},
      ISSN = {0002-9947,1088-6850},
   MRCLASS = {31B25},
  MRNUMBER = {2912450},
MRREVIEWER = {Erkan\ Nane},
       DOI = {10.1090/S0002-9947-2012-05542-5},
       URL = {https://doi.org/10.1090/S0002-9947-2012-05542-5},
}

@article {garain_lindgren_holder_pmixed,
    AUTHOR = {Garain, Prashanta and Lindgren, Erik},
     TITLE = {Higher {H}\"older regularity for mixed local and nonlocal
              degenerate elliptic equations},
   JOURNAL = {Calc. Var. Partial Differential Equations},
  FJOURNAL = {Calculus of Variations and Partial Differential Equations},
    VOLUME = {62},
      YEAR = {2023},
    NUMBER = {2},
     PAGES = {Paper No. 67, 36},
      ISSN = {0944-2669,1432-0835},
   MRCLASS = {35B65 (26A33 35D30 35J70 35R09 35R11)},
  MRNUMBER = {4530314},
       DOI = {10.1007/s00526-022-02401-6},
       URL = {https://doi.org/10.1007/s00526-022-02401-6},
}

@article {xifeng_valdinoci_mixed_c2alpha_regularity,
    AUTHOR = {Su, Xifeng and Valdinoci, Enrico and Wei, Yuanhong and Zhang,
              Jiwen},
     TITLE = {On some regularity properties of mixed local and nonlocal
              elliptic equations},
   JOURNAL = {J. Differential Equations},
  FJOURNAL = {Journal of Differential Equations},
    VOLUME = {416},
      YEAR = {2025},
     PAGES = {576--613},
      ISSN = {0022-0396,1090-2732},
   MRCLASS = {35B65 (35J67 35R11)},
  MRNUMBER = {4808805},
MRREVIEWER = {Xuan\ Xuan\ Xi},
       DOI = {10.1016/j.jde.2024.10.003},
       URL = {https://doi.org/10.1016/j.jde.2024.10.003},
}

@article {shang_zhang_smp_mixed_ev,
    AUTHOR = {Shang, Bin and Zhang, Chao},
     TITLE = {A strong maximum principle for mixed local and nonlocal
              {$p$}-{L}aplace equations},
   JOURNAL = {Asymptot. Anal.},
  FJOURNAL = {Asymptotic Analysis},
    VOLUME = {133},
      YEAR = {2023},
    NUMBER = {1-2},
     PAGES = {1--12},
      ISSN = {0921-7134,1875-8576},
   MRCLASS = {35J92},
  MRNUMBER = {4595187},
       DOI = {10.3233/asy-221803},
       URL = {https://doi.org/10.3233/asy-221803},
}

@article {antonini_cozzi_gradient_regularity_mixed,
    AUTHOR = {Antonini, Carlo Alberto and Cozzi, Matteo},
     TITLE = {Global gradient regularity and a {H}opf lemma for quasilinear
              operators of mixed local-nonlocal type},
   JOURNAL = {J. Differential Equations},
  FJOURNAL = {Journal of Differential Equations},
    VOLUME = {425},
      YEAR = {2025},
     PAGES = {342--382},
      ISSN = {0022-0396,1090-2732},
   MRCLASS = {35B65 (35D30 35J60 35J92 35R11)},
  MRNUMBER = {4853424},
MRREVIEWER = {Makson\ S.\ Santos},
       DOI = {10.1016/j.jde.2025.01.030},
       URL = {https://doi.org/10.1016/j.jde.2025.01.030},
}

@article {xifeng_valdinoci_mixed_regularity,
    AUTHOR = {Su, Xifeng and Valdinoci, Enrico and Wei, Yuanhong and Zhang,
              Jiwen},
     TITLE = {Regularity results for solutions of mixed local and nonlocal
              elliptic equations},
   JOURNAL = {Math. Z.},
  FJOURNAL = {Mathematische Zeitschrift},
    VOLUME = {302},
      YEAR = {2022},
    NUMBER = {3},
     PAGES = {1855--1878},
      ISSN = {0025-5874,1432-1823},
   MRCLASS = {35R11 (35B65 35J67)},
  MRNUMBER = {4492518},
       DOI = {10.1007/s00209-022-03132-2},
       URL = {https://doi.org/10.1007/s00209-022-03132-2},
}

@book {serra_book,
    AUTHOR = {Badiale, Marino and Serra, Enrico},
     TITLE = {Semilinear elliptic equations for beginners},
    SERIES = {Universitext},
      NOTE = {Existence results via the variational approach},
 PUBLISHER = {Springer, London},
      YEAR = {2011},
     PAGES = {x+199},
      ISBN = {978-0-85729-226-1},
   MRCLASS = {35-01 (35B38 35J20 35J61 35J91 47J30)},
  MRNUMBER = {2722059},
MRREVIEWER = {Stephen\ B.\ Robinson},
       DOI = {10.1007/978-0-85729-227-8},
       URL = {https://doi.org/10.1007/978-0-85729-227-8},
}

@article {palais_principle_symmetric,
    AUTHOR = {Palais, Richard S.},
     TITLE = {The principle of symmetric criticality},
   JOURNAL = {Comm. Math. Phys.},
  FJOURNAL = {Communications in Mathematical Physics},
    VOLUME = {69},
      YEAR = {1979},
    NUMBER = {1},
     PAGES = {19--30},
      ISSN = {0010-3616,1432-0916},
   MRCLASS = {58E05},
  MRNUMBER = {547524},
MRREVIEWER = {A.\ J.\ Tromba},
       URL = {http://projecteuclid.org/euclid.cmp/1103905401},
}

@article {chen_li_symmetry_pfrac,
    AUTHOR = {Chen, Wenxiong and Li, Congming},
     TITLE = {Maximum principles for the fractional {$p$}-{L}aplacian and
              symmetry of solutions},
   JOURNAL = {Adv. Math.},
  FJOURNAL = {Advances in Mathematics},
    VOLUME = {335},
      YEAR = {2018},
     PAGES = {735--758},
      ISSN = {0001-8708,1090-2082},
   MRCLASS = {35R11 (35B50)},
  MRNUMBER = {3836677},
MRREVIEWER = {Vincenzo\ Ambrosio},
       DOI = {10.1016/j.aim.2018.07.016},
       URL = {https://doi.org/10.1016/j.aim.2018.07.016},
}
\end{document}